\documentclass[11pt]{amsart}
\usepackage{amscd,amsmath,amssymb,amsfonts,verbatim}
\usepackage[cmtip, all]{xy}
\usepackage{MnSymbol}
\usepackage{comment}

\newtheorem{thm}{Theorem}[section]
\newtheorem{prop}[thm]{Proposition}
\newtheorem{lem}[thm]{Lemma}
\newtheorem{cor}[thm]{Corollary}
\newtheorem{conj}[thm]{Conjecture}
\theoremstyle{definition}

\newtheorem{defn}[thm]{Definition}
\theoremstyle{remark}
\newtheorem{remk}[thm]{Remark}
\newtheorem{remks}[thm]{Remarks}

\newtheorem{exm}[thm]{Example}
\newtheorem{exms}[thm]{Examples}
\newtheorem{notat}[thm]{Notation}
\numberwithin{equation}{section}

{\hfill$\square$\end{defn}}
{\hfill$\square$\end{remk}}
{\hfill$\square$\end{remks}}
{\hfill$\square$\end{exm}}
{\hfill$\square$\end{exms}}
{\hfill$\square$\end{notat}}

\newcommand{\thmref}{Theorem~\ref}
\newcommand{\propref}{Proposition~\ref}
\newcommand{\corref}{Corollary~\ref}

\newcommand{\lemref}{Lemma~\ref}

\newcommand{\sC}{{\mathcal C}}

\newcommand{\sF}{{\mathcal F}}

\newcommand{\sI}{{\mathcal I}}

\newcommand{\sK}{{\mathcal K}}
\newcommand{\sL}{{\mathcal L}}
\newcommand{\sM}{{\mathcal M}}

\newcommand{\sO}{{\mathcal O}}
\newcommand{\sP}{{\mathcal P}}
\newcommand{\sQ}{{\mathcal Q}}
\newcommand{\sR}{{\mathcal R}}

\newcommand{\sV}{{\mathcal V}}

\newcommand{\sZ}{{\mathcal Z}}

\newcommand{\C}{{\mathbb C}}

\newcommand{\N}{{\mathbb N}}
\renewcommand{\P}{{\mathbb P}}
\newcommand{\Q}{{\mathbb Q}}

\newcommand{\Z}{{\mathbb Z}}

\newcommand{\fm}{{\mathfrak m}}

\newcommand{\ff}{{\mathfrak f}}

\newcommand{\Irr}{{\rm Irr}}
\newcommand{\Ker}{{\rm Ker}}
\newcommand{\gr}{{\rm gr}}

\newcommand{\CH}{{\rm CH}}

\newcommand{\surj}{\twoheadrightarrow}
\newcommand{\inj}{\hookrightarrow}
\newcommand{\red}{{\rm red}}

\newcommand{\Div}{{\rm Div}}
\newcommand{\Hom}{{\rm Hom}}

\newcommand{\Spec}{{\rm Spec \,}}
\newcommand{\sing}{{\rm sing}}

\newcommand{\ms}{{\rm ms}}

\newcommand{\ab}{\rm ab}

\newcommand{\Gal}{{\rm Gal}}
\newcommand{\divf}{{\rm div}}

\newcommand{\Sch}{{\operatorname{\mathbf{Sch}}}}

\newcommand{\Ab}{{\mathbf{Ab}}}

\newcommand{\cyc}{{\operatorname{\rm cyc}}}

\newcommand{\ds}{{/\kern-3pt/}}

\newcommand{\un}{\underline}
\newcommand{\ov}{\overline}

\renewcommand{\dim}{\text{\rm dim}}

\newcommand{\tuborg}{\left\{\begin{array}{ll}}
\newcommand{\sluttuborg}{\end{array}\right.}

\newcommand{\zar}{{\rm zar}}
\newcommand{\nis}{{\rm nis}}

\newcommand{\reg}{{\rm reg}}

\newcommand{\adiv}{{\rm adiv}}

\newcommand{\wt}{\widetilde}
\newcommand{\wh}{\widehat}
\newcommand{\cont}{{\rm cont}}

\newcommand{\coker}{{\rm Coker}}
\newcommand{\Fil}{{\rm fil}}

\def\cO{\mathcal{O}}

\newcounter{elno}   
\newenvironment{romanlist}{
                         \begin{list}{\roman{elno})
                                     }{\usecounter{elno}}
                      }{
                         \end{list}}
                         
\newcounter{elno-abc}   
\newenvironment{listabc}{
                         \begin{list}{\alph{elno-abc})
                                     }{\usecounter{elno-abc}}
                      }{
                         \end{list}}
\newcounter{elno-abc-prime}   
\newenvironment{listabcprime}{
                         \begin{list}{\alph{elno-abc-prime}')
                                     }{\usecounter{elno-abc-prime}}
                      }{
                         \end{list}}

\begin{document}
\title{Idele class groups with modulus}
\author{Rahul Gupta, Amalendu Krishna}
\address{Fakult\"at f\"ur Mathematik, Universit\"at Regensburg, 
93040, Regensburg, Germany.}
\email{Rahul.Gupta@mathematik.uni-regensburg.de}
\address{Department of Mathematics, Indian Institute of Science,  
Bangalore, 560012, India.}
\email{amalenduk@iisc.ac.in}



\maketitle

\begin{quote}\emph{Abstract.}  
We prove Bloch's formula for the Chow group of 0-cycles with modulus
on smooth projective varieties over finite fields.
The proof relies on two new results in global ramification theory.
\end{quote}

\setcounter{tocdepth}{1}

\tableofcontents

\section{Introduction}\label{sec:Intro}
\subsection{Motivation}\label{sec:Motn}
In the classical theory of algebraic cycles, Bloch's formula describes the Chow groups of a smooth
variety over a field in terms of certain Zariski or Nisnevich cohomology groups of the $K$-theory
sheaves. This provides a direct bridge between algebraic cycles and algebraic $K$-theory.
For codimension one cycles, such a formula was known to ancients and for codimension two
cycles (on surfaces), this was shown by Bloch \cite{Bloch74}. The most general case was shown
by Quillen \cite{Quillen}. A further progress was made much later by Kato \cite{Kato86} and
Kerz \cite{Kerz09} who showed that Bloch's formula also holds if one uses Milnor $K$-theory
instead of Quillen $K$-theory. From the point of view of algebraic cycles, this is a
significant improvement because Milnor $K$-theory is much simpler to study than Quillen $K$-theory.

The theory of Chow groups with modulus is a recent generalization of the classical Chow groups.
One of the main objectives of Chow groups with modulus is to provide a cycle theoretic
description of the $K$-theory of singular varieties in the same way as
the classical Chow groups do for the $K$-theory of smooth varieties.
Bloch's formula for the Chow groups with modulus is one of the most important steps in
reaching this goal. However, very few cases of this formula are currently known 
(see \cite{BKS}, \cite{Gupta-Krishna} and \cite{Krishna-ANT}), and its general case
is one of the most challenging problems in the theory of
Chow groups with modulus at present.

In this paper, we shall focus only on Bloch's formula for the Chow group of 0-cycles.
Even this special case appears to be so hard at the moment that one is not sure if
this formula would indeed be true eventually. Recently, Kerz and Saito
\cite{Kerz-Saito-2} have conjectured a weaker version of this formula for smooth varieties
over perfect fields. The main contribution of this paper is to verify Bloch's formula for
smooth projective varieties over finite fields.
We do this by proving two fundamental
results in the global ramification theory over finite fields:
a reciprocity theorem for the idele class group with modulus
{\`a} la Kato-Saito \cite{Kato-Saito-2}, and a comparison theorem for two
known notions of abelianized {\'e}tale fundamental groups with modulus. 
Below we describe our main results.

\subsection{Bloch's formula over finite fields}\label{sec:Results}
Let us fix a reduced quasi-projective scheme $X$ of 
pure dimension $d \ge 1$ over a field $k$ and an effective Cartier divisor $D \subset X$.
Let $\sK^M_{d, (X,D)}$ be the Nisnevich
sheaf of relative Milnor $K$-theory on $X$, defined as the kernel of the
canonical surjection $\sK^M_{d, X} \surj \iota_*(\sK^M_{d,D})$, where $\iota \colon D \inj X$
is the inclusion.
We let $C_{KS}(X, D) = H^d_\nis(X, \sK^M_{d, (X,D)})$ and call it `the Kato-Saito idele class group'.
Let $x \in X \setminus D$ be a regular closed point.
One then knows by \cite{Kato86} (see also \cite[Lemma~3.7]{Gupta-Krishna-REC})
that there is a canonical isomorphism
$\Z \xrightarrow{\cong} K^M_0(k(x)) \xrightarrow{\cong} 
H^d_x(X, \sK^M_{d, (X,D)})$, where the latter is the Nisnevich cohomology 
with support. Hence, using the `forget support' map for $x$ and extending
it linearly to the free abelian group on all regular closed points of
$X \setminus D$, we get the `cycle class map' 
\begin{equation}\label{eqn:Cycle-map}
\cyc_{X|D} \colon \sZ_0(X_\reg \setminus D) \to C_{KS}(X, D).
\end{equation}
It is not known if this map factors through $\CH_0(X|D)$ even if $X$ is smooth.
The following result settles Bloch's formula for smooth projective varieties over finite
fields.

\begin{thm}\label{thm:Main-0}
Let $k$ be either a finite field or an algebraic closure of a finite field. 
Let $X$ be a smooth projective variety over $k$
 and let $D \subset X$ be an effective Cartier divisor. 
Then the cycle class map induces an isomorphism 
\[
\cyc_{X|D} \colon \CH_0(X|D) \xrightarrow{\cong} C_{KS}(X, D).
\]
\end{thm}

\vskip .3cm

If $X$ is not regular, it is perhaps not expected that
\thmref{thm:Main-0} would hold. Nonetheless, we can prove the following version of 
Bloch's formula.
Let $X$ be an integral normal projective variety of dimension $d \ge 1$ over 
a finite field and $D$ an effective Cartier divisor on $X$. 
Assume that $X \setminus D$ is regular.
Let $\sI_D$ denote the ideal sheaf of $D$. 
For any $n \ge 1$, let $nD \subset X$ be the effective Cartier divisor 
defined by
$\sI^n_D$.

\begin{thm}\label{thm:Main-1*}
The cycle class map induces an isomorphism of pro-abelian groups
\[
\cyc^{\bullet}_{X|D} \colon \{\CH_0(X|nD)\}_{n \in \N} 
\xrightarrow{\cong} \{C_{KS}(X, nD)\}_{n \in \N}.
\]
\end{thm}

\vskip .3cm

\subsection{Reciprocity theorem for Kato-Saito idele class group}\label{sec:RKSCG}
Let $X$ be an integral and normal quasi-projective variety over a field $k$ and 
$D \subset X$ a closed subscheme of pure codimension one whose complement $U$ is regular.
In this case, there is a notion of fundamental group with modulus $\pi^{\ab}_1(X,D)$
due to Deligne and Laumon \cite{Laumon}. However, it is not known if
$\pi^{\ab}_1(X,D)$ admits a class field theory in terms of the Kato-Saito idele class group with
modulus.
This is arguably the main obstacle in finding a bridge between the
cycle-theoretic and $K$-theoretic ramified class field theories.
In this paper, we completely solve this problem, which immediately implies
\thmref{thm:Main-0}. 

Our solution to the above problem consists of two main steps.
The first step is a reciprocity theorem for $C_{KS}(X,D)$ which goes as follows.
In \cite{Gupta-Krishna-REC}, we introduced a new fundamental group with 
modulus $\pi^{\adiv}_1(X,D)$ (see \S~\ref{sec:Rec-*}) and
constructed a reciprocity map $\rho_{X|D} \colon C_{KS}(X,D)
\to \pi^{\adiv}_1(X,D)$. 
The key difference between $\pi^{\adiv}_1(X,D)$ and
$\pi^{\ab}_1(X,D)$ is that the former measures the ramification of a finite
{\'e}tale covering of $U$ at the
generic points of $D$ while the latter at the points on the 
scheme theoretic inverse image of $D$ on the normalizations 
of integral curves not contained in $D$. 
If $X$ is projective over $k$, there are
degree maps $\deg \colon C_{KS}(X,D) \to \Z$ and 
$\deg' \colon \pi^{\adiv}_1(X,D) \to \wh{\Z}$. We let
$ C_{KS}(X,D)_0 = \Ker(\deg)$ and $\pi^{\adiv}_1(X,D)_0 = \Ker(\deg')$.
We consider $\pi^{\adiv}_1(X,D)$ and $C_{KS}(X,D)$ as topological abelian groups with the profinite
and the discrete topology, respectively.
We show the following.

\begin{thm}\label{thm:Rec-mod}
Assume that $k$ is a finite field. Then the reciprocity map
$\rho_{X|D}$ is injective with dense
image. It induces an isomorphism of finite groups
\[
\rho^0_{X|D} \colon C_{KS}(X,D)_0 \xrightarrow{\cong} \pi^{\adiv}_1(X,D)_0
\]
if $X$ is projective over $k$.
\end{thm}

\vskip .3cm

\subsection{Comparison between two fundamental groups with modulus}
\label{sec:Compn}
In the second step, we prove the following result which may independently act as a key
tool in understanding the ramifications of a finite morphism
from a normal to a regular scheme. This eliminates the difference between
the two notions of ramifications discussed above.

\begin{thm}\label{thm:Main-5}
Let $X$ be a smooth projective variety of pure dimension $d \ge 1$ over a finite
field and let $D \subset X$ be an effective Cartier divisor. 
Then the surjections $\pi^{\ab}_1(X \setminus D) \surj \pi^{\adiv}_1(X,D)$ and
$\pi^{\ab}_1(X \setminus D) \surj \pi^{\ab}_1(X,D)$ have identical kernels.
In particular, there is a canonical isomorphism of profinite groups
\[
\pi^{\adiv}_1(X,D) \cong \pi^{\ab}_1(X,D).
\]
\end{thm}

\vskip .3cm

When $D_\red$ is a simple normal crossing divisor on $X$, the isomorphism
between the tame quotients of $\pi^{\adiv}_1(X,D)$ and $\pi^{\ab}_1(X,D)$
can be deduced from the main results of \cite{Kerz-Schmidt-10} and
\cite{Grothendieck-Murre}.

\vskip .3cm

\subsection{Application to Nisnevich descent for Chow groups 
with modulus}\label{sec:Desc-MC}
If $D \subset X$ is an effective Cartier divisor on a smooth scheme,
the Nisnevich hypercohomology of the sheafified cycle complex $z^p(X|D, *)$ 
with modulus \cite{Binda-Saito}
gives rise to the motivic cohomology with modulus $H^p_{\sM}(X|D, \Z(q))$
together with a canonical map $\CH^p(X|D) \to H^{2p}_{\sM}(X|D, \Z(p))$.
The Nisnevich descent for the Chow groups with modulus asks
whether this map is an isomorphism.
Using \cite[Lemma~3.9]{Gupta-Krishna-2} and \cite[Theorem~1]{Rulling-Saito},
we get the following consequence of \thmref{thm:Main-0}.
This improves \cite[Theorem~2]{Rulling-Saito}.

\begin{cor}\label{cor:Motivic-coh}
Let $X$ be a smooth projective variety of pure dimension $d \ge 1$ over a finite
field and let $D \subset X$ be an effective Cartier divisor such that
$D_\red$ is a simple normal crossing divisor. Then the canonical map
of pro-abelian groups
\[
\{\CH_0(X|nD)\}_{n \in \N} \to  \{H^{2d}_{\sM}(X|nD, \Z(d))\}_{n \in \N}
\]
is an isomorphism.
\end{cor}

\vskip .3cm

\subsection{Application to the functoriality of Kato-Saito idele class
group}\label{sec:PF00}
A very useful application of the main results and which was not known
before is the following transfer property of $C_{KS}(X,D)$.

\begin{cor}\label{cor:PF**}
Let $X$ be a smooth projective variety of pure dimension $d \ge 1$ over a finite
field $k$ and let $D \subset X$ be an effective Cartier divisor. 
Let $f \colon X' \to X$ be a proper map from a smooth projective 
variety $X'$ of dimension $d' \ge 1$ over $k$. Let $D' \subset X'$ be an
effective Cartier divisor such that $f^*(D) \subset D'$.
Then there is a push-forward map
\[
f_* \colon C_{KS}(X',D') \to C_{KS}(X,D).
\]
This is an isomorphism if $D' = f^*(D)$ and $f \colon X' \setminus D' \to
X \setminus D$ is an isomorphism. 
\end{cor}

We end the discussion of our main results and their applications with the following.

\begin{conj}\label{conj:PF-KS}
Corollary~\ref{cor:PF**} holds if $k$ is any perfect field.
\end{conj}

\vskip .3cm

\subsection{An overview of proofs}\label{sec:Outline}
We now give a brief overview of the contents of this paper.
The proofs of our main results are based two key ingredients among
several steps.

The first is to show the finiteness of the degree zero 
Kato-Saito idele class group $C_{KS}(X,D)_0$ when the base field is finite
(see \thmref{thm:KS-fin}). 
It is known that the Kerz-Saito idele class group $\CH_0(X|D)_0$ in this case 
is finite
(see \cite[Theorem~8.1]{Esnault-Kerz}). Since the latter is a quotient of
the Wiesend class group $W(X\setminus D)$ and there is a continuous
reciprocity homomorphism from $W(X\setminus D)$ to the limit (over $n$) of
$C_{KS}(X,nD)$, one would expect that this would descend to a map of 
pro-abelian groups $\{\CH_0(X|nD)\} \to \{C_{KS}(X,nD)\}$.
This would imply the finiteness of  $C_{KS}(X,D)_0$. 
Unfortunately, this approach breaks down because of the nature of the
topology of  $W(X\setminus D)$ (see \cite[\S~4.5]{Gupta-Krishna-REC} for details).
We use some results in $K$-theory of surfaces and some 
class field theory results of
Bloch \cite{Bloch81} and Kato-Saito \cite{Kato-Saito-2} to independently prove the
finiteness theorem. As an application, we prove a very general finiteness
theorem for the Kato-Saito idele class group of non-proper schemes over finite
fields (see \corref{cor:KS-fin-aff}).

The second key ingredient is a moving lemma for the Chow group of 0-cycles
with modulus (see \thmref{thm:MLC}). Combined with some results of
Kato \cite{Kato89} and Matsuda \cite{Matsuda}, this allows us to prove the 
equivalence
between the filtrations of the first {\'e}tale cohomology of the function
field of $X$ whose Pontryagin duals define the two fundamental groups with
modulus $\pi^{\adiv}_1(X,D)$ and $\pi^{\ab}_1(X,D)$.

Besides the finiteness theorem, the proof of \thmref{thm:Rec-mod} also relies
on some results about $\pi^{\adiv}_1(X,D)$ already shown in
\cite{Gupta-Krishna-REC} and the ramification theory
results proven in op. cit. and \S~\ref{sec:Ramification}. 
The proofs of the remaining main results are based on \thmref{thm:Rec-mod}
and the above key steps.

We prove the finiteness theorem in \S~\ref{sec:KS-surface},
\S~\ref{sec:3-term} and \S~\ref{sec:Fin-KS}. \thmref{thm:Rec-mod} 
is proven in \S~\ref{sec:Rec-Iso}. 
We construct the cycle class map in pro-setting
in \S~\ref{sec:Curve-based} which is used in the proof of \thmref{thm:Main-1*}.
The moving lemma is
proven in \S~\ref{sec:ML} and \S~\ref{sec:Move-2}. We prove the remaining
main results in \S~\ref{sec:BF}.

\vskip .3cm

\subsection{Notation}\label{sec:Notns}
In this paper, $k$ will be the base field of characteristic
$p \ge 0$. For most parts, this
will be a finite field in which case we shall let $q = p^s$ be the
order of $k$ for some integer $s \ge 1$. A {\sl scheme} will usually
mean a separated and essentially of finite type $k$-scheme.
We shall denote the category of such schemes by $\Sch_k$.
The product $X \times_{\Spec(k)} Y$ in $\Sch_k$ will be written as
$X \times Y$. If $k \subset k'$ is an extension of fields and
$X \in \Sch_k$, we shall let $X_{k'}$ denote $X \times \Spec(k')$. 
%We shall write $X_{\ov{k}}$ as $\ov{X}$.
For a subscheme $D \subset X$, we shall let $|D|$ denote the set of points
lying on $D$.

For a morphism $f \colon X' \to X$ of schemes and
$D \subset X$ a subscheme, we shall write $D \times_X X'$ as $f^*(D)$.
For a point $x \in X$,
we shall let $\fm_x$ denote  the maximal ideal of $\sO_{X,x}$ and
$k(x)$ the residue field of $\sO_{X,x}$. We shall let $\sO^h_{X,x}$ (resp.
$\sO^{sh}_{X,x}$) denote the Henselization (resp. strict Henselization)
of $\sO_{X,x}$. We shall let $\ov{\{x\}}$ denote
the closure of $\{x\}$ with its integral subscheme structure.  
We shall let $k(X)$ denote the total ring of quotients of $X$.
If $X$ is reduced, we shall let $X_n$ denote the normalization of $X$.

If $X$ is a Noetherian scheme, we shall consider sheaves on $X$
and their cohomology
with respect to the Nisnevich topology unless mentioned
otherwise. 
%In particular, all sheaf cohomology groups will be considered with respect to the Nisnevich topology.
We shall let $X_{(q)}$ (resp. $X^{(q)}$) denote the set of points on $X$ 
of dimension (resp. codimension) $q$. 
We shall let $\pi^{\ab}_1(X)$ denote the abelianized {\'e}tale fundamental
group of $X$. We shall consider
$\pi^{\ab}_1(X)$ as a topological abelian group with its profinite
topology.

For a field $K$ of characteristic $p > 0$ 
with a discrete valuation $\lambda$, we shall let
$K_\lambda$ denote the fraction field of the Henselization of the
ring of integers associated to $\lambda$.
For a Galois extension (possibly infinite) $K \subset L$ of fields,
we shall let $G(L/K)$ denote the Galois group of $L$ over $K$.
We shall let $G_K$ denote the absolute Galois group of $K$.
All Galois groups will be considered as topological abelian groups 
with their profinite topology. If $A$ is an abelian group, we let
${A}/{\infty} = {\underset{m \in \N}\varprojlim} A/{mA}$.

\vskip .3cm

Throughout this paper, we shall freely use the notations and terminology of
\cite{Gupta-Krishna-REC} without recalling them. In particular,
we refer the reader to \cite[\S~2, 3]{Gupta-Krishna-REC} for the notion
of Parshin chains and their Milnor $K$-theory and
the idele class group $C(X,D)$.

\section{Idele class group of a surface}
\label{sec:KS-surface}
Our goal in the next several sections is to prove first of the two
key steps for proving the main results: 
the finiteness of the
degree zero Kato-Saito idele class group $C_{KS}(X,D)_0$.
The proof of this step is by induction on the dimension of the scheme.
In this section, we recall the degree maps for the
Kato-Saito idele class group with modulus and prove a crucial step showing that
the kernel of the degree map for a surface
is a torsion group of bounded exponent. 
Recall that for any excellent scheme $X$ of dimension $d$ and closed
subscheme $D \subset X$, the Kato-Saito idele class group with modulus
$C_{KS}(X,D)$ is the Nisnevich cohomology group $H^d_\nis(X, \sK^M_{d, (X,D)})$, 
where $\sK^M_{i, (X,D)}$ is the sheaf of relative Milnor $K$-groups
(see \cite[\S~2.4]{Gupta-Krishna-REC}).

\subsection{The degree map for Kato-Saito idele class group}
\label{sec:Deg-*}
We define the
degree map for $C_{KS}(X,D)$ and prove some of its properties.
Let $k$ be any field and $X$ a projective integral scheme of dimension 
$d \ge 1$ over $k$. Let $D \subset X$ be a nowhere dense closed subscheme.
Let $K$ denote the function field of $X$.

For a closed subscheme $Y \subset X$, we let
$C^Y_{KS}(X,D) = H^d_Y(X, \sK^M_{d, (X,D)})$ denote the Nisnevich cohomology with 
support. For a closed point $x \in X$, we let $C_{KS}(X_x, D_x) =
C^x_{KS}(\Spec(\sO^h_{X,x}), \Spec({\sO^h_{X,x}}/{\sI^h_{D,x}}))$.

By \cite[\S~1]{Kato86}, there is, for every $n \ge 0$, a complex of Nisnevich 
sheaves on $X$ given by
\begin{equation}\label{eqn:KS-Chow-0}
\sK^M_{n,X} \to {\underset{x \in X^{(0)}}\oplus} (\iota_x)_*(K^M_{n}(k(x)))
\to {\underset{x \in X^{(1)}}\oplus} (\iota_x)_*(K^M_{n-1}(k(x))) \to
\cdots 
\hspace*{2cm}
\end{equation}
\[
\hspace*{3.8cm} \cdots
\to {\underset{x \in X^{(d-1)}}\oplus} (\iota_x)_*(K^M_{n-d+1}(k(x))) 
\xrightarrow{\partial}
{\underset{x \in X^{(d)}}\oplus} (\iota_x)_*(K^M_{n-d}(k(x))) \to 0.
\]
%should we write $\sK^M_{n, x}$ instead of $K^M_{n}(k(x))$ because this is not a constant sheaf 
%on Nisn. site of \{x\}. This notation is from Kato-Saito paper. 

An elementary cohomological argument shows that by taking the cohomology of 
the sheaves in this complex, one gets a canonical  
homomorphism $\nu_x \colon H^r_x(X, \sK^M_{d,X}) \to K^M_{d-r}(k(x))$
for $x \in X^{(r)}$, and a commutative square
\begin{equation}\label{eqn:KS-Chow-1}
\xymatrix@C.8pc{
H^r_y(X, \sK^M_{d,X}) \ar[r]^-{\nu_y} \ar[d]_-{\partial} & K^M_{d-r}(k(y))
\ar[d]^-{\partial} \\
H^{r+1}_x(X, \sK^M_{d,X}) \ar[r]^-{\nu_x} & K^M_{d-r-1}(k(x))}
\end{equation}
for $y \in X^{(r)}$ and $x \in X^{(r+1)} \cap \ov{\{y\}}$
(e.g., see \cite[(2.1.2)]{Kato-Saito-2}).

\vskip .2cm

For any $n \ge 0$, let $\CH^F_n(X)$ denote the Chow group of cycles of
dimension $n$ on $X$ in the sense of \cite[Chapter~1]{Fulton}.
We prove the following.

\begin{lem}\label{lem:KS-Chow}
Let $Y \subset X$ be a closed immersion. Then the following hold.
\begin{enumerate}
\item
There is a canonical homomorphism
\[
\nu_Y \colon C^Y_{KS}(X,D) \to \CH^F_0(Y)
\]
such that for any regular closed point $x \in Y \setminus D$, the composition
\[
\lambda_x \colon \Z \xrightarrow{\cong} C^x_{KS}(X,D) \to C^Y_{KS}(X,D) \to 
\CH^F_0(Y)  
\]
has the property that $\lambda_x(1) = [x]$, the cycle class of $x$.
\item
If $\iota \colon Z \inj Y$ is a closed immersion, then the diagram
\begin{equation}\label{eqn:KS-Chow-2}
\xymatrix@C1pc{
C^Z_{KS}(X,D) \ar[r] \ar[d]_-{\iota_*} & \CH^F_0(Z) \ar[d]^-{\iota_*} & \\
C^Y_{KS}(X,D) \ar[r] & \CH^F_0(Y)}
\end{equation}
commutes, where the left vertical 
arrow is the canonical homomorphism between cohomology groups
with support and the right vertical arrow is the push-forward homomorphism.
\item
The map $C^Y_{KS}(X,D) \to \CH^F_0(Y)$ is an isomorphism
if $Y \subset X_\reg \setminus D$.
\end{enumerate}
\end{lem}
\begin{proof}
To prove parts (1) and (2) of the lemma, it suffices to consider the
case when $D = \emptyset$ using the canonical map
$C^Y_{KS}(X,D) \to C^Y_{KS}(X)$ which is clearly functorial for closed
immersions $Z \subset Y$.

We now recall 
%a basic fact that if $W \subset Z$ is a closed immersion
%in $\Sch_k$, then the excision theorem for the cohomology with support
%implies that the restriction map
%\begin{equation}\label{eqn:Coh-support}
%{\underset{Z_1}\varinjlim} H^q_{W \setminus Z_1}(Z \setminus Z_1, \sF) \to 
%{\underset{Z_2}\varinjlim} H^q_{W \setminus Z_2}(Z \setminus Z_2, \sF)
%\end{equation}
%is an isomorphism for any Nisnevich sheaf $\sF$ on $X$ and integer $q \ge 0$, 
%if we let $Z_1$ run through all nowhere dense closed subschemes of
%$W$ and $Z_2$ run through all closed subschemes of $Z$ which do not contain 
%any generic point of $W$.
%This implies 
that for any Nisnevich sheaf $\sF$ on $X$,
the coniveau spectral sequence for $H^*_Y(X, \sF)$
%$H^*_Y(X, \sK^M_{d,X})$ is 
is of the form
\begin{equation}\label{eqn:KS-Chow-3} 
E^{p,q}_1 = {\underset{x \in X^{(p)}\cap Y}\bigoplus} 
H^{p+q}_x(X, \sF) \Rightarrow H^{p+q}_Y(X, \sF).
\end{equation}
Using the cohomological vanishing and the exact sequence
\[
H^{p+q-1}_\nis(X_x \setminus \{x\}, \sF) \to
H^{p+q}_x(X, \sF) \to H^{p+q}_\nis(X_x, \sF)
\] 
for $x \in X^{(p)}$, we see that $E^{p,q}_1 = 0$ for $q \ge 1$.
Hence, the above spectral sequence degenerates to an exact sequence
\begin{equation}\label{eqn:KS-Chow-4-ex} 
{\underset{y \in Y_{(1)}}\bigoplus} H^{d-1}_y(X, \sF) 
\xrightarrow{\partial}  {\underset{x \in Y_{(0)}}\bigoplus} 
H^{d}_x(X, \sF) \to H^d_Y(X,\sF) \to 0.
\end{equation} 

Applying this to $\sK^M_{d,X}$, we get an exact sequence
\begin{equation}\label{eqn:KS-Chow-4} 
{\underset{y \in Y_{(1)}}\bigoplus} H^{d-1}_y(X, \sK^M_{d,X}) 
\xrightarrow{\partial}  {\underset{x \in Y_{(0)}}\bigoplus} 
H^{d}_x(X, \sK^M_{d,X}) \to C^Y_{KS}(X) \to 0.
\end{equation}

We now consider the diagram
\begin{equation}\label{eqn:KS-Chow-5} 
\xymatrix@C.8pc{
{\underset{y \in Y_{(1)}}\bigoplus} H^{d-1}_y(X, \sK^M_{d,X}) 
\ar[r]^-{\partial} \ar[d] &  
{\underset{x \in Y_{(0)}}\bigoplus} H^{d}_x(X, \sK^M_{d,X})  \ar[r] \ar[d] &
C^Y_{KS}(X) \ar@{.>}[d] \ar[r] & 0 \\
{\underset{y \in Y_{(1)}}\bigoplus} K^M_1(k(y)) \ar[r]^-{\partial} &
{\underset{x \in Y_{(0)}}\bigoplus} K^M_0(k(x)) \ar[r] & \CH^F_0(Y) \ar[r] &
0,}
\end{equation}
where the square on the left is the commutative square of
~\eqref{eqn:KS-Chow-1} with $r = d-1$. 
It is well known that the boundary map $\partial$ in the bottom row
is the map which takes a rational function to its divisor
and the right horizontal bottom arrow is the cycle class map 
(e.g., see the proof of \cite[Theorem~3]{Kato86}). In particular,
the bottom row is exact. The top row is the exact sequence
~\eqref{eqn:KS-Chow-4}. The first part of the lemma now follows.
Furthermore, it is clear from the above construction that the
~\eqref{eqn:KS-Chow-2} is commutative if $Z \subset Y$, proving (2).

If $Y \subset X_\reg \setminus D$, then we can assume by excision 
that $X$ is regular with $D = \emptyset$ so that
$\sK^M_{d, (X,D)} \cong \sK^M_{d,X}$. In this case, the two vertical arrows
in ~\eqref{eqn:KS-Chow-5} on the left are isomorphisms by
\cite[Lemma~3.7]{Gupta-Krishna-REC}. It follows that 
the third vertical arrow (which is now an honest map) is also an isomorphism.
This proves (3).
This also shows that for a closed point $x \in (X_\reg \cap Y) \setminus D$,
the composite $\lambda_x \colon H^{d}_x(X, \sK^M_{d,X}) \to \CH^F_0(Y)$
has the property stated in part (1) of the lemma.
We have thus finished the proof.
\end{proof}

Since $X$ is projective, there is a push-forward map
$\deg \colon \CH^F_0(Y) \to \Z$, which takes a closed point to the degree
of its residue field over $k$. By composition with $\nu_Y$, we get a
degree map
\begin{equation}\label{eqn:Deg-map}
\deg \colon C^Y_{KS}(X,D) \to \Z
\end{equation}
which clearly factors through the degree map $\deg \colon C_{KS}(X,D) \to \Z$.
We let $C^Y_{KS}(X,D)_0$ be the kernel of the degree map.
We define $C_{KS}(X,D)_0$ similarly.

\vskip .2cm

Recall that \cite[Lemma~1.6.3]{Kato-Saito-2} provides a recipe for
constructing a homomorphism from $C_{KS}(X,D)$ to an abelian group.
This says that giving a homomorphism from $C_{KS}(X,D)$ to an abelian
group $A$ is same as defining group homomorphisms
$K^M_d(k(P)) \to A$ (where $P$ runs through all maximal Parshin chains in $X$)
which annihilate the images (by the residue homomorphisms)
of the Milnor $K$-groups of certain $Q$-chains.
We shall refer to this as the Kato-Saito recipe in the sequel.

One can easily check (using ~\eqref{eqn:KS-Chow-4} with $Y= X$) 
that the above degree map $C_{KS}(X,D) \to \Z$ is same as the one obtained by 
the Kato-Saito recipe, where for a maximal Parshin chain 
$P = (p_0, \ldots , p_d)$,
we define our desired homomorphism to be the composite
%\begin{equation}\label{eqn:KS-recipe}
\[
K^M_d(k(P)) \cong H^0_{P_d}(X, \sK^M_{d, (X,D)}) \xrightarrow{\partial} 
H^1_{P_{d-1}}(X, \sK^M_{d, (X,D)}) \xrightarrow{\partial} \cdots 
\xrightarrow{\partial} H^d_{P_0}(X, \sK^M_{d, (X,D)}) \to \Z,
\]
%\end{equation}
in which $\partial$ denotes the boundary map and
$P_i = (p_0, \ldots , p_i)$. The last arrow is induced by
~\eqref{eqn:KS-Chow-0}.

\begin{lem}\label{lem:Deg-PF}
Let $f \colon Y \to X$ be a projective morphism with $Y$ integral.
Let $E \subset Y$ be a closed subscheme such that $f^*(D) \subset E$.
Assume that the image of $f$ is not contained in $D \cup X_{\sing}$
and the image of $E$ is nowhere dense in $X$.
Suppose that the Kato-Saito recipe defines a
push-forward map $f_* \colon  C_{KS}(Y,E) \to C_{KS}(X,D)$.
Then the diagram
\begin{equation}\label{eqn:Deg-PF-0}
\xymatrix@C1pc{
C_{KS}(Y,E) \ar[dr]^-{\deg_Y} \ar[d]_-{f_*} & \\
C_{KS}(X,D) \ar[r]_-{\deg_X} & \Z}
\end{equation}
is commutative. 
\end{lem}
\begin{proof}
By our assumption, there is a dense open $U \subset X$ away from $D$
such that $U$ and $f^{-1}(U)$ are both  regular. Furthermore,
$E \cap f^{-1}(U) = \emptyset$. It follows from
\cite[Theorem~2.5]{Kato-Saito-2} that $C_{KS}(Y,E)$ is generated by the
classes of closed points in $f^{-1}(U)$. Hence, it suffices to show
that for every closed point $y \in f^{-1}(U)$, one has
$\deg_X \circ f_*([y]) = \deg_Y([y])$.
But this follows immediately from \lemref{lem:KS-Chow}
and the fact that the proper push-forward map on the classical Chow group of
0-cycles commutes with the degree map.
\end{proof}

One easy consequence of \lemref{lem:Deg-PF} is the following.
\begin{cor}\label{cor:Deg-PF-*}
Let $D \subset D'$ be two nowhere dense closed subschemes. Then
the canonical map $C_{KS}(X,D')_0 \to C_{KS}(X,D)_0$ is surjective.
\end{cor}

\begin{lem}\label{lem:Deg-birational}
Let $f \colon X' \to X$ be a projective birational morphism
and let $D' \subset f^*(D)$. Then there is a commutative diagram
\[
\xymatrix@C1pc{
C_{KS}(X,D) \ar[dr]^-{\deg_X} \ar[d]_-{f^*} &  \\
 C_{KS}(X',D') \ar[r]_-{\deg_{X'}} & \Z.}
\]
\end{lem}
\begin{proof}
By \cite[Theorem~2.5]{Kato-Saito-2}, it suffices to show that
$\deg_{X'} \circ f^*([x]) = \deg_X([x])$ for every closed point
$x \in (X_\reg \cap f(U)) \setminus D$, where $U$ is an open subset of $X'$ on 
which $f$ is an isomorphism. But this is obvious by \lemref{lem:KS-Chow}.
\end{proof}

Let $C(X,D)$ denote the idele class group with modulus 
due to Kerz \cite{Kerz11} (see \cite[\S~3]{Gupta-Krishna-REC}).
Recall from \cite[Theorem~8.2]{Kerz11} 
(or \cite[Theorem~3.8]{Gupta-Krishna-REC}) that when $X \setminus D$ is
regular, there are canonical maps
\begin{equation}\label{eqn:ICG-maps}
\Z \cong K^M_0(k(x)) \xrightarrow{\tau_x} C(X,D) 
{\underset{\cong}{\xrightarrow{\psi_{X|D}}}} C_{KS}(X,D)
\end{equation}
for every $x \in X_{(0)} \setminus D$, where $\tau_x$ 
is induced by the inclusion of the Parshin chain of length
zero and the composite arrow is the forget support map
$\tau'_x \colon K^M_0(k(x)) \cong C^x_{KS}(X,D) \to C_{KS}(X,D)$.
Furthermore, $\psi_{X|D}$ is an isomorphism.
There is a degree map $\deg \colon C(X,D) \to \Z$ by 
\cite[Proposition~4.8]{Gupta-Krishna-REC}, whose kernel is $C(X,D)_0$.

The following result shows the compatibility
between the degree maps for $C(X,D)$ and $C_{KS}(X,D)$ via $\psi_{X|D}$.

\begin{lem}\label{lem:KS-K-deg}
Assume that $X \setminus D$ is regular. Then the diagram
\begin{equation}\label{eqn:KS-K-deg-0}
\xymatrix@C.8pc{
C(X,D) \ar[dr]^-{\deg} \ar[d]_-{\psi_{X|D}} & \\
C_{KS}(X,D) \ar[r]_-{\deg} & \Z}
\end{equation}
is commutative.
\end{lem}
\begin{proof}
Let $U = X \setminus D$. Since $\psi_{X|D}$ is an isomorphism,
it suffices to show using \cite[Theorem~2.5]{Kato-Saito-2} that
for every closed point $x \in U$, the above diagram commutes if we
replace $C(X,D)$ by $K^M_0(k(x))$ via the canonical map
$\tau_x \colon K^M_0(k(x)) \to C(X,D)$, where we consider $\{x\}$ as a 
Parshin chain on $(U \subset X)$.
However, we have $\deg \circ \tau_x(1) = [k(x):k]$ by 
\cite[Proposition~4.8]{Gupta-Krishna-REC}.
Since $\psi_{X|D} \circ \tau_x(1) =
\tau'_x(1)$, we are done by \lemref{lem:KS-Chow} which says that
$\deg \circ \tau'_x(1) = [k(x):k]$.
\end{proof}

\subsection{Some $K$-theory results}\label{sec:K-thry}
In this subsection, we prove some $K$-theory results of general
interest which will be used in the proof of the finiteness theorem.
The following is an elementary but very useful lemma.

\begin{lem}\label{lem:Milnor-tor}
Let $Z$ be a Noetherian scheme over a field of
characteristic $p > 0$ and let $W \subset Z$ be a closed
subscheme defined by a nilpotent ideal sheaf. Then
$H^i_{\nis}(Z, \sK^M_{j, (Z,W)})$ is a $p$-primary torsion group of bounded 
exponent (which depends only on $j$) for all $i \ge 0$ and $j \ge 1$.
\end{lem}
\begin{proof}
We show that the sheaf $\sK^M_{j, (Z,W)}$ itself is $p$-primary 
torsion of bounded
exponent for every $j \ge 1$. This will prove the lemma.
By the definition of  $\sK^M_{j, (Z,W)}$ (see \cite[(2.2)]{Gupta-Krishna-REC}), 
it suffices to prove the statement for $j =1$.
In this case, we can use the iterative process to reduce the problem to the
case when the ideal sheaf $\sI_W$ defining $W$ is square-zero.
But then, we must have $\sK^M_{1, (Z,W)} \cong \sI_W$ and the latter is
a $p$-torsion sheaf. 
\end{proof}

\begin{lem}\label{lem:Conductor}
  Let $f \colon X_n \to X$ be the normalization morphism for a reduced Noetherian
  scheme $X$ such that $f$ is a finite morphism (e.g., $X$ is essentially of finite type
  over a field). Then we can find a conductor closed subscheme $Y \inj X$ such that
  for $Y' = f^*(Y)$, the canonical restriction map
  \[
    {f_*(\sK_{2,X_n})}/{\sK_{2,X}} \to {f_*(\sK_{2,Y'})}/{\sK_{2,Y}}
\]
is an isomorphism of Nisnevich sheaves on $X$.
\end{lem}
\begin{proof}
We fix a conductor subscheme $Y \subset X$ (which always exists) for $f$ and
  look at the commutative diagram of Nisnevich sheaves (on $X$)
  \begin{equation}\label{eqn:Conductor-0}
    \xymatrix@C.8pc{
      \sK_{2,(X,Y)} \ar[r] \ar[d]_-{f^*}  & \sK_{2, X} \ar[r] \ar[d]^-{f^*} & \sK_{2,Y}
        \ar[d]^-{f^*} \ar[r] & 0 \\
      f_*(\sK_{2,(X_n,Y')}) \ar[r] & f_*(\sK_{2, X_n}) \ar[r] &
      f_*(\sK_{2,Y'}) \ar[r] & 0.}
  \end{equation}
  Note that the maps $\sK_{2, X} \to  \sK_{2,Y}$ and $\sK_{2, X_n} \to  \sK_{2,Y'}$
are surjective by 
\cite[Proposition~10, Theorem~13]{Kerz10} because the surjectivity clearly
holds for the Milnor $K$-theory sheaf. In particular, the top row is exact.
The bottom row is exact because $f$ is finite, and a finite
push-forward is an exact functor on the category of Nisnevich sheaves.

On the other hand, we also have a double relative $K$-theory exact sequence
    \begin{equation}\label{eqn:Conductor-1}
\sK_{2,(X,Y)} \xrightarrow{f^*}  f_*(\sK_{2,(X_n,Y')}) \to \sK_{1, (X, X_n, Y)}
\to  \sK_{1,(X,Y)}  \xrightarrow{f^*}  f_*(\sK_{1,(X_n,Y')}),
\end{equation}
where $\sK_{1, (X,X_n, Y)}$ is the sheaf of double relative $K$-theory.
Since $\sK_{1, (X,Y)} \cong (1 + \sI_Y)^{\times} \inj (1 + \sI_{Y'})^{\times} \cong 
\sK_{1, (X_n,Y')}$, and
 $\sK_{1, (X,X_n, Y)} \cong {\sI_Y}/{\sI^2_Y} \otimes_{Y'}
 \Omega^1_{{Y'}/{Y}}$ by \cite[Theorem~0.2]{Geller-Weibel},
a combination of ~\eqref{eqn:Conductor-0} and ~\eqref{eqn:Conductor-1}
yields (via a diagram chase) an exact sequence
\begin{equation}\label{eqn:Conductor-2}
{\sI_Y}/{\sI^2_Y} \otimes_{Y'} \Omega^1_{{Y'}/{Y}} \to
{f_*(\sK_{2,X_n})}/{\sK_{2,X}} \to {f_*(\sK_{2,Y'})}/{\sK_{2,Y}} \to 0.
\end{equation}
Comparing this exact sequence for $Y$ and $2Y$ (where $mY \subset X$ is
defined by $\sI^m_Y$), we get the desired isomorphism if we
choose our conductor subscheme to be $2Y$.
\end{proof}

The next result is of independent interest and plays a fundamental
role in the study of 0-cycles on singular schemes.
Recall that for any $X \in \Sch_k$, the edge map of the
Thomason-Trobaugh spectral sequence  yields a split surjection
$K_1(X) \surj H^0_\nis(X, \sO^{\times}_X)$. We let $SK_1(X)$ denote the
kernel of this surjection.

\begin{prop}\label{prop:Sing-surface}
Let $X$ be a reduced quasi-projective surface over a
field and let $f \colon X_n \to X$ be the normalization map.
Then we can find a conductor closed subscheme $Y \inj X$ such that
for $Y' = f^*(Y)$, there is an exact sequence
\[
0 \to \frac{SK_1(X_n)}{SK_1(X)} \to \frac{SK_1(Y')}{SK_1(Y)} \to
C_{KS}(X)_0 \to C_{KS}(X_n)_0 \to 0.
\]
\end{prop}
\begin{proof}
A version of this result is shown in \cite[Proposition~2.3]{Krishna-JKT}
for surfaces over $\C$. We shall modify that argument and, in particular, use
\lemref{lem:Conductor} to prove the
proposition. By \cite[Proposition~10, Theorem~13]{Kerz10}, we
can replace $C_{KS}(X)$ by $H^2_\nis(X, \sK_{2,X})$, where
$\sK_{*,X}$ is the Quillen $K$-theory sheaf. The same for $X_n$ too.

We choose a conductor subscheme $Y \inj X$ and consider the exact sequence
of Nisnevich sheaves on $X$:
\[
\sK_{2,X} \to f_*(\sK_{2,X_n}) \to {f_*(\sK_{2,X_n})}/{\sK_{2,X}} \to 0.
\]
Since $H^i_\nis(Z, \sK_{2,Z}) = 0$ for $i > 0$ for a semilocal scheme $Z$,
the Leray spectral sequence tells us that 
$H^i_\nis(X_n, \sK_{2,X_n}) \cong H^i_\nis(X, f_*(\sK_{2,X_n}))$ for all
$i \ge 0$. Since the kernel and cokernel of the map 
$\sK_{2,X} \to f_*(\sK_{2,X_n})$ are supported on $Y$, it follows from 
the above sheaf exact sequence that there is an exact cohomology 
sequence
\begin{equation}\label{eqn:Sing-surface-0} 
0 \to \frac{H^1_\nis(X_n, \sK_{2,X_n})}{H^1_\nis(X, \sK_{2,X})} \to
H^1_\nis(X, {f_*(\sK_{2,X_n})}/{\sK_{2,X}}) \to 
H^2_\nis(X, \sK_{2,X}) \to H^2_\nis(X_n, \sK_{2,X_n}) \to 0.
\end{equation}
We can replace the last two terms by their degree zero subgroups
without disturbing the exactness.

Since 
\begin{equation}\label{eqn:Sing-surface-1} 
H^2_\nis(X, \sK_{3,X}) \to SK_1(X) \to H^1_\nis(X, \sK_{2,X}) \to 0
\end{equation}
is exact by the Thomason-Trobaugh spectral sequence 
\cite[Theorem~10.8]{TT} and since $H^2_\nis(X, \sK_{3,X}) \surj
H^2_\nis(X_n, \sK_{3,X_n})$, the left-end term in
~\eqref{eqn:Sing-surface-0} is same as the quotient ${SK_1(X_n)}/{SK_1(X)}$.
Since the edge map of the spectral sequence
induces an isomorphism $SK_1(Z) \cong H^1_\nis(Z, \sK_{2,Z})$ for
any Noetherian scheme $Z$ of Krull dimension at most one,
what we are left to show is that the canonical map
\[
H^1_\nis(X, {f_*(\sK_{2,X_n})}/{\sK_{2,X}}) \to 
H^1_\nis(Y, {f_*(\sK_{2,Y'})}/{\sK_{2,Y}})
\]
is an isomorphism if we replace $Y$ by some of its infinitesimal thickenings.
But this is a direct consequence of \lemref{lem:Conductor} (and exactness of
$f_*$).
\end{proof}

\begin{comment}
To prove this, we compare the $K$-theory exact sequences for the
pairs $(X,Y)$ and $(X_n, Y')$ and use the isomorphism
$\sK_{1, (X,Y)} \cong (1 + \sI_Y)^{\times} \cong (1 + \sI_{Y'})^{\times} \cong 
\sK_{1, (X_n,Y')}$
to get an exact sequence of sheaves
\[
\sK_{1, (X,X_n, Y)} \to  {f_*(\sK_{2,X_n})}/{\sK_{2,X}} \to
{f_*(\sK_{2,Y'})}/{\sK_{2,Y}} \to 0,
\]
where $\sK_{1, (X,X_n, Y)}$ is the sheaf of double relative $K$-theory.
Note that the middle arrow is surjective by 
\cite[Proposition~10, Theorem~13]{Kerz10} because this surjectivity clearly
holds for the Milnor $K$-theory sheaf.

Now, one knows that $\sK_{1, (X,X_n, Y)} \cong {\sI_Y}/{\sI^2_Y} \otimes_{Y'}
\Omega^1_{{Y'}/{Y}}$ by \cite[Theorem~0.2]{Geller-Weibel}. 
We thus get an exact sequence
\begin{equation}\label{eqn:Sing-surface-1} 
{\sI_Y}/{\sI^2_Y} \otimes_{Y'} \Omega^1_{{Y'}/{Y}} \to
{f_*(\sK_{2,X_n})}/{\sK_{2,X}} \to {f_*(\sK_{2,Y'})}/{\sK_{2,Y}} \to 0.
\end{equation}

Comparing this exact sequence for $Y$ and $2Y$ (where $mY \subset X$ is
defined by $\sI^m_Y$), we get the desired isomorphism if we
choose our conductor subscheme to be $2Y$.
\end{proof}
\end{comment}

\begin{lem}\label{lem:SK-1}
Let $Y$ be a one-dimensional Noetherian $k$-scheme, where $k$ is a 
finite field.
Then $SK_1(Y) = 0$ if $Y$ is affine and $SK_1(Y)$ is a torsion group
of bounded exponent if $Y$ is projective.
\end{lem}
\begin{proof}
The affine case follows from \cite{Nestler}. We therefore assume that
$Y$ is projective. The Thomason-Trobaugh spectral sequence 
\cite[Theorem~10.8]{TT} yields isomorphisms
(see ~\eqref{eqn:Sing-surface-1})
\begin{equation}\label{eqn:SK-1-0}
SK_1(Y) \cong H^1_\nis(Y, \sK_{2,Y}) \cong H^1_\nis(Y, \sK^M_{2,Y})
\cong H^1_\zar(Y, \sK^M_{2,Y}).
\end{equation}
By \lemref{lem:Milnor-tor}, we can therefore assume that $Y$ is reduced.
Let $f \colon Y_n \to Y$ be the normalization map.

Since the kernel of the first arrow in the exact sequence
\[
  \sK_{2,Y} \xrightarrow{f^*}   f_*(\sK_{2,Y_n}) \to {f_*(\sK_{2,Y_n})}/{\sK_{2,Y}}
  \to 0
\]
is supported on $Y_\sing$, there is an exact cohomology sequence
\begin{equation}\label{eqn:SK-1-00}
H^0_\nis(Y, {f_*(\sK_{2,Y_n})}/{\sK_{2,Y}}) \to 
H^1_\nis(Y, \sK_{2,Y}) \to H^1_\nis(Y_n, \sK_{2,Y_n}) \to 0.
\end{equation}
Combining this with \lemref{lem:Conductor}, we can find a conductor closed subscheme
$Z \subset Y$ such that one has an exact sequence
\begin{equation}\label{eqn:SK-1-000}
H^0_{\nis}(Z, {f_*(\sK_{2,Z'})}/{\sK_{2,Z}}) \to 
H^1_\nis(Y, \sK_{2,Y}) \to H^1_\nis(Y_n, \sK_{2,Y_n}) \to 0,
\end{equation}
where  $Z' = f^*(Z)$.
Since $K_2(Z') = H^0_{\nis}(Z,f_*(\sK_{2, Z'})) \surj
H^0_{\nis}(Z, {f_*(\sK_{2,Z'})}/{\sK_{2,Z}})$,
~\eqref{eqn:SK-1-0} and ~\eqref{eqn:SK-1-000} together give us an exact sequence
\[
K_2(Z') \to SK_1(Y) \to SK_1(Y_n) \to 0.
\]

Since $\dim(Z') = 0$, we see that $Z'_\red$ is the spectrum of a finite
product of finite fields. Hence, $K_2(Z'_\red) = 0$.
It follows from \lemref{lem:Milnor-tor} that $K_2(Z')$ is a $p$-primary
torsion group of bounded exponent. Since $Y_n$ is a smooth projective
curve over $k$, the $K$-theory localization sequence and the affine
case of the lemma shows that the push-forward map
$K_*(Y_n) \to K_*(k)$ induces an isomorphism $SK_1(Y_n) \cong
(k^{\times})^r$, where $r$ is the number of irreducible components of $Y$.
The result follows.
\end{proof}

\begin{comment}
% Let $Z \subset Y$ be a conductor subscheme. Then by the proof of the exact sequence 
%  ~\eqref{eqn:Sing-surface-1}, we get  an exact sequence
%\begin{equation*} 
%{\sI_Z}/{\sI^2_Z} \otimes_{Z'} \Omega^1_{{Z'}/{Z}} \to
%{f_*(\sK_{2,Y_n})}/{\sK_{2,Y}} \to {f_*(\sK_{2,Z'})}/{\sK_{2,Z}} \to 0,
%\end{equation*}
%where $Z' = f^*(Z)$. 
%As before, comparing this exact sequence for $Z$ and $2Z$, we get the isomorphism if we
%choose our conductor subscheme to be $2Y$.
%
 Using an exact sequence similar to ~\eqref{eqn:Sing-surface-1}, we choose a conductor subscheme $Z \subset Y$ such that 
 \[
 {f_*(\sK_{2,Y_n})}/{\sK_{2,Y}} \xrightarrow{\cong} {f_*(\sK_{2,Z'})}/{\sK_{2,Z}}, 
 \]
 where $Z' = f^*(Z)$. Since $Z'$ is the spectrum of a product of Artinian local rings, we have  
 a surjection 
 \[
 K_2(Z') = H^0_{\nis}(Z,f_*(\sK_{2, Z'})) \surj  H^0_{\nis}(Z, {f_*(\sK_{2,Z'})}/{\sK_{2,Z}}).
 \]
Combining the above results, we get an exact sequence
\[
K_2(Z') \to SK_1(Y) \to SK_1(Y_n) \to 0.
\]

Since $\dim(Z') = 0$, we see that $Z'_\red$ is the spectrum of a finite
product of finite fields. Hence, $K_2(Z'_\red) = 0$.
It follows from \lemref{lem:Milnor-tor} that $K_2(Z')$ is a $p$-primary
torsion group of bounded exponent. Since $Y_n$ is a smooth projective
curve over $k$, the $K$-theory localization sequence and the affine
case of the lemma shows that the push-forward map
$K_*(Y_n) \to K_*(k)$ induces an isomorphism $SK_1(Y_n) \cong
(k^{\times})^r$, where $r$ is the number of irreducible components of $Y$.
The result follows.
\end{proof}
\end{comment}

\subsection{Torsion in the idele class group of a surface}
\label{sec:KSS}
Let $X$ be an integral projective
scheme of dimension two over a finite field $k$ and 
$D \subset X$ a nowhere dense
closed subscheme. The following result is the starting point of the proof of
\thmref{thm:KS-fin}.

\begin{prop}\label{prop:Surface-tor}
$C_{KS}(X,D)_0$ is a torsion group of bounded exponent.
\end{prop}
\begin{proof}
Using the exact sequence 
\[
SK_1(D) \to C_{KS}(X,D)_0 \to C_{KS}(X)_0 \to 0
\]
and \lemref{lem:SK-1}, it suffices to show that $C_{KS}(X)_0$ is
a torsion group of bounded exponent. Using \propref{prop:Sing-surface}
and \lemref{lem:SK-1}, we can assume that $X$ is normal.

Let $f \colon \wt{X} \to X$ be a resolution of singularities of $X$
with reduced exceptional divisor $E$. 
It is then shown in the proof of \cite[Theorem~1.4]{Krishna-ANT} that 
there are isomorphisms
\begin{equation}\label{eqn:Surface-tor-0*}
C_{KS}(X) \cong C_{KS}(\wt{X}, nE) \cong \CH_0(\wt{X}| nE)
\end{equation}
for all $n \gg 0$, where the last group is the Chow group of
0-cycles with modulus (see \S~\ref{sec:cycles}). 
Using \lemref{lem:Deg-birational},
we get $C_{KS}(X)_0 \cong \CH_0(\wt{X}| nE)_0$.
We are now done by \cite[Corollary~8.7]{BKS}. 
\end{proof}

\section{Generic fibration over a curve}\label{sec:3-term}
Let $X$ be an integral projective
scheme of dimension two over a finite field $k$. 
We assume that there exists a projective dominant morphism
$f \colon X \to S$, where $S$ is an integral projective curve over $k$.
We assume  that there exists a dense open $S' \subset S$ such that the fibers of
$f$ over $S'$ are integral. 
We assume also that $f$ has a section over the generic point of $S$
and the generic fiber is regular.

We shall deduce an important corollary of \propref{prop:Surface-tor}
under the above set-up.
By \lemref{lem:KS-Chow}, we have the maps
$C_{KS}(X,D) \xrightarrow{\nu_X} \CH^F_0(X) \xrightarrow{f_*} \CH^F_0(S)$.
We let $C_{KS}(X,D)_S$ denote the kernel of the composite map.
It follows from \lemref{lem:KS-Chow} that there are inclusions
\begin{equation}\label{eqn:Deg-0-*}
C_{KS}(X,D)_S \inj C_{KS}(X,D)_0 \inj C_{KS}(X,D).
\end{equation}

Let $t \in S$ denote the generic point of $S$.
For any $s \in S$, we let $X_s = f^{-1}(s)$ denote the scheme theoretic
fiber.
For every proper closed subset $Z \subset S$, 
there is an exact sequence of Nisnevich cohomology groups with support
\[
H^1_\nis(f^{-1}(S \setminus Z), \sK^M_{2, (X,D)}) \to
H^2_{f^{-1}(Z)}(X, \sK^M_{2, (X,D)}) \to H^2_\nis(X, \sK^M_{2, (X,D)}).
\]
Taking the inductive limit over all proper closed subsets $Z$, we get
an exact sequence
\begin{equation}\label{eqn:V-seq}
SK_1(X_t) \xrightarrow{\partial}
{\underset{s \in S_{(0)}}\bigoplus} C^{X_s}_{KS}(X,D)
\to  C_{KS}(X,D) \to 0,
\end{equation}
where the second map is surjective by our assumption and
\cite[Theorem~2.5]{Kato-Saito-2}.

We now consider the diagram 
\begin{equation}\label{eqn:V-seq-0}
\xymatrix@C.8pc{
SK_1(X_t) \ar[r]^-{\partial} \ar[d]_-{f_*} &  
{\underset{s \in S_{(0)}}\bigoplus} C^{X_s}_{KS}(X,D) \ar[r] \ar[d]^-{f_*} &
 C_{KS}(X,D) \ar[r] \ar[d]^-{f_*} & 0 \\
K^M_1(k(t)) \ar[r]^-{\divf} & {\underset{s \in S_{(0)}}\bigoplus} \Z
\ar[r] & \CH^F_0(S) \ar[r] & 0,}
\end{equation}
where the left vertical arrow is induced by the
push-forward map $f_* \colon K_1(X_t) \to K_1(k(t)) \cong K^M_1(k(t))$ 
which exists because $X_t$ is regular.
%this is not clear to me. We do write about this map in the proof 
%of the following lemma. Look at equation (11.5) and para containing the same. 
The bottom exact sequence
is the standard one which defines $\CH^F_0(S)$.
The middle vertical arrow is the direct sum of
the compositions of the maps $C^{X_s}_{KS}(X,D) \to \CH^F_0(X_s)$
from \lemref{lem:KS-Chow}
and the push-forward maps $f_* \colon \CH^F_0(X_s) \to 
\CH^F_0(\Spec(k(s))) \cong \Z$. The right vertical arrow is the composition of 
the map $\nu_X \colon C_{KS}(X,D) \to \CH^F_0(X)$ from \lemref{lem:KS-Chow} and 
the push-forward map $f_* \colon \CH^F_0(X) \to \CH^F_0(S)$. 

\begin{lem}\label{lem:V-seq-00}
The diagram ~\eqref{eqn:V-seq-0} is commutative.
\end{lem}
\begin{proof}
To show that the right square in ~\eqref{eqn:V-seq-0} commutes,
it suffices to show that it commutes when restricted to each
direct summand $C^{X_s}_{KS}(X,D)$. 
Furthermore, we can replace $C^{X_s}_{KS}(X,D)$ and
$C_{KS}(X,D)$ by $C^{X_s}_{KS}(X)$ and $C_{KS}(X)$, respectively
because $f_*$ factors through the latter groups.

We now fix $s \in S_{(0)}$ and consider the diagram
\begin{equation}\label{eqn:V-seq-01}
\xymatrix@C.8pc{
C^{X_s}_{KS}(X) \ar[r] \ar[d] & \CH^F_0(X_s) \ar[r]^-{f_*} \ar[d] & \Z 
\ar[d] \\
C_{KS}(X) \ar[r] & \CH^F_0(X) \ar[r]^-{f_*} & \CH^F_0(S),}
\end{equation}
where the vertical arrows are induced by
the inclusions $X_s \inj X$ and $\{s\} \inj S$.

The middle (resp. right) vertical arrow in ~\eqref{eqn:V-seq-0} is the
composition of the top (resp. bottom) horizontal arrows in 
~\eqref{eqn:V-seq-01} by the definition of $f_*$.
The left square in ~\eqref{eqn:V-seq-01} commutes by 
\lemref{lem:KS-Chow} and the right square is well known to be
commutative (see \cite{Fulton}). This shows that
the right square of ~\eqref{eqn:V-seq-0} commutes.

We now show that the left square in ~\eqref{eqn:V-seq-0} commutes.
Since $X_t$ is a regular scheme of dimension one 
over the infinite field $k(t)$, 
it is well known (e.g., see \cite{Kerz09}) that the Gersten sequence 
~\eqref{eqn:KS-Chow-0} gives an exact sequence
\begin{equation}\label{eqn:SK-1-presentation}
K^M_2(K) \xrightarrow{\partial} 
{\underset{x \in (X_t)_{(0)}}\bigoplus} K^M_1(k(x)) \to SK_1(X_t) \to 0,
\end{equation}
and the left vertical arrow in ~\eqref{eqn:V-seq-0} is induced by the
norm maps $K^M_1(k(x)) \to K^M_1(k(t))$ for $x \in (X_t)_{(0)}$.
Note that this norm map is same as the composition
$K^M_1(k(x)) \xrightarrow{(\iota_x)_*} K_1(X_t) \xrightarrow{f_*} K_1(k(t))$.
Hence, it suffices to show that the left square commutes after we replace
the map $\partial$ by its composition with 
\[
H^1_x(X_t, \sK^M_{2,X_t}) \xrightarrow{\cong}  
K^M_1(k(x)) \to SK_1(X_t) \cong  H^1_\nis(X_t, \sK^M_{2,X_t})
\]
on the left for $x \in (X_t)_{(0)}$. 
%Note that the inidcated isomorphisms follow from
%~\eqref{eqn:KS-Chow-0} as $X_t$ is regular.
Furthermore, since the map $f_* \colon C^{X_s}_{KS}(X,D) \to K^M_0(k(s))$
factors through $ C^{X_s}_{KS}(X) \to K^M_0(k(s))$, we can also replace
$\partial$ by its composition with $C^{X_s}_{KS}(X,D) \to 
C^{X_s}_{KS}(X)$ on the right.

We now fix a closed point $x \in X_t$ and let $Y \subset X$ be
the closure of $\{x\}$ with integral subscheme structure. Then $f$
restricts to a finite dominant map $f \colon Y \to S$.
The composite map $K^M_1(k(x))  \to H^1_\nis(X_t, \sK^M_{2,X_t}) 
\xrightarrow{\partial} {\underset{s \in S_{(0)}}\bigoplus} C^{X_s}_{KS}(X)$
factors through the composition 

\[
K^M_1(k(x)) \xrightarrow{\cong} 
H^1_x(X, \sK^M_{2,X}) \xrightarrow{\partial} 
{\underset{y \in Y_{(0)}}\bigoplus} C^{y}_{KS}(X) \cong
{\underset{s \in S_{(0)}}\bigoplus} C^{Y_s}_{KS}(X) \to
{\underset{s \in S_{(0)}}\bigoplus} C^{X_s}_{KS}(X),
\]
where the left arrow is an isomorphism by \cite[Lemma~3.7]{Gupta-Krishna-REC}
because $x \in X_\reg$.
We therefore have to show that the left square in the diagram
\begin{equation}\label{eqn:V-seq-02}
\xymatrix@C.8pc{
K^M_1(k(x)) \ar[r]^-{\partial} \ar[d]_-{N} & 
{\underset{s \in S_{(0)}}\bigoplus} C^{Y_s}_{KS}(X) \ar[d]^-{f_*} 
\ar[r] & {\underset{s \in S_{(0)}}\bigoplus} \CH^F_0(Y_s) \ar[dl]^-{f_*} \\
K^M_1(k(t)) \ar[r]^-{\divf} & {\underset{s \in S_{(0)}}\bigoplus} \Z &}
\end{equation}
is commutative, where $N$ is the norm map.

Since the triangle on the right is commutative by the definition of 
$f_* \colon C^{Y_s}_{KS}(X) \to \Z$ for $s \in S_{(0)}$, we need to
show that big outer diagram in ~\eqref{eqn:V-seq-02} commutes.
But this is classical because the top composite horizontal arrow is the 
boundary map of the complex ~\eqref{eqn:KS-Chow-0}. 
One knows that this is same as the boundary map
$\partial \colon K_1(k(x)) \to {\underset{s \in Y_{(0)}}\bigoplus}
K_0(k(s))$ on Quillen $K$-groups (see \cite[Lemma~11.2]{Gupta-Krishna-1})
and Quillen \cite{Quillen} showed that this map takes an
element of $K_1(k(x))$ to its divisor.
The diagram therefore commutes by \cite[Proposition~1.4]{Fulton}.
\end{proof}

Let $V(X_t) = \Ker(SK_1(X_t) \to K_1(k(t)))$.
Since the kernel of $f_* \colon C^{X_s}_{KS}(X,D) \to \CH^F_0(\Spec(k(s)))$
is same as $C^{X_s}_{KS}(X,D)_0$, the commutative diagram
~\eqref{eqn:V-seq-0} induces a 3-term complex
\begin{equation}\label{eqn:V-seq-1}
V(X_t) \xrightarrow{\partial} {\underset{s \in S_{(0)}}\bigoplus} 
C^{X_s}_{KS}(X,D)_0 \to  C_{KS}(X,D)_{S} \to 0.
\end{equation}

Since $f$ has a section over $\Spec(k(t))$, 
it follows that the left vertical arrow in
~\eqref{eqn:V-seq-0} is split surjective.
A diagram chase shows that the middle arrow in ~\eqref{eqn:V-seq-1}
is surjective.
For any integer $m \ge 1$, let
\begin{equation}\label{eqn:Short-term}
A_m = {\underset{s \in S_{(0)} \setminus S'_\reg}
\oplus} {C^{X_s}_{KS}(X,D)_0}/m \ \ \mbox{and} \ \ C^{S'_\reg}_{KS}(X,D)_0 =
{\underset{s \in (S'_\reg)_{(0)}} \oplus}
C^{X_s}_{KS}(X,D)_0.
\end{equation}

In view of ~\eqref{eqn:Deg-0-*}, 
an important consequence of \propref{prop:Surface-tor} is the following.

\begin{cor}\label{cor:V-seq-04}
For all $m \gg 0$, the exact sequence ~\eqref{eqn:V-seq} induces a 3-term 
complex
\[
{V(X_t)} \xrightarrow{\partial} A_m \bigoplus C^{S'_\reg}_{KS}(X,D)_0 
\to C_{KS}(X,D)_{0} \to 0
\]
in which the middle arrow is surjective.
\end{cor}

For a morphism of schemes $T \to S$, we let $X_T = X \times_S T$.
We now let $s \in S \setminus S'_\reg$ be a closed point. We let $k(t)_s$
denote the total quotient ring of $\sO^h_{S,s}$. 
We let $Z = \Spec(\sO^h_{S,s})$. Note that 
the normalization $Z_n$ of $Z$ is canonically isomorphic to 
$\sO^h_{S_n,s}$, where the latter is the Henselization of
$\sO_{S_n,s}$ with respect to its Jacobson radical. 
In particular, if $\nu \colon S_n \to S$ denotes the normalization map, 
then $\sO^h_{S_n,s}$ is the product of the Henselian
discrete valuation rings $\sO^h_{S_n, s_i}$ for $1 \le i \le r$,
where $\Sigma = \{s_1, \ldots , s_r\} = \nu^{-1}(s)$.
Let $k(t)_i$ denote the quotient field of $\sO^h_{S_n,s_i}$
so that $k(t)_s = \stackrel{r}{\underset{i = 1}\prod} k(t)_i$.
Note that $r = 1$ if $s \in S_\reg$.

By \cite[Proposition~4.2]{Kato-Saito-2}, there exists
a nowhere dense closed subscheme $D' \subset X_{S_n}$ containing
$D \times_S S_n$ and a norm map of sheaves
$\nu_*(\sK^M_{2, (X_{S_n}, D')}) \to \sK^M_{2, (X,D)}$.
Taking the cohomology with support, we get a
push-forward map $\nu_* \colon C^{X_\Sigma}_{KS}(X_{S_n}, D') \to
C^{X_s}_{KS}(X,D)$. Note that we have used here the fact that the push-forward
of Nisnevich sheaves under a finite map is an exact functor.
By excision, the above is same as the map
$\nu_* \colon C^{X_\Sigma}_{KS}(X_{Z_n}, D'_{Z_n}) \to C^{X_s}_{KS}(X_Z,D_Z)$.
We thus get a commutative diagram
\begin{equation}\label{eqn:Hensel-V}
\xymatrix@C.8pc{
&  {\underset{s' \in S_{(0)}}\bigoplus} C^{X_{s'}}_{KS}(X,D)  \ar@{->>}[r] &
C^{X_s}_{KS}(X,D) \ar[d]^-{\cong} \\
SK_1(X_t) \ar[r]^-{\alpha_s} \ar[dr]_-{\alpha_s} \ar[ur]^-{\partial} 
& SK_1(X_{k(t)_s}) \ar[r]^-{\partial} & C^{X_s}_{KS}(X_Z,D_Z)  \\
& \stackrel{r}{\underset{i = 1}\prod} SK_1(X_{k(t)_{i}}) \ar[u]^-{\cong}_-{\nu_*}
\ar[r]^-{\partial} & C^{X_\Sigma}_{KS}(X_{Z_n}, D'_{Z_n}) \ar[u]_-{\nu_*},}
\end{equation}
where the horizontal arrow on the top is the projection and
$\alpha_s$ is induced by the canonical inclusion $k(t) \subset k(t)_s$.
One gets a similar commutative diagram by reducing all the abelian
groups modulo any integer $m \ge 1$.

For a fixed integer $m \ge 1$, we let 
$SK_1(X_t, m, S')$ be the kernel of the canonical map
$SK_1(X_t) \xrightarrow{\alpha_s} {\underset{v}\prod} {SK_1(X_{k(t)_{v}})}/m$,
where $v$ runs through all closed points of $S_n$ lying over
$S \setminus S'_\reg$. 
We let $V(X_t, m, S') = SK_1(X_t, m, S') \cap V(X_t)$.
It follows from ~\eqref{eqn:Hensel-V} that
$SK_1(X_t, m, S')$ is annihilated by the
composite map
\[
SK_1(X_t) \xrightarrow{\partial} 
{\underset{s \in S_{(0)}}\bigoplus} {C^{X_s}_{KS}(X,D)}/m \surj
{\underset{s \in S_{(0)} \setminus S'_\reg}\bigoplus} {C^{X_s}_{KS}(X,D)}/m.
\]
Since ${C^{X_s}_{KS}(X,D)_0}/m \inj {C^{X_s}_{KS}(X,D)}/m$,
it follows therefore from ~\eqref{eqn:V-seq-1} that
$V(X_t, m, S')$ is annihilated by the composite map
\begin{equation}\label{eqn:BS-0}
V(X_t) \xrightarrow{\partial} {\underset{s \in S_{(0)}}\bigoplus} 
{C^{X_s}_{KS}(X,D)_0}/m \surj
{\underset{s \in S_{(0)} \setminus S'_\reg}\bigoplus} {C^{X_s}_{KS}(X,D)_0}/m.
\end{equation}

Using \corref{cor:V-seq-04}, ~\eqref{eqn:Deg-0-*} and 
\lemref{lem:KS-Chow}, we obtain the following
consequence of \propref{prop:Surface-tor}.

\begin{cor}\label{cor:Vanishing}
For all $m \gg 0$, the exact sequence ~\eqref{eqn:V-seq} restricts to 
a 3-term complex
\[
V(X_t, m, S') \xrightarrow{\partial} C^{S'_\reg}_{KS}(X,D)_0 \to C_{KS}(X,D)_0.
\]
If $f^{-1}(S'_\reg)$ is regular and $D \cap f^{-1}(S'_\reg) = \emptyset$, then
we have a complex
\begin{equation}\label{eqn:Vanishing-0}
V(X_t, m, S') \xrightarrow{\partial} {\underset{s \in (S'_\reg)_{(0)}}\bigoplus} 
\CH^F_0(X_s)_0 \to C_{KS}(X,D)_0.
\end{equation}
\end{cor}

\subsection{Local Kato-Saito idele class group}
\label{sec:KS-local}
Let $X$ be a Noetherian scheme of Krull dimension $d \ge 1$.
Let $\sP_r(X)$ denote the set of all Parshin chains of length $r$ on 
$X$. If $P \in \sP_r(X)$, we let $P_i = (p_0, \ldots , p_i)$ for
$0 \le i \le r$. 
Let $x \in X$ be a closed point and $X_x = \Spec(\sO^h_{X,x})$. 
We write $\sP^x_r(X)$ for the set of
Parshin chains of length $r$ on $X$ beginning with $x$ and
$\sQ^x(X)$ for the set of Q-chains on $X$ beginning with $x$.
Let $\sF$ be a Nisnevich sheaf on $X$.
For $P = (p_0, \ldots , p_r) \in \sP^x_r(X)$ and $q \ge 0$, 
we have the sequence of
boundary maps
\begin{equation}\label{eqn:KS-boundary}
H^q_{P}(X,\sF) \xrightarrow{\partial} H^{q+1}_{P_{r-1}}(X,\sF) 
\xrightarrow{\partial} \cdots \xrightarrow{\partial} 
H^{q+r-1}_{P_{1}}(X,\sF)   \xrightarrow{\partial} H^{q+r}_{x}(X,\sF).
\end{equation}
We let $\partial^x_P \colon H^q_{P}(X,\sF) \to H^{q+r}_{x}(X,\sF)$ denote the
composite map.

If $Q = (p_0, \ldots , p_{s-1}, p_{s+1}, \ldots , p_r) \in \sQ^x(X)$ is a 
$Q$-chain with 
break at $s$, we let $B'(Q)$ be the set of all $y \in X$ such that
$Q_y = (p_0, \ldots , p_{s-1}, y, p_{s+1}, \ldots , p_r) \in \sP^x_r(X)$.
For any $y \in B'(Q)$, we have the restriction (localization) map
$H^q_Q(X,\sF) \to H^q_{Q_y}(X,\sF)$. For $a \in H^q_Q(X,\sF)$,
we shall let $a_{Q_y}$ denote the image of $a$ under this restriction. 
The following is a straightforward extension of 
\cite[Lemma~1.6.3]{Kato-Saito-2}.

\begin{lem}\label{lem:KS-support}
We have the following.
\begin{enumerate}
\item
The map
\[
{\underset{P \in \sP^x_d(X)}\bigoplus} H^0_P(X, \sF) 
\xrightarrow{\sum_P \partial^x_P} H^d_x(X,\sF)
\]
is surjective.
\item
If $d \ge 2$ and $A$ is an abelian group, then 
the group $\Hom(H^d_x(X,\sF), A)$ is
canonically isomorphic to the group of all families
$(g_P)_{P \in \sP^x_d(X)}$ of homomorphisms
$g_P \colon  H^0_P(X, \sF) \to A$ satisfying the following reciprocity
law $(R)$.

$(R)$ For any $0 < s < d$, any $Q$-chain $Q \in \sQ^x(X)$ of length $d$ 
with break at $s$, and any $a \in   H^0_Q(X, \sF)$, the element
$g_{Q_y}(a_{Q_y}) = 0$ for almost all $y \in B'(Q)$ and
\[
{\underset{y \in B'(Q)}\sum} g_{Q_y}(a_{Q_y}) = 0.
\]
\item
For $d = 1$, there is an exact sequence
\[
0 \to H^0_\nis(X_x, \sF) \to H^0_\nis(X_x \setminus \{x\}, \sF) \to
H^1_x(X,\sF) \to 0.
\]
\end{enumerate}
\end{lem}
\begin{proof}
We let $U = X_x \setminus \{x\}$.
Then $U$ is a Noetherian scheme of Krull dimension $d-1$ and
the correspondence $\sP_{d-1}(U) \to \sP_d(X_x)$,
which takes $P$ to $P_x := (x, P)$, is a bijection.
Let $f \colon X_x \to X$ denote the canonical map and let
$f_* \colon \sP_d(X_x) \to \sP^x_{d}(X)$ denote the induced map.
For $P \in \sP_{d-1}(U)$, we write $f_*(P_x)$ also as $P_x$.

We now have a commutative diagram
\begin{equation}\label{eqn:KS-support-0}
\xymatrix@C.8pc{
{\underset{P \in \sP_{d-1}(U)}\bigoplus} H^0_P(U, \sF) 
\ar@{->>}[d]_-{c_U}  \ar[r]^-{\cong} & 
{\underset{P \in \sP_{d}(X_x)}\bigoplus} H^0_P(X_x, \sF)  
\ar[d]^-{\sum_P {\partial}^x_P} & 
{\underset{P \in \sP^x_{d}(X)}\bigoplus} H^0_P(X, \sF) \ar[l]_-{\cong}^-{f^*} 
\ar[d]^-{\sum_P {\partial}^x_P} \\
H^{d-1}_\nis(U, \sF) \ar@{->>}[r]^-{\partial} & H^{d}_x(X_x, \sF) &
H^{d}_x(X, \sF). \ar[l]_-{f^*}^-{\cong}}
\end{equation}

The arrow $c_U$ is surjective by \cite[Lemma~1.6.3]{Kato-Saito-2}.
The first assertion of the lemma therefore follows from a diagram chase.
The second assertion follows from the 
isomorphism $H^{d-1}_\nis(U, \sF) \xrightarrow{\cong} H^d_x(X, \sF)$
for $d \ge 2$ and by applying \cite[Lemma~1.6.3]{Kato-Saito-2} to 
$H^{d-1}_\nis(U, \sF)$. The last part is obvious.
\end{proof}

\begin{lem}\label{lem:Coh-supp}
Let $f \colon Y \inj X$ be a closed immersion of integral schemes.
Let $D \subset X$ be a nowhere dense closed subscheme.
Let $E \subset Y$ be a nowhere dense closed subscheme
such that the Kato-Saito recipe gives a push-forward map
$f_* \colon C_{KS}(Y,E) \to C_{KS}(X,D)$. Let $x \in Y \setminus D $
be a closed point. Let $Z \subset X$ be a closed subset such that
$x \in Z$ and $Z \subset X_\reg$.
Then the diagram
\begin{equation}\label{eqn:Coh-supp-0}
\xymatrix@C.8pc{
C^{x}_{KS}(Y,E) \ar[r] \ar[d]_-{f_*} & C_{KS}(Y,E) \ar[d]^-{f_*} \\
\CH^F_0(Z) \ar[r] & C_{KS}(X,D)}
\end{equation}
is commutative.
\end{lem}
\begin{proof}
Note that the composition of the bottom horizontal arrow and
left vertical arrow is the given by
\[
C^{x}_{KS}(Y,E) \to \CH^F_0(\Spec(k(x))) \cong K^M_0(k(x))
\cong C^x_{KS}(X,D) \to C^{Z}_{KS}(X,D) \to C_{KS}(X,D).
\]

Let $d = \dim(X)$ and $r = \dim(Y)$.
Let $P = (p_0, \ldots , p_r) \in \sP^x_r(Y)$ be a maximal Parshin chain.
By \lemref{lem:KS-support} (Part~1), it suffices to show that 
the square on the right in the diagram
\begin{equation}\label{eqn:Coh-supp-1}
\xymatrix@C.8pc{
K^M_r(k(P)) \ar[r]^-{\cong} \ar[d]_-{\cong} & 
H^0_P(Y, \sK^M_{r, (Y,E)}) \ar[r] 
\ar[d]^-{\partial^x_P} & C_{KS}(Y,E) \ar[d]^-{f_*} \\
H^{d-r}_P(X, \sK^M_{d, (X, D)}) \ar[r]^-{\partial^x_P} &
K^M_0(k(x)) \ar[r] & C_{KS}(X,D)}
\end{equation}
is commutative.
Since the square on the left is clearly commutative and
the left horizontal arrow on the top is an isomorphism,
all we need to show is that the big outer rectangle commutes.
But this follows from \cite[Proposition~2.9(i)]{Kato-Saito-2}.
\end{proof}

\subsection{Bloch's map}\label{sec:BS}
Let $X$ be a projective integral scheme of dimension $d \ge 2$ over 
a finite field $k$. 
Let $f \colon X \to S$ be a dominant projective morphism between integral 
schemes in $\Sch_k$ of relative dimension one.
Assume that there exists a dense open subscheme
$S' \subset S$ such that the induced morphism $f' \colon X' := f^{-1}(S') \to S'$ is
smooth over $S'$ and the fibers of $f'$ are integral. Assume also that $f'$ has a
section. Since $f'$ is the restriction of $f$ to the open subscheme $X'$,
we shall often write
it simply as $f$ in the sequel if no confusion arises.

Let $j \colon S' \inj S$ be the inclusion.
Let $T' \subset S'$ be an integral 
curve and let $Y' = X' \times_{S'} T'$.
Since $Y' \to T'$ is a smooth projective morphism to an integral curve
of relative dimension one whose all fibers are integral, it follows that
$Y'$ is integral.

Let $T$ (resp. $Y$) be the scheme theoretic closure of $T'$ (resp. $Y'$)
in $S$ (resp. $X$).
Then $T$ and $Y$ are integral closed subschemes of $S$ and $X$,
respectively.
Let $\gamma \colon T \inj S$ and $\gamma' \colon Y \to X$ be the
inclusion maps. Let $t \in T$ denote the
generic point of $T$. For any point $s \in S$, we let
$X_s = f^{-1}(s)$ be the scheme theoretic fiber.

The long cohomology sequence with support gives us the boundary map
$\partial \colon SK_1(Y_t) \to {\underset{s \in T_{(0)}}\bigoplus} 
H^2_{X_s}(Y, \sK^M_{2,Y})$.
For every $s \in T_{(0)}$, \lemref{lem:KS-Chow} says that there is a 
canonical map $\gamma'_* \colon 
H^2_{X_s}(Y, \sK^M_{2,Y}) \to \CH^F_0(X_s)$. Let $\partial_s \colon SK_1(Y_t)
\to \CH^F_0(X_s)$ denote the composition
\begin{equation}\label{eqn:Commute-main-1}
SK_1(Y_t) \xrightarrow{\partial} 
{\underset{u \in T_{(0)}}\bigoplus} H^2_{X_u}(Y, \sK^M_{2,Y}) \surj
H^2_{X_s}(Y, \sK^M_{2,Y}) \xrightarrow{\gamma'_*} \CH^F_0(X_s),
\end{equation}
where the middle arrow is the projection.
By ~\eqref{eqn:V-seq-1}, the map $\partial_s$ restricts to
$\partial_s \colon V(Y_t) \to H^2_{X_s}(Y, \sK_{2,Y})_0 \to \CH^F_0(X_s)_0$.
Summing these maps over $s \in T_{(0)}$ and composing with the
inclusion of the direct sums of Chow groups via $T_{(0)} \inj S_{(0)}$, 
we get a boundary map
$\partial_{X,t} \colon V(Y_t) \to {\underset{s \in S_{(0)}}\bigoplus} 
\CH^F_0(X_s)_0$.
Moreover, there is a commutative diagram

\begin{equation}\label{eqn:Bloch-map-0}
\xymatrix@C.8pc{
V(Y_t) \ar[d]_-{j^*}^-{\cong} \ar[r]^-{\partial_{X,t}} & 
{\underset{s \in S_{(0)}}\bigoplus} \CH^F_0(X_s)_0 \ar@{->>}[d]^-{j^*} \\
V(Y_t) \ar[r]^-{\partial_{X',t}} & {\underset{s \in S'_{(0)}}\bigoplus} 
\CH^F_0(X_s)_0.}
\end{equation}

We let $V(Y_t, m, T')$ be as in \corref{cor:Vanishing} and write it
as $V(Y_t, m, S')$. Composing the lower horizontal arrow in
~\eqref{eqn:Bloch-map-0} with the inclusion $V(Y_t, m, S') \inj V(Y_t)$, we get
the boundary map

\begin{equation}\label{eqn:Bloch-map-1}
\partial_{X',t} \colon V(Y_t, m, S') \to {\underset{s \in S'_{(0)}}\bigoplus} 
\CH^F_0(X_s)_0.
\end{equation}
This boundary map is same as the one used in
\cite[Theorem~5.4]{Kato-Saito-2}
and was first considered by Bloch \cite[Theorem~4.2]{Bloch81} in a special case.

\vskip .2cm

For $s \in S'_{\reg}$, \lemref{lem:KS-Chow} yields a well defined map 
$\CH_0^F(X_s) \to C_{KS}(X, D)_{0}$.
As a consequence of the proof of \corref{cor:Vanishing}, we get the following.

\begin{thm}\label{thm:Vanishing-0-*}
Assume that $S'$ is regular and $D \cap X' = \emptyset$.
Then the composition
\[
V(Y_t, m, S') \xrightarrow{\partial_{X',t}} 
{\underset{s \in S'_{(0)}}\bigoplus} \CH^F_0(X_s)_0
\to C_{KS}(X,D)_0
\]
is zero for all $m \gg 0$.
\end{thm}
\begin{proof}
By \cite[Proposition~2.9, Lemma~2.10]{Kato-Saito-2}, we can find a
closed subscheme $E \subset Y$ containing $\gamma'^*(D)$
such that $E \cap X' = \emptyset$
and the Kato-Saito recipe provides a push-forward map
$\gamma'_* \colon C_{KS}(Y,E) \to C_{KS}(X,D)$. 
Lemma~\ref{lem:Deg-PF} says that this map is degree preserving.

Let $s \in T'_\reg$ be a closed point. Our first claim is that the
diagram
\begin{equation}\label{eqn:Commute-main}
\xymatrix@C.8pc{
C^{X_s}_{KS}(Y,E) \ar[rr] \ar[d]_-{\gamma'_*} & & 
C_{KS}(Y,E) \ar[d]^-{\gamma'_*} \\
\CH^F_0(X_s) \ar[r]^-{\cong} & C^{X_s}_{KS}(X,D) \ar[r] & C_{KS}(X,D)}
\end{equation}
commutes.
Since $X_s \subset Y_\reg$ and $X_s \cap E = \emptyset$, 
it follows from \lemref{lem:KS-Chow} that 
the left vertical arrow is an isomorphism. Hence, the claim follows
from \lemref{lem:Coh-supp}.

The diagram
\begin{equation}\label{eqn:Commute-main-2}
\xymatrix@C.8pc{
V(Y_t, m, T') \ar[r]^-{\partial} \ar[d]_-{\cong} & 
{\underset{s \in T'_{(0)}}\bigoplus} \CH^F_0(X_s)_0  \ar[d]^-{\gamma'_*} \\
V(Y_t, m, S') \ar[r]^-{\partial_{X',t}} & 
{\underset{s \in S'_{(0)}}\bigoplus} \CH^F_0(X_s)_0}
\end{equation}
commutes by the construction of the map $\partial_{X',t}$ above
(see ~\eqref{eqn:Commute-main-1}).

Our second claim is that the bottom boundary map in ~\eqref{eqn:Commute-main-2}
becomes trivial after composing with the projection
${\underset{s \in S'_{(0)}}\bigoplus} \CH^F_0(X_s)_0 \surj
{\underset{s \in S'_{(0)} \setminus T'_\reg}\bigoplus} \CH^F_0(X_s)_0$ for
all $m \gg 0$.

To prove the claim, we first note that the composite map is anyway
zero if $s \notin T'$. For points in $T'$, we project the horizontal arrows in 
~\eqref{eqn:Commute-main-2} to closed points in the finite set $T'_\sing$, 
which gives us a commutative diagram
\begin{equation}\label{eqn:Commute-main-3}
\xymatrix@C.8pc{
V(Y_t, m, T') \ar[r]^-{\partial} \ar[d]_-{\cong} & 
{\underset{s \in T'_\sing}\bigoplus} {\CH^F_0(X_s)_0}/m  \ar[d]^-{\gamma'_*} \\
V(Y_t, m, S') \ar[r]^-{\partial_{X',t}} & 
{\underset{s \in T'_\sing}\bigoplus} {\CH^F_0(X_s)_0}/m}
\end{equation}
for any $m \ge 1$.

It follows from ~\eqref{eqn:BS-0} that the top horizontal arrow in
~\eqref{eqn:Commute-main-3} is zero for all $m \gg 0$.
Hence, the same holds for the bottom horizontal arrow.
To finish the proof of the claim, it suffices therefore to show that
$\CH^F_0(X_s)_0$ is a finite group for every $s \in T'_\sing$. 
But this follows because $X_s$ is a smooth projective curve over $k(s)$ and 
hence $\CH^F_0(X_s)_0 \cong J(X_s)(k(s))$, and the latter group is finite. 

It follows from the claim that the bottom horizontal arrow in
~\eqref{eqn:Commute-main-2} factors through
$V(X_t, m, S') \xrightarrow{\partial_{X',t}} 
{\underset{s \in (T'_\reg)_{(0)}}\bigoplus} \CH^F_0(X_s)_0$ for all
$m \gg 0$. Hence, the theorem is equivalent to showing that 
the composition
\[
V(Y_t, m, S') \xrightarrow{\partial}  
{\underset{s \in  (T'_\reg)_{(0)}}\bigoplus} \CH^F_0(X_s)_0 \to 
C_{KS}(X,D)_0
\]
is zero for all $m \gg 0$.

For this, we consider the commutative diagram
\begin{equation}\label{eqn:Commute-main-4}
\xymatrix@C.8pc{
V(Y_t, m, T') \ar[r]^-{\partial} \ar[d]_-{\cong} & 
{\underset{s \in (T'_\reg)_{(0)}}\bigoplus} \CH^F_0(X_s)_0 \ar[r] 
\ar[d]^-{\gamma'_*} &
C_{KS}(Y,E)_0 \ar[d]^-{\gamma'_*}_-{\cong} \\
V(Y_t, m, S') \ar[r]^-{\partial} & 
{\underset{s \in  (T'_\reg)_{(0)}}\bigoplus} \CH^F_0(X_s)_0 \ar[r] & 
C_{KS}(X,D)_0}
\end{equation}
for all $m \gg 0$, where the right square commutes by our first claim
that ~\eqref{eqn:Commute-main} is commutative.
But the top composite arrow is zero for all $m \gg 0$ by
\corref{cor:Vanishing}. We conclude that the same holds for the
bottom composite arrow. This finishes the proof of the theorem.
\end{proof}

\section{The finiteness theorem}
\label{sec:Fin-KS}
Let $k$ be a finite field and $X$ an integral projective scheme of
dimension $d \ge 1$ over $k$. Let $D \subset X$ be a nowhere dense
closed subscheme. Let $K$ denote the function field of $X$.
Our goal in this section is to prove the finiteness of $C_{KS}(X,D)_0$.
This will be done in several steps beginning with the case of curves.

\subsection{The case of curves}\label{sec:D-1}
Suppose first that $X$ is a curve. In this case, we have an exact sequence
\[
H^0_\nis(D, \sO^{\times}_D) \to C_{KS}(X,D)_0 \to C_{KS}(X)_0 \to 0.
\]
The first term is clearly finite. So we need to see that the third term is 
finite. We do this by comparing with the normalization $f \colon X_n \to X$.
We have an exact sequence
\[
H^0_\nis(Y, \sO^{\times}_Y) \to C_{KS}(X)_0 \to C_{KS}(X_n)_0 \to 0
\]
by \lemref{lem:Deg-birational}, where $Y \subset X_n$ is a finite closed
subscheme. We are therefore reduced to proving that $C_{KS}(X)_0$ is finite
when $X$ is a smooth projective curve over a finite field. But this
is classical as $C_{KS}(X)_0 = J(k)$, where $J$ is the Jacobian variety of $X$.

\subsection{Two general reduction steps}\label{sec:Gen-redn}
The following are two reduction steps which apply without special condition 
on $X$.

\begin{lem}\label{lem:Normal-*}
Assume that $C_{KS}(X,D)_0$ is finite when $X$ is normal. Then
the same holds for arbitrary $X$.
\end{lem}
\begin{proof}
Let $f \colon X_n \to X$ be the normalization map.
Then, it follows from \cite[Proposition~4.2]{Kato-Saito-2} that there
is a nowhere dense closed subscheme $D' \subset X_n$ such that
$D'_\red \subset f^{-1}(X_\sing \cup D_\red)$ and
$f^*(D) \subset D'$. Moreover, the norm map between the Milnor $K$-theory of
the residue fields of the maximal Parshin chains defines a norm map
between the Nisnevich sheaves 
\begin{equation}\label{eqn:Normal-*-0}
N_{{X_n}/X} \colon f_*(\sK^M_{d, (X_n, D')}) \to \sK^M_{d, (X,D)}.
\end{equation}
Taking the cohomology, this defines a push-forward
map $f_* \colon C_{KS}(X_n,D') \to C_{KS}(X,D)$.

By construction, the above map has the property that for a regular closed point
$x \in X_n \setminus D'$ such that $f(x) \in X_\reg$, the diagram
\begin{equation}\label{eqn:Surface-tor-0}
\xymatrix@C1pc{
\Z \ar[r]^-{\cong} & C^x_{KS}(X_n, D') \ar[r] \ar[d]^-{f_*}_-{\cong} &
C_{KS}(X_n, D') \ar[d]^-{f_*} \\
& C^{f(x)}_{KS}(X, D) \ar[r] & C_{KS}(X, D)}
\end{equation}
is commutative.
% (see \cite[\S~4.4]{Kato-Saito-2}). 

We conclude from \cite[Theorem~2.5]{Kato-Saito-2} that 
$f_* \colon C_{KS}(X_n, D') \to C_{KS}(X, D)$ is surjective.
In particular, it is surjective on the degree zero subgroups by
\lemref{lem:Deg-PF}.
\end{proof}

\begin{remk}\label{remk:Normal-open}
  The proof of \lemref{lem:Normal-*} yields more than what its statement asserts.
  Namely, the map $f_* \colon C_{KS}(X_n, D') \to C_{KS}(X, D)$ is surjective
  even if $X$ is not projective. This surjectivity will be used in the
  proof of \corref{cor:KS-fin-aff}.
  \end{remk}

\begin{lem}\label{lem:divisor}
Assume that $C_{KS}(X,D)_0$ is finite when $D_\red$ is an effective Weil
divisor whose complement is regular. Then the same holds for arbitrary $D$.
\end{lem}
\begin{proof}
Given any irreducible component of $D_\red \cup X_\sing$, 
we can find an irreducible
prime divisor of $X$ containing this component. This implies that
there is a reduced Weil divisor $E$ (with reduced
induced closed subscheme structure) containing $D_\red \cup X_\sing$.
But this implies that $D \subset mE$ for all $m \gg 0$.
Since $E$ is nowhere dense, the canonical map
$C_{KS}(X,mE) \to C_{KS}(X,D)$ is surjective for all $m \gg 0$.
Hence, it is surjective on the degree zero subgroups.
We are therefore done.
\end{proof}

\subsection{The case of generic fibration}\label{sec:Fib}
Let $X$ and $D$ be as above with $d \ge 2$.
Let $S$ be an integral projective scheme of dimension $d-1$ over $k$ 
and $f \colon  X \to S$ a dominant morphism with the following
properties:
\begin{enumerate}
\item
$f_*(\sO_X) \cong \sO_S$;
%$f$ has geometrically connected fibers.
\item
There is a regular dense open $S' \subset S$ such that the restriction
$f \colon f^{-1}(S') \to S'$ is smooth of relative dimension one;
\item
The fibers of $f \colon f^{-1}(S') \to S'$ are integral;
\item
The map $f \colon f^{-1}(S') \to S'$ has a section $\iota: S' \inj f^{-1}(S')$;
\item
$D \cap f^{-1}(S') = \emptyset$.
\end{enumerate}

\begin{defn}\label{defn:Fibration}
A morphism $f \colon X \to S$ which satisfies the above properties
will be called a generic fibration of relative dimension one.
\end{defn}

We shall now prove the finiteness theorem
when $X$ admits a generic fibration of relative dimension one
over a $(d-1)$-dimensional integral
scheme. We shall prove this using \thmref{thm:Vanishing-0-*},
an exact sequence of Bloch \cite[Theorem~4.2]{Bloch81} and its
generalization by Kato-Saito \cite[Theorem~5.4]{Kato-Saito-2}.

\begin{lem}\label{lem:KS-gen-fin}
Suppose that there exists a generic fibration $f \colon X \to S$
of relative dimension one. Then $C_{KS}(X,D)_0$ is finite.
\end{lem}
\begin{proof}
We shall follow the notations of the definition of generic fibration of
relative dimension one.
We let $\eta \in S$ be the generic point. 
We let $\wt{S} \subset X$ be the scheme theoretic
closure of $\iota(S')$ in $X$. Note that $\wt{S}$ is an
integral scheme such that $f|_{\wt{S}} \colon \wt{S} \to S$
is projective and birational. We let $X' = f^{-1}(S')$ and
write $\wt{S} \cap X'$ as $S'$. We let $\iota \colon \wt{S} \inj X$
denote the inclusion. 
By \cite[Proposition~2.9, Lemma~2.10]{Kato-Saito-2},
we can find a closed subscheme $\wt{D} \subset \wt{S}$ containing
$\iota^*(D)$ such that $\wt{D}_\red = (\iota^*(D))_\red$ and
the Kato-Saito recipe gives a push-forward map
$\iota_* \colon  C_{KS}(\wt{S}, \wt{D}) \to C_{KS}(X,D)$.

We now consider the commutative diagram
\begin{equation}\label{eqn:KS-gen-fin-0}
\xymatrix@C.8pc{
{\underset{x \in S'_{(0)}}\bigoplus} \Z \ar@{^{(}->}[d]  \ar@{=}[r] &
{\underset{x \in S'_{(0)}}\bigoplus} \Z \ar@{^{(}->}[d]^-{\iota_*}  
\ar@{->>}[r] &
C_{KS}(\wt{S}, \wt{D}) \ar[d]^-{\iota_*} \\
{\underset{x \in X'_{(0)}}\bigoplus} \Z \ar[r] &
{\underset{x \in S'_{(0)}}\bigoplus} \CH^F_0(X_s) \ar@{->>}[r] \ar[d] & 
C_{KS}(X,D) \ar[d] \\
& M \ar@{->>}[r] \ar[d] & N \ar[d] \\
& 0 & 0,}
\end{equation}
where $M$ and $N$ are defined so that the middle and the right columns are
exact.
The horizontal arrows on the right side are given by
\lemref{lem:KS-Chow} and are surjective by
\cite[Theorem~2.5]{Kato-Saito-2}.
Note also that since $\CH^F_0(X_s)_0 = 
\Ker(\CH^F_0(X_s)_0 \xrightarrow{f_*} \CH^F_0(\Spec(k(s)))$, it follows that
$M = {\underset{x \in S'_0}\bigoplus} M_s$ is
such that the composition
\[
\CH^F_0(X_s)_0 \inj \CH^F_0(X_s) \to M_s
\]
is an isomorphism for all $s \in S'_{(0)}$.
Hence, we can identify $M$ with 
${\underset{x \in S'_0}\bigoplus} \CH^F_0(X_s)_0$.

Next, it follows from \lemref{lem:Deg-PF} that the right column in
~\eqref{eqn:KS-gen-fin-0} induces an exact sequence
\begin{equation}\label{eqn:KS-gen-fin-1}
C_{KS}(\wt{S}, \wt{D})_0 \xrightarrow{\iota_*} C_{KS}(X,D)_0 \to N
\end{equation}
By induction on $\dim(X)$, our task therefore remains to show that $N$ is 
finite.
Using the above description of $M$ and the surjection $M \surj N$, 
it suffices to show that the composite map
\begin{equation}\label{eqn:KS-gen-fin-2}
{\underset{x \in S'_{(0)}}\bigoplus} \CH^F_0(X_s)_0
\inj {\underset{x \in S'_{(0)}}\bigoplus} \CH^F_0(X_s) \to  C_{KS}(X,D) \to N
\end{equation}
factors through a finite quotient.

We now choose a point $t \in S'_{(1)}$ and let $T \subset S$ be its 
scheme theoretic closure. Then $T \subset S$ is an integral curve
whose generic point lies in $S'$. We let $\gamma \colon T \inj S$ denote the
inclusion map and let $T' = S' \cap T$. 
We let $Y' = X' \times_S T'$. Then the restriction $Y' \to T'$ of $f$
is a smooth projective morphism with integral fibers of dimension one.  
This implies that $Y'$ must be integral. We let $Y \subset X$ be the
scheme theoretic closure of $Y'$ in $X$. 
This gives us a commutative square of integral projective schemes
\begin{equation}\label{eqn:KS-gen-fin-4}
\xymatrix@C.8pc{
Y \ar[r]^-{\gamma'} \ar[d]_-{g} & X \ar[d]^-{f} \\
T \ar[r]^-{\gamma} & S.}
\end{equation}

It follows from \thmref{thm:Vanishing-0-*} that the composite map
\begin{equation}\label{eqn:V-seq-2-8}
V(Y_t, m_t, S') \xrightarrow{\partial_{X',t}} 
{\underset{s \in S'_{(0)}}\bigoplus} \CH^F_0(X_s)_0
\to C_{KS}(X,D)_0
\end{equation}
is zero for all $m_t \gg 0$. We write
$V(Y_t, m_t, S')$ as $V(X_t, m_t, S')$ (note that $Y_t = X_t$ for any
$t \in S'_{(1)}$) and  let 
\[
{\partial}_{X'} \colon {\underset{t \in S'_{(1)}}\bigoplus} V(X_t, m_t, S')
\to  {\underset{x \in S'_{(0)}}\bigoplus} {\CH^F_0(X_s)_0}
\]
denote the sum of the boundary maps $\partial_{X',t}$ over the points
$t \in S'_{(1)}$.

We now let $\pi^{\ab}_1(X_\eta)_0 := \Ker(\pi^{\ab}_1(X_\eta) \xrightarrow{f_*} 
\pi^{\ab}_1(\Spec(k(\eta))))$. It follows from a theorem of
Katz and Lang \cite{Katz-Lang} that $\pi^{\ab}_1(X_\eta)_0$ is finite.
By \cite[Theorem~5.4]{Kato-Saito-2} (the exact sequence of
Bloch and Kato-Saito), there exists an exact sequence
\begin{equation}\label{eqn:KS-gen-fin-3}
{\underset{t \in S'_{(1)}}\bigoplus} V(X_t, m_t, S') 
\xrightarrow{{\partial}_{X'}} {\underset{x \in S'_{(0)}}\bigoplus} 
{\CH^F_0(X_s)_0} \to \pi^{\ab}_1(X_\eta)_0 \to 0,
\end{equation}
where for every $t \in S'_{(1)}$, we can take $m_t$ to be any nonzero 
integer which annihilates $\pi^{\ab}_1(X_\eta)_0$. 

We saw above in ~\eqref{eqn:V-seq-2-8} and in \S~\ref{sec:3-term}
that if for $t \in  S'_{(1)}$,
we choose $m_t$ large enough such that
$m_t$ annihilates the resulting $C_{KS}(Y, E)_0$, 
${\underset{t \in T'_\sing}\bigoplus} \CH^F_0(X_s)_0$ and 
$\pi^{\ab}_1(X_\eta)_0$, then the forget support map
\[
{\underset{x \in S'_{(0)}}\bigoplus} {\CH^F_0(X_s)_0} \to
C_{KS}(X,D)
\]
annihilates the image of ${\partial}_{X'}$.
It follows that the composite map in ~\eqref{eqn:KS-gen-fin-2}
factors through the finite group $\pi^{\ab}_1(X_\eta)_0$.
We have therefore proven the lemma.
\end{proof}

\subsection{The case of schemes fibered by curves}
\label{sec:Fiber**}
We shall say that a general fiber of a morphism between schemes
has property $\sP$ if all (scheme theoretic) 
fibers over a dense open subset of the base have property $\sP$. 
Let $f \colon X \to S$ be a dominant projective morphism between integral
schemes over $k$ whose generic fiber has dimension one. Let $F$ be the
function field of $S$. We shall say that $f$ is `nice' if the following
hold:
\begin{enumerate}
\item
$X$ and $S$ are normal;
\item
$f$ is generically smooth; 
\item
$f_* \sO_{X} = \sO_{S}$;
\item
All fibers of $f$ are geometrically connected;
\item
General fibers of $f$ are geometrically integral of dimension one.
\end{enumerate}

For a finite field extension $F \inj F'$, let
$S'$ denote the normalization of $S$ in $F'$ and let $X'$ denote the
normalization of $X$ in the compositum $KF'$ inside an algebraic closure of
$K$.
Then $f$ induces a projective morphism $f' \colon X' \to S'$ such that
the diagram
\begin{equation}\label{eqn:Stein-fact-0}
\xymatrix@C.8pc{
X' \ar[r]^-{\phi}  \ar[d]_-{f'} & X \ar[d]^-{f} \\
S' \ar[r]^-{\psi} & S}
\end{equation}
is commutative. The horizontal arrows are finite dominant and vertical
arrows are projective dominant morphisms between integral schemes.

\begin{lem}\label{lem:Nice}
If $f \colon X \to S$ is nice, then so is $f' \colon X' \to S'$.
\end{lem}
\begin{proof}
The property (1) follows by construction. We let $Y = X_F$.
Then $Y$ is a smooth projective geometrically integral curve over $F$.
This implies that $Y_{F'}$ also has the same property.
In particular, $Y_{F'}$ coincides with the normalization of $Y$ in
the composite field $KF'$. Equivalently, $Y_{F'}$ is the generic fiber
of the map $f' \colon X' \to S'$. This proves (2).

Since we showed above that 
the generic fiber $X'_{F'}$ of $f'$ is geometrically reduced and
connected, it follows that $H^0(X'_{F'}, \sO_{X'_{F'}}) = F'$ (see  
\cite[Lemma~33.9.3]{SP}). 
%Varieties
An easy application of Stein factorization
tells us that $f_* \sO_{X'} = \sO_{S'}$ (e.g., see 
\cite[Lemma~37.49.6]{SP}). 
%More on morphisms
In particular, all fibers of
$f'$ are geometrically connected. This proves (3) and (4).
The first part of (5) now easily follows from
\cite[Lemmas~37.24.4, 37.25.5]{SP}. The second part is an 
easy consequence of the fact that $f$ is dominant and hence flat over
a dense open subset of $S$ by the generic flatness.
\end{proof}

\begin{lem}\label{lem:Stein-fact}
We can choose $F'$ to be a purely inseparable
extension such that the resulting map $f' \colon X' \to S'$ is nice.
\end{lem}
\begin{proof}
We let $S_n \to S$ be the normalization morphism and let $Y$
denote the normalization of $X$. Then $Y$ is a normal
integral scheme and $f$ induces a projective dominant map 
$f_1 \colon Y \to S_n$.
Since the generic fiber of $f_1$ is a projective limit
of open subschemes of $Y$ and the latter is normal, it follows that
$Y_F$ is integral and normal. Since $\dim(Y_F) = 1$, it follows that
$Y_F$ is an integral regular curve over $F$.

By \cite[Lemma~33.27.3]{SP}, there exists a finite purely inseparable
field extension $F \inj F'$ such that the normalization $Z$ of 
$(Y_F \otimes_F F')_{\red}$ is geometrically normal. Since
$\dim(Y_F) = 1$, it follows that $Z$ is geometrically regular.
In particular, it is geometrically reduced.
Since $Y_F \otimes_F F' \to Y_F$ is a base change of the
radicial morphism $\Spec(F') \to \Spec(F)$, it is also radicial.
It follows that $Y_F \otimes_F F'$ is irreducible.
We conclude that $Z$ is an integral and geometrically regular 
(equivalently, smooth) projective curve over $F'$.

We let $S'$ denote the normalization of $S$ in $F'$ and let 
$X'$ be the normalization of $Y$ in the composite field $KF'$.
Let $f' \colon X' \to S'$ be the induced morphism.
Since $Y_F$ is a projective limit
of open subschemes of $Y$, it follows that the generic fiber
of $f'$ is the normalization $Z'$ of $Y_F$ in $KF'$.
On the other hand, it is easily seen that the canonical map
$Z \to Z'$ is a birational morphism between normal projective curves
over $F'$ (with function field $KF'$), and hence is an isomorphism.

Since $X'$ is same as the normalization of $X$ in the composite field
$KF'$, we conclude that there exists 
a commutative square such as ~\eqref{eqn:Stein-fact-0} for which the
generic fiber of $f'$ is a smooth projective curve integral over the 
function field $F'$ of $S'$. 
We have thus shown (1) and (2).
The proof of (3), (4) and (5) is identical to the proof
of the second part of \lemref{lem:Nice}.
\end{proof}

\begin{lem}\label{lem:Fiber-KS-fin}
Let $f \colon X \to S$ be a dominant projective morphism between integral
schemes over $k$ whose generic fiber has dimension one. Then
$C_{KS}(X,D)_0$ is finite.
\end{lem}
\begin{proof}
Let $F$ denote the function field of $S$.
By \lemref{lem:Stein-fact}, there is a purely inseparable finite extension
$F \subset F'$ such that the resulting map $f' \colon X' \to S'$ is nice.
By \cite[Proposition~4.2]{Kato-Saito-2}, there is a nowhere dense
closed subscheme $D' \subset X'$ containing $\phi^*(D)$ (see
~\eqref{eqn:Stein-fact-0}) and a push-forward map
$\phi_* \colon C_{KS}(X', D') \to C_{KS}(X,D)$. This map is
degree preserving by \lemref{lem:Deg-PF}.
It follows from \cite[Corollary~4.10]{Kato-Saito-2} that 
$\phi_* \colon C_{KS}(X', D') \to C_{KS}(X,D)$ is surjective. 
Hence, it is surjective on the degree zero subgroups. It suffices therefore to 
prove the lemma when $f \colon X \to S$ is nice. 
We shall assume this to be the case in the rest of the proof
and divide the proof into various sub-cases. We let $Y$ denote the
generic fiber of $f$.

We first consider the 
case when $D \cap Y = \emptyset$ and $Y(F) \neq \emptyset$. 
Since $f$ is projective, it follows that the scheme theoretic image of
$D$ in $S$ is a nowhere dense closed subscheme. We can therefore 
find an open dense subscheme $S' \subset S$ such that $D \cap f^{-1}(S')
= \emptyset$. Since
$Y$ is smooth over $F$ and $Y(F) \neq \emptyset$, we can assume 
(after shrinking $S'$ if necessary) that $f$ is smooth over $S'$ and
has a section over $S'$. It follows that $f \colon X \to S$ is a 
generic fibration of relative dimension one.
We conclude the finiteness of $C_{KS}(X,D)_0$ by
\lemref{lem:KS-gen-fin}.

We now consider the case when $D \cap Y = \emptyset$.
We can find a finite Galois extension $F \subset F'$ such that
the resulting map $f' \colon X' \to S'$ has the property that 
$Y'(F') \neq \emptyset$, where $Y'$ is the generic fiber of $f'$.
It follows from \lemref{lem:Nice} that $f'$ is also nice.
Since $X$ is normal, there exists a push-forward map
$\phi_* \colon C_{KS}(X',D') \to C_{KS}(X,D)$ where we can take
$D' = \phi^*(D)$. It follows from \cite[Corollary~4.10]{Kato-Saito-2}
that $\coker(\phi_*)$ is finite. Since $\phi_*$ is degree preserving
by \lemref{lem:Deg-PF}, it follows that 
\begin{equation}\label{eqn:Fiber-KS-fin-**}
C_{KS}(X',D')_0 \xrightarrow{\phi_*} C_{KS}(X,D)_0 \to \coker(\phi_*)
\end{equation}
is exact. Since we have shown above that $C_{KS}(X',D')_0$ is finite,
it follows that so is $C_{KS}(X,D)_0$.

We now consider the remaining case $D \cap Y \neq \emptyset$.
We let $T = \{x \in Y^{(1)}| \ov{\{x\}} \subset D\}$, where
the closure of $\{x\}$ is taken in $X$. 
It is clear that $T$ is a finite set. We can now choose a
closed subscheme $D' \subset D$ such that
\begin{equation}\label{eqn:Fiber-KS-fin-0}
\sI_{D'}\sO_{X,x} = \sI_D \sO_{X,x} \ \ \mbox{for} \ \ 
x \in X^{(1)} \setminus T \ \ \mbox{and} \ \ D' \cap Y = \emptyset.
\end{equation}
It follows from the choice of $D'$ that $\sI_{D'} \sO_{X,x} = \sO_{X,x}$ for
every $x \in T$. By 
\cite[Theorem~8.3, Proposition~8.4, Corollary~8.5]{Kato-Saito-2}, there
exists an exact sequence 
\begin{equation}\label{eqn:Fiber-KS-fin-1}
{\underset{x \in T}\bigoplus} C(x, i_x) \to C_{KS}(X,D) \to C_{KS}(X,D') \to 0,
\end{equation}
where $i_x \gg 0$ is an integer and $C(x, i_x)$ is finite for each $x \in T$.
We remark here that this is shown in \cite[Proposition~8.4]{Kato-Saito-2}
when $T$ is a singleton. However, an easy induction on the cardinality
of $T$ yields the general case. 

Taking the degree zero parts, we get an exact sequence
\begin{equation}\label{eqn:Fiber-KS-fin-2}
{\underset{x \in T}\bigoplus} C(x, i_x) \to C_{KS}(X,D)_0 \to C_{KS}(X,D')_0 \to 
0.
\end{equation}
We have shown earlier that $C_{KS}(X,D')_0$ is finite and this concludes
the proof.
\end{proof}

\subsection{The general case of finiteness theorem}
\label{sec:KS-gen}
Let $k$ be a finite field.
Let $X$ be a projective integral scheme in $\Sch_k$ of dimension $d \ge 1$
and $D \subset X$ a nowhere dense closed subscheme.
Let $K$ be the function field of $X$.
We shall now prove the finiteness of $C_{KS}(X,D)_0$ in general.
The result is the following.

\begin{thm}\label{thm:KS-fin} 
The degree zero idele class group $C_{KS}(X,D)_0$ is finite.
In particular, $C(X,D)_0$ is finite if $X \setminus D$ is regular.
\end{thm}
\begin{proof}
We only need to prove the  finiteness of $C_{KS}(X,D)_0$ in view
of \cite[Theorem~3.8]{Gupta-Krishna-REC} (see ~\eqref{eqn:ICG-maps})
and \lemref{lem:KS-K-deg}.

First of all, we can assume that $X$ is normal by \lemref{lem:Normal-*}.
Next, by enlarging $D$ if necessary, we can assume that its 
complement is regular by \lemref{lem:divisor}. We let $Y = D_\red$ with reduced
closed subscheme structure and let $V = X \setminus Y$.
Note that $X_\sing \subset Y$. 
We can now find a finite field extension $k \inj k'$ such that
if $X'$ is the normalization of $X$ in the composite field $Kk'$
and if $Y'$ is the inverse image of $Y$ under the projection
$\phi \colon X' \to X$, then there is a blow-up $g \colon X'' \to X'$
with center away from $Y'$ and a dominant projective
morphism $f \colon X'' \to \P^{d-1}_{k'}$.
This is proven in \cite[Lemma~9.5]{Kato-Saito-2} if $k$ is infinite, where
one can take $k' = k$. The finite field case is easily reduced to the
infinite field because we can get $g \colon X'' \to X'$ and $f \colon X'' \to
\P^{d-1}_{E}$, where $E$ is a pro-$\ell$ extension of $k$ for some prime
$\ell$. Any such construction is then actually defined over a finite
extension $k'$ of $k$.

Note here that the push-forward $\phi_* \colon C_{KS}(X',D') \to C_{KS}(X,D)$
of ~\eqref{eqn:Fiber-KS-fin-**} exists if we let $D' = f^*(D)$
because $X$ is normal (see \cite[Proposition~4.2]{Kato-Saito-2}). By \cite[Corollary~4.10]{Kato-Saito-2}, 
the cokernel of $\phi_*$ is finite, so it suffices to prove the theorem for $X'$.
We can therefore assume that there is a blow-up $g \colon \wt{X} \to X$
with center away from $Y$ and a dominant projective
morphism $f \colon \wt{X} \to \P^{d-1}_{k}$. We let $U \subset X$ be
an open dense subscheme over which $g$ is an isomorphism.
Then $Y \subset U$.

By \cite[Lemma~9.6]{Kato-Saito-2}, the Kato-Saito recipe gives a
push-forward map $g_* \colon C_{KS}(\wt{X}, \wt{D}) \to C_{KS}(X,D)$
for any nowhere dense closed subscheme $\wt{D} \subset \wt{X}$
containing $g^*(D)$. 
The map $g_*$ has the property that for any closed point 
$x \in g^{-1}(U_\reg) \setminus \wt{D}$, the diagram
\begin{equation}\label{eqn:KS-fin-00}
\xymatrix@C.8pc{
\Z \ar[r]^-{\cong} \ar[dr]_-{\cong} & 
C^x_{KS}(\wt{X}, \wt{D}) \ar[r] \ar[d]_-{g_*}^-{\cong} & C_{KS}(\wt{X}, \wt{D})
\ar[d]^-{g_*} \\
& C^x_{KS}({X}, {D}) \ar[r] & C_{KS}({X}, {D})}
\end{equation}
commutes where the horizontal arrows are the forget support maps.

It follows from this diagram and
\cite[Theorem~2.5]{Kato-Saito-2} that $g_*$ is surjective.
Since $g_*$ is degree preserving by \lemref{lem:Deg-PF},
this implies that $g_* \colon C_{KS}(\wt{X}, \wt{D})_0 \to C_{KS}({X}, {D})_0$
is also surjective. It suffices therefore to show that 
$C_{KS}(\wt{X}, \wt{D})_0$ is finite.
But this follows from \lemref{lem:Fiber-KS-fin} because $f$ is a dominant 
morphism between integral schemes of relative dimension one and
hence its generic fiber has dimension one. This concludes the proof of
the theorem.
\end{proof}

As a consequence of \thmref{thm:KS-fin}, we obtain the following new finiteness theorem
for the Kato-Saito idele class group of non-projective schemes.

\begin{cor}\label{cor:KS-fin-aff}
Let $X$ be an integral quasi-projective scheme of dimension $d \ge 1$ over a finite field and
$D \subset X$ a nowhere dense closed subscheme. Assume that $X$ is not projective over $k$.
Then $C_{KS}(X,D)$ is finite.
\end{cor}
\begin{proof}
Thanks to Remark~\ref{remk:Normal-open}, we can assume that $X$ is normal.
  We can now find an open immersion $j \colon X \inj \ov{X}$ such that $\ov{X}$ is an
  integral and normal projective scheme. By blowing up the complement of $X$ (with the reduced
  closed subscheme structure) in $\ov{X}$ and
  normalizing again, we can furthermore assume that $\dim(\ov{X} \setminus X) = d-1$. We let
  $Y = \ov{X} \setminus X$ with reduced closed subscheme structure.

  We let $\ov{D} \subset \ov{X}$ be the scheme theoretic closure of $D$ in $\ov{X}$.
  We let $Z = \ov{D} \cup \ov{X}_\sing$.
  As $D$ is a dense open subscheme of $\ov{D}$, it follows that $\dim(\ov{D} \setminus D)
  = \dim(Y \cap \ov{D}) \le d-2$. Since $\ov{X}$ is normal, we must also have
  $\dim (Y \cap Z) \le d-2$. Since $\dim(Y) = d-1$, we find that $Y \setminus Z \neq \emptyset$.
  We choose a closed point $x \in Y \setminus Z$ and consider the 
  commutative diagram

  \begin{equation}\label{eqn:KS-fin-aff-0}
    \xymatrix@C.8pc{
      0 \ar[r] &
      C^Y_{KS}(\ov{X}, \ov{D})_0 \ar[r] \ar[d] & C^Y_{KS}(\ov{X}, \ov{D}) \ar[r]^-{\deg} \ar[d] &
      \Z \ar@{=}[d] \\
      0 \ar[r] &
      C_{KS}(\ov{X}, \ov{D})_0 \ar[r] \ar[d] & C_{KS}(\ov{X}, \ov{D}) \ar[r]^-{\deg}
      \ar@{->>}[d] & \Z \\
      & C_{KS}({X}, {D}) \ar@{=}[r] & C_{KS}({X}, {D}). & }
  \end{equation}
  
All rows as well as the middle column of this diagram are exact. The right vertical arrow on the bottom is
    surjective by \cite[Theorem~2.5]{Kato-Saito-2}.
It follows from \lemref{lem:KS-Chow} that $\deg (\lambda_x(1)) \neq 0$.    
In particular, the cokernel of composite map
$C^Y_{KS}(\ov{X}, \ov{D}) \to C_{KS}(\ov{X}, \ov{D}) \to \Z$
is finite. Since $C_{KS}(\ov{X}, \ov{D})_0$ is finite by \thmref{thm:KS-fin}, a diagram
chase shows that $ C_{KS}({X}, {D})$ must be finite.
\end{proof}

\section{The reciprocity theorem}\label{sec:Rec-Iso}
In this section, we shall prove our main reciprocity theorem. 
We begin by recalling the fundamental groups with modulus and the
reciprocity map.

\subsection{Recollection of fundamental groups with modulus}
\label{sec:Rec-*}
Let $K$ be a Henselian discrete valuation field. Let
$G^{(\bullet)}_K$ be the Abbes-Saito filtration of $G_K$
(see \cite{Abbes-Saito} or \cite[\S~6.1]{Gupta-Krishna-REC}). 
Let $L/K$ be a finite separable extension and $n \ge 0$ an integer.
Recall from \cite[\S~7.1]{Gupta-Krishna-REC} that the ramification of 
$L/K$ is said to be bounded by $n$ if $G^{(n)}_K$ is contained in 
$\Gal(\ov{K}/{L})$ under the inclusions 
$G^{(n)}_K \subset G_K \supset \Gal(\ov{K}/{L})$, where $\ov{K}$ denotes 
a fixed separable closure of $K$.

Let $k$ be a field and $X \in \Sch_k$ an integral normal scheme of
dimension $d \ge 1$. Let $D \subset X$ be an effective Weil divisor
and $C$ the support of $D$ with reduced closed subscheme structure.
Set $U = X \setminus C$. Let $K$ denote the function field of $X$.
For any generic point $\lambda$ of
$C$, let $K_\lambda$ denote the Henselization of $K$ along $\lambda$.
Recall from \cite[Definition~7.5]{Gupta-Krishna-REC} that
the co-1-skeleton (or divisorial) {\'e}tale fundamental group with modulus
$\pi^{\adiv}_1(X, D)$ is a quotient of $\pi^{\ab}_1(U)$
which classifies finite abelian covers $f \colon U' \to U$
having the property that for every generic point $\lambda$ of $C$ 
and every point $\lambda' \in f^{-1}(\lambda)$, the extension of fields 
$K_\lambda \inj K'_{\lambda'}$ has ramification bounded by $n_\lambda$.
Here, $X'$ is the normalization of $X$ in $K' = k(U')$
and $f \colon X' \to X$ is the resulting map. We thus have 
continuous surjective homomorphisms of profinite groups 
\[
G^{\ab}_{k(X)} \surj \pi^{\ab}_1(U) \stackrel{q_{X|D}}{\surj} \pi^{\adiv}_1(X,D)
\surj \pi^{\ab}_1(X).
\]
We let $\pi^{\ab}_1(U)_0$ be the kernel of the composite continuous
homomorphism
$\deg' \colon \pi^{\ab}_1(U) \surj \pi^{\ab}_1(X) \to G_k$. We define
$\pi^{\adiv}_1(X,D)_0$ similarly.

For any field $L$, let us write $H^1(L) = H^1_{\rm et}(L, {\Q}/{\Z})$ and let 
$\Fil^{\ms}_{\bullet} H^1(L)$ denote  the Matsuda filtration of
$H^1(L)$ if $L$ is a Henselian discrete valuation field
(see \cite[\S~6.2]{Gupta-Krishna-REC}). Recall the
following from \cite[\S~7.5]{Gupta-Krishna-REC}. 
We write $D = \sum_{x \in X^{(1)}} n_x \ov{\{x\}}$ as an element of the group
of Weil divisors $\Div(X)$.
Let $\Div_C(X)$ denote the set of closed subschemes of $X$
of pure codimension one whose support is $C$. This is a directed set with
respect to inclusion.
Note also that every $D' \in \Div_C(X)$ defines a unique
effective Weil divisor on $X$, which we shall also denote by $D'$.

\begin{defn}\label{defn:Fid_D}
Let $\Fil_D H^1(K)$ denote the subgroup of characters
$\chi \in H^1(K)$ such that for every $x \in X^{(1)}$, the image $\chi_x$ of
$\chi$ under the canonical map $H^1(K) \to H^1(K_x)$ 
lies in $\Fil^{\ms}_{n_x} H^1(K_x)$.
%By \thmref{thm:Fil-main}, this is equivalent to the condition that
%$\chi_x \in \Fil^{\as}_{n_x} H^1(K_x)$.

Let $\Fil^c_D H^1(K)$ be the subgroup of characters $\chi \in H^1(U)$ such that
for every integral curve $Y \subset X$ not contained in $D$ and
normalization $Y_n$, the finite map $\nu \colon Y_n \to X$ has the property that
the image of $\chi$ under $f^* \colon H^1(U) \to H^1(\nu^{-1}(U))$
lies in $\Fil_{\nu^*(D)} H^1(k(Y))$.
\end{defn}

Since $n_x = 0$ for all $x \in U^{(1)}$, it follows from various properties 
of the Matsuda filtration (see \cite[\S6]{Gupta-Krishna-REC})
that $\Fil_D H^1(K)$ lies inside $H^1(U)$ under the canonical inclusion 
$H^1(U) \inj H^1(K)$. We shall therefore write $\Fil_D H^1(K)$ also
as $\Fil_D H^1(U)$. We shall use the following properties of
$\Fil_D H^1(K)$. For a profinite (or discrete) 
group $G$, let $G^{\vee}$ denote the
Pontryagin dual of $G$ (see \cite[\S~7.4]{Gupta-Krishna-REC}).

\begin{prop}\label{prop:Fil-D-prop}
The canonical map 
${\underset{D' \in \Div_C(X)}\varinjlim} \Fil_{D'} H^1(K) \to H^1(U)$
is an isomorphism.
The canonical homomorphism of profinite groups $\pi^{\ab}_1(U) \surj
\pi^{\adiv}_1(X,D)$ defines an isomorphism of discrete groups
$(\pi^{\adiv}_1(X,D))^\vee \cong \Fil_D H^1(K)$.
\end{prop}
\begin{proof}
See \cite[Proposition~7.13, Theorem~7.16]{Gupta-Krishna-REC}.
\end{proof}

\begin{defn}\label{defn:Curve-fund-grp}
We let $\pi^{\ab}_1(X,D)$ be the Pontryagin dual of the discrete group
$\Fil^c_D H^1(K)$ and call it the 1-skeleton (or curve based) 
{\'e}tale fundamental group with modulus of the pair $(X,D)$.
We let $\pi^{\ab}_1(X,D)_0$ be the kernel of the canonical
continuous composite homomorphism $\deg' \colon \pi^{\ab}_1(X,D) \surj
\pi^{\ab}_1(X) \to G_k$. 
\end{defn}

The fundamental group $\pi^{\ab}_1(X,D)$ was first considered by
Deligne and Laumon \cite{Laumon} when $k$ is finite 
to study ramifications at infinity of abelian
coverings of smooth quasi-projective schemes. This group can be identified
with the Pontryagin dual of the group of rank one lisse $\ov{\Q}_{\ell}$ 
sheaves on $U$ with ramification bounded by $D$, a notion due to Deligne 
(see \cite{Esnault-Kerz}).

It is clear from the definitions that there are surjective continuous
homomorphisms of profinite abelian groups
$\pi^{\adiv}_1(X,D) \stackrel{q_{X|D}}{\twoheadleftarrow} \pi^{\ab}_1(U) 
\stackrel{q'_{X|D}}{\surj} \pi^{\ab}_1(X,D)$. 
Furthermore, we have the following analogue of \propref{prop:Fil-D-prop}
(see \cite[Proposition~3.9]{Esnault-Kerz} and 
\cite[Proposition~2.10]{Kerz-Saito-2}).

\begin{prop}\label{prop:Fil-c-prop}
The canonical map 
$\pi^{\ab}_1(U) \to {\underset{D' \in \Div_C(X)}\varprojlim}\pi^{\ab}_1(X,D')$
is an isomorphism if $k$ is perfect.
\end{prop}

Since $\pi^{\ab}_1(U)_0$ is closed in $\pi^{\ab}_1(U)$, we get 
from the definition of the map $\deg'$, Propositions~\ref{prop:Fil-D-prop},
~\ref{prop:Fil-c-prop} and \cite[Corollary~1.1.8]{Pro-fin} the following.

\begin{cor}\label{cor:Fil-exaust-C-0}
Assume $k$ is perfect. Then we have the
isomorphisms of profinite groups
\[
{\underset{D' \in \Div_C(X)}\varprojlim} \pi^{\adiv}_1(X,D')_0
\xleftarrow{\cong} \pi^{\ab}_1(U)_0 \xrightarrow{\cong} 
{\underset{D' \in \Div_C(X)}\varprojlim} \pi^{\ab}_1(X,D')_0.
\]
\end{cor}

It is evident from the definitions that
$\pi^{\adiv}_1(X,D) \cong \pi^{\ab}_1(X,D)$ if $d =1$.
On the other hand, it is not even clear a priori that there is any
map in either direction between $\pi^{\adiv}_1(X,D)$ and $\pi^{\ab}_1(X,D)$ 
when $d \ge 2$. One of the central results of this paper is
that these groups are in fact isomorphic for any $d \ge 1$ when $X$ is regular
and $k$ is finite.

\begin{remk}\label{remk:Curve-using-Galois}
One can mimic the construction of 
\cite[\S~7]{Gupta-Krishna-REC} (where the ramification bound has to be
redefined by restricting to curves) to show that there is a Galois category
such that $\pi^{\ab}_1(X,D)$ is the abelianization of the automorphism group 
$\pi_1(X,D)$ of the associated fiber functor. But we shall not need this 
Tannakian interpretation of $\pi^{\ab}_1(X,D)$.
\end{remk}

\vskip .3cm

\subsection{The reciprocity map for $C(X,D)$}\label{sec:Rec-00}
We let $C_{U/X}$ denote the idele class group of $(U \subset X)$. 
Recall from \cite[\S~3.3]{Gupta-Krishna-REC} that this is the quotient of the 
direct sum $I_{U/X}$ of the Milnor $K$-groups of the residue fields of 
Parshin chains on $(U \subset X)$ by the boundaries of the Milnor $K$-groups of 
the residue fields of $Q$-chains on $(U \subset X)$.
$C_{U/X}$ is a topological abelian group which has the quotient 
topology, induced by the canonical topology of the Milnor $K$-groups of the 
residue fields of Parshin chains and $C(X,D)$ has discrete topology.
The main result of \cite{Gupta-Krishna-REC} is the following.

\begin{thm}\label{thm:Rec-D-map}
Assume that $k$ is finite.
Then there exist continuous reciprocity homomorphisms 
$\rho_{U/X} \colon C_{U/X} \to \pi^{\ab}_1(U)$ and
$\rho_{X|D} \colon C(X,D) \to \pi^{\adiv}_1(X,D)$ such that the diagram
\begin{equation}\label{eqn:Rec-D-map-0}
\xymatrix@C.8pc{
{C}_{U/X} \ar[r]^-{\rho_{U/X}} \ar@{->>}[d] & \pi^{\ab}_1(U) \ar@{->>}[d] \\
C(X,D) \ar[r]^-{\rho_{X|D}} & \pi^{\adiv}_1(X,D)}
\end{equation}
is commutative, where the vertical arrows are the canonical surjections.
\end{thm}

The map $\rho_{U/X}$ has the property that it takes the
canonical generator of $K^M_0(k(x))$ for a regular
closed point $x \in U$ to the image of the Frobenius substitution at $x$.

Let $\wt{C}_{U/X} = {\underset{|D| = X \setminus U}\varprojlim} \ C(X,D)$
with the inverse limit topology.
The limit of the degree maps $C(X,D) \to \Z$ defines a continuous
homomorphism $\deg \colon \wt{C}_{U/X} \to \Z$. We let $(\wt{C}_{U/X})_0$ be the
kernel of this map (see \cite[Corollary~4.9]{Gupta-Krishna-REC}).
It is clear that there is a canonical continuous homomorphism $C_{U/X} \to
\wt{C}_{U/X}$. Moreover, \thmref{thm:Rec-D-map} and
\cite[Corollary~7.17]{Gupta-Krishna-REC} together say that
$\rho_{U/X}$ factors through a continuous reciprocity
homomorphism $\wt{\rho}_{U/X} \colon \wt{C}_{U/X} \to \pi^{\ab}_1(U)$
which is the limit of the maps $\rho_{X|D}$.

\subsection{Completions of idele class groups}\label{sec:Com}
Let $k$ be any field and $X \in \Sch_k$ an integral projective scheme.
Let $C \subset X$ be a nowhere dense closed reduced subscheme such that 
$U = X \setminus D$ is regular.
%Let $\Div_C(X)$ denote the set of closed subschemes of $X$
%of pure codimension one and with support $C$. This is a directed set with
%respect to inclusion.
%Note also that every $D \in \Div_C(X)$ defines a unique
%effective Weil divisor on $X$, which we shall also denote by $D$.
We shall let $\N$ be the set of all positive integers.
We shall say that $m_1 \le' m_2$ if $m_1$ divides $m_2$.
This makes $\N$ a directed set.
We let $I = \Div_C(X) \times \N$, 
where $\Div_C(X)$ is as in \S~\ref{sec:Rec-*}, 
and say that $(D_1, m_1) \le (D_2, m_2)$
if $D_1 \subseteq D_2$ and $m_1 \le' m_2$. It is easy to see that
$I$ is a directed set with this coordinate-wise partial order.
This partial order naturally makes $I$ a diagram category.
An inverse (resp. direct) system $\{G_{ij}\}_{(i,j) \in I}$ in a category
$\sC$ is a contravariant (resp. covariant) functor from $I$ to $\sC$.

If $\sC$ admits all small limits and $\{G_{ij}\}_{(i,j) \in I}$ is
an inverse system in $\sC$, then it is well
known that there are canonical isomorphisms

\begin{equation}\label{eqn:Limit-commute}
{\underset{i \in \Div_C(X)}\varprojlim}
{\underset{j \in \N}\varprojlim} \ G_{ij} 
\cong {\underset{(i,j) \in I}\varprojlim} G_{ij} 
\cong  {\underset{j \in \N}\varprojlim} 
{\underset{i \in \Div_C(X)}\varprojlim}  G_{ij}.
\end{equation}

Similarly, if $\sC$ admits all small colimits and $\{G_{ij}\}_{(i,j) \in I}$ is
a direct system in $\sC$, then there are canonical isomorphisms  
\begin{equation}\label{eqn:Colimit-commute}
{\underset{i \in \Div_C(X)}\varinjlim}
{\underset{j \in \N}\varinjlim} \ G_{ij} 
\cong {\underset{(i,j) \in I}\varinjlim} G_{ij} 
\cong  {\underset{j \in \N}\varinjlim}
{\underset{i \in \Div_C(X)}\varinjlim}  G_{ij}.
\end{equation}

We shall apply the above machinery to the contravariant functor 
$I \to \Ab$ given by $(D, m) \mapsto C(X,D)/m$. 
We shall denote this by $\{C(X,D)/m\}$. 
Before we study the completions of this inverse system, we state the
following crucial result.

\begin{lem}\label{lem:Constant-inverse}
For every $D \in \Div_C(X)$, the map
$C(X,D)_0 \to {C(X,D)_0}/{\infty}$ is an isomorphism.
In particular, the map $(\wt{C}_{U/X})_0 \to 
{\underset{D, m}\varprojlim} \ {C(X,D)_0}/m$
is an isomorphism.
\end{lem}
\begin{proof}
The first isomorphism is immediate from \thmref{thm:KS-fin}.
The second isomorphism follows from the first isomorphism,
\cite[Corollary~4.9]{Gupta-Krishna-REC} and ~\eqref{eqn:Limit-commute}.
\end{proof}

It follows from \cite[Proposition~4.8]{Gupta-Krishna-REC} that the sequence 
\begin{equation}\label{eqn:CI-0}
0 \to {C(X,D)_0}/m \to C(X,D)/m \xrightarrow{(\deg)/n} {\Z}/m \to 0
\end{equation}
is exact for every $(D, m) \in I$.

\begin{lem}\label{lem:Inv-mod-m}
For every $m \in \N$, the canonical map ${\wt{C}_{U/X}}/m \to
{\underset{D}\varprojlim} \ {C(X,D)}/m$ is an isomorphism.
\end{lem}
\begin{proof}
Let us denote the map of the lemma by $\theta_m$.
It follows from \cite[Proposition~4.8]{Gupta-Krishna-REC}, 
the Mittag-Leffler property
of $\{{C(X,D)_0}/m\}$ and injectivity of ${C(X,D)_0}/m \to  C(X,D)/m$
that there is a commutative diagram of short exact sequences 

\begin{equation}\label{eqn:Inv-mod-m-3}
\xymatrix@C.8pc{
0 \ar[r] & {(\wt{C}_{U/X})_0}/m \ar[r] \ar[d] & {\wt{C}_{U/X}}/m 
\ar[r]^-{(\deg)/n} \ar[d]^-{\theta_m} & {\Z}/m \ar[r] \ar@{=}[d] & 0 \\ 
0 \ar[r] & {\underset{D}\varprojlim} \ {C(X,D)_0}/m \ar[r] &
{\underset{D}\varprojlim} \ C(X,D)/m \ar[r]^-{(\deg)/n} & {\Z}/m \ar[r] & 0.}
\end{equation}

It suffices therefore to show that the left vertical arrow in 
this diagram is an isomorphism.
It follows from \cite[Corollary~4.9]{Gupta-Krishna-REC} and 
\thmref{thm:KS-fin} that
$(\wt{C}_{U/X})_0$ is profinite. Hence, the desired isomorphism follows from
\lemref{lem:Pro-fin-mod-m} below.
\end{proof}

Since each $C(X,D)_0$ is finite by \thmref{thm:KS-fin}, it follows from
\cite[Lemma~1.1.5]{Pro-fin} that the pro-abelian group
$\{{\underset{D}\varprojlim} \ {C(X,D)_0}/m\}_{m \in \N}$ is Mittag-Leffler.
Using \lemref{lem:Constant-inverse} and ~\eqref{eqn:Inv-mod-m-3}, we 
therefore conclude the following.

\begin{cor}\label{cor:Completions}
There exists a commutative diagram of short exact sequences of
continuous homomorphisms of topological abelian groups
\begin{equation}\label{eqn:Completions-0}
\xymatrix@C1.7pc{
0 \ar[r] & (\wt{C}_{U/X})_0 \ar[r] \ar@{=}[d] & \wt{C}_{U/X} \ar[r]^-{(\deg)/n}
\ar@{^{(}->}[d] & \Z \ar@{^{(}->}[d] \ar[r] & 0 \\
0 \ar[r] & (\wt{C}_{U/X})_0 \ar[r] \ar[d]_-{\theta_{U/X}}^-{\cong} & 
{\wt{C}_{U/X}}/{\infty} 
\ar[r]^-{(\deg)/n} \ar[d]^-{\theta_{U/X}}_-{\cong} & 
\wh{\Z} \ar@{=}[d] \ar[r] & 0 \\
0 \ar[r] & {\underset{D,m}\varprojlim} {C(X,D)_0}/m \ar[r] &
{\underset{D,m}\varprojlim} {C(X,D)}/m \ar[r]^-{(\deg)/n} & \wh{\Z} \ar[r] 
& 0,}
\end{equation} 
where the top vertical arrows are completion maps and the
bottom vertical arrows are isomorphisms.
In particular, $\wt{C}_{U/X}$ is dense in the profinite group
${\wt{C}_{U/X}}/{\infty}$ and there is a short exact sequence 
\begin{equation}\label{eqn:Completions-1}
0 \to \wt{C}_{U/X} \to {\wt{C}_{U/X}}/{\infty} \xrightarrow{(\deg)/n} 
{\wh{\Z}}/{\Z} \to 0.
\end{equation}
\end{cor}
\begin{proof}
We only need to give an argument for density of $\wt{C}_{U/X}$ in
${\wt{C}_{U/X}}/{\infty}$ as everything else is clear from what we 
have shown above. But this density follows immediately from
\cite[Lemma~7.11]{Gupta-Krishna-REC} using a diagram chase in 
~\eqref{eqn:Completions-0}.
\end{proof}

\begin{cor}\label{cor:Rec-final}
For every $m \in \N$ and $D \in \Div_C(X)$, 
the reciprocity map $\rho_{U/X}$ gives rise to a
commutative diagram

\begin{equation}\label{eqn:Rec-final-0}
\xymatrix@C.8pc{
& & & {C}_{U/X} \ar[dd]^-{\rho_{U/X}} \ar[dlll] \ar[dl] \\
\wt{C}_{U/X} \ar@{^{(}->}[rr] \ar[dr]^-{\wt{\rho}_{U/X}} \ar@{->>}[dd]_-{p_{X|D}} 
& & {\wt{C}_{U/X}}/{\infty} 
\ar[dr]^-{\rho^{\infty}_{U/X}}_-{\cong} \ar@{->>}[dd]_->>>>>>>>>>>>>{p_{X|D}} & \\
& \pi^{\ab}_1(U) \ar@{=}[rr] \ar@{->>}[dd]_->>>>>>>>>>>>>{q_{X|D}} & & 
\pi^{\ab}_1(U) \ar@{->>}[dd]^-{q_{X|D}} \\
C(X,D) \ar[dr]_-{\rho_{X|D}} \ar@{->>}[rr] & & 
{C(X,D)}/m \ar[dr]^-{{\rho_{X|D}}/m} & \\
& \pi^{\adiv}_1(X,D) \ar@{->>}[rr] & & {\pi^{\adiv}_1(X,D)}/m.}
\end{equation}
\end{cor}
\begin{proof}
The only thing we need to show is that $\wt{\rho}_{U/X}$ factors through
$\rho^{\infty}_{U/X}$ and the latter is an isomorphism.
Everything else is clear from the construction of $\rho_{U/X}$.
Concerning the factorization, it suffices to show that the
canonical map $\pi^{\ab}_1(U) \to {\underset{D,m}\varprojlim} 
{\pi^{\adiv}_1(X,D)}/m$ is an isomorphism.
However, it follows from \cite[Corollary~7.17]{Gupta-Krishna-REC} and
\lemref{lem:Pro-fin-mod-m} that the map
${\pi^{\ab}_1(U)}/m \to {\underset{D}\varprojlim} \ 
{\pi^{\adiv}_1(X,D)}/m$ is an isomorphism for every $m \in \N$.
Taking the limit over $m \in \N$, using ~\eqref{eqn:Limit-commute} and
the known isomorphism $\pi^{\ab}_1(U) \xrightarrow{\cong}
{\underset{m}\varprojlim} \ {\pi^{\ab}_1(U)}/m$, we get the desired
isomorphism. 

To see that $\rho^{\infty}_{U/X}$ is an isomorphism, we observe using
\cite[Lemmas~7.9, 8.3, Corollary~7.17]{Gupta-Krishna-REC}
that $\rho^{\infty}_{U/X}$ coincides with the reciprocity map
of \cite[Theorem~9.1]{Kato-Saito-2}. But the latter is an isomorphism.
\end{proof}

\begin{lem}\label{lem:Pro-fin-mod-m}
Let $G = {\underset{i}\varprojlim} \ G_i$ be the limit of an
inverse system of compact Hausdorff topological abelian groups $\{G_i\}$
whose transition maps are surjective.
Let $m \in \N$ be any integer. Then the canonical map
$G/m \to {\underset{i}\varprojlim} \ {G_i}/m$ is an isomorphism.
\end{lem}
\begin{proof}
We have a short exact sequence of pro-abelian groups
\begin{equation}\label{eqn:Inv-mod-m-0}
0 \to \{m G_i\} \to \{G_i\} \to \{{G_i}/m\} \to 0.
\end{equation}
Since $\{G_i\} \surj \{mG_i\}$ is a surjective morphism of pro-abelian
groups and $\{G_i\}$ is Mittag-Leffler, it follows that $\{m G_i\}$ is also
Mittag-Leffler. In particular, $ {\underset{i}\varprojlim}^1 \ m G_i = 0$.
We thus get a short exact sequence of limits
\begin{equation}\label{eqn:Inv-mod-m-1}
0 \to {\underset{i}\varprojlim} \ m G_i \to G \to 
{\underset{i}\varprojlim} \ {G_i}/m \to 0.
\end{equation}

We now consider the commutative diagram of exact sequences
\begin{equation}\label{eqn:Inv-mod-m-2}
\xymatrix@C.8pc{
0 \ar[r] & m G \ar[r] \ar[d] & G \ar@{=}[d] \ar[r] &
{G}/m \ar[r] \ar[d] & 0 \\
0 \ar[r] & {\underset{i}\varprojlim} \ m G_i \ar[r] & G \ar[r] &
{\underset{i}\varprojlim} \ {G_i}/m \ar[r] & 0.}
\end{equation}
Since the transition maps of $\{G_i\}$ are surjective, it follows
that $G \surj G_i$ for each $i$. Hence, $m G \surj m G_i$ for every $i$.
This implies from \cite[Corollary~1.1.8]{Pro-fin}
that $\ov{m G} = {\underset{i}\varprojlim} \ m G_i$, where $\ov{m G}$ denotes 
the closure of $m G$ in $G$.
On the other hand, the multiplication map $G \xrightarrow{m} G$ is continuous
as $G$ is a topological abelian group.
Since each $G_i$ is compact Hausdorff,
it follows that $G$ is compact Hausdorff. Hence, $m G$ must be closed
in $G$. This implies that the left vertical arrow in ~\eqref{eqn:Inv-mod-m-2}
is an isomorphism. A diagram chase now shows that the right vertical arrow is
also an isomorphism, as desired.
\end{proof}

\subsection{A result from local ramification theory}\label{sec:Ramification}
We need one more ingredient to prove \thmref{thm:Rec-mod}. This is about 
a property of the ramification filtrations.
Let $K$ be a Henselian discrete valuation field with ring of integers $\sO_K$,
maximal ideal $\fm_K$ and residue field $\ff$. We have the inclusions
$K \inj K^{sh} \inj \ov{K}$, where $\ov{K}$ is a fixed separable closure of $K$ and 
$K^{sh}$ is the strict Henselization of $K$.
Let $\Fil^{\ms}_\bullet H^1(K)$ denote the Matsuda filtration of $H^1(K)
:= H^1_{\rm et}(K, {\Q}/{\Z})$ (see \cite[\S~6.2]{Gupta-Krishna-REC}).

\begin{prop}\label{prop:Fil-SH}
There is a short exact sequence
\[
0 \to \Fil^{\ms}_0 H^1(K) \to H^1(K) \to H^1(K^{sh}) \to 0.
\]

For $m \ge 1$, the canonical square
\[
\xymatrix@C.8pc{
\Fil^{\ms}_m H^1(K) \ar[r] \ar[d] & H^1(K) \ar[d] \\
\Fil^{\ms}_m H^1(K^{sh}) \ar[r]  & H^1(K^{sh})}
\]
is Cartesian.
\end{prop}
\begin{proof}
Recall that $K^{sh}$ is the quotient field of the strict Henselization 
$\sO^{sh}_K$ of $\sO_K$. Since $\sO_K$ is Henselian, it is easy to see that
$K^{sh}$ is same as the maximal unramified extension of $K$ inside $\ov{K}$.
The exact sequence in the proposition now follows from a property of the 
Abbes-Saito filtration and \cite[Theorem~6.1(1)]{Gupta-Krishna-REC}.

We now prove that the square in the proposition is Cartesian. 
 Since $\Fil^{\ms}_{\bullet} H^1(K)$ is an exhaustive
filtration of $H^1(K)$ by \cite[Theorem~6.1(3)]{Gupta-Krishna-REC}, it 
suffices to show that for every $m \ge 2$, the square
\[
\xymatrix@C.8pc{
\Fil^{\ms}_{m-1} H^1(K) \ar[r] \ar[d] & \Fil^{\ms}_{m} H^1(K) \ar[d] \\
\Fil^{\ms}_{m-1} H^1(K^{sh}) \ar[r]  & \Fil^{\ms}_{m} H^1(K^{sh})}
\]
is Cartesian. Equivalently, it suffices to show that the
map 
\[
\phi^*_m \colon \gr^{\ms}_m H^1(K) \to \gr^{\ms}_m H^1(K^{sh}),
\]
induced by the inclusion $\phi \colon K \inj K^{sh}$, is injective.

We first assume that either $p \neq 2$ or $m \neq 2$. 
In this case, it follows from 
\cite[Proposition~3.2.3]{Matsuda} that the refined Artin conductor
\begin{equation}\label{eqn:Fil-SH-0}
{\rm rar}_K \colon \gr^{\ms}_m H^1(K) \to \frac{\fm^{-m}_K}{\fm^{-m+1}_K}
\otimes_{\sO_K} \Omega^1_{\sO_K} 
\end{equation}
induced by the map
\begin{equation}\label{eqn:Fil-SH-1}
F^rd \colon W_r(K) \to \Omega^1_K; \ \ \un{a} \mapsto
\stackrel{r-1}{\underset{i = 0}\sum} a^{p^{i}-1}_{i} da_i,
\end{equation}
is injective.

Since ${\rm rar}_K$ is clearly functorial in $K$, it suffices to show that the
map $\frac{\fm^{-m}_K}{\fm^{-m+1}_K} \otimes_{\sO_K} \Omega^1_{\sO_K} \to
\frac{\fm^{-m}_{K^{sh}}}{\fm^{-m+1}_{K^{sh}}}
\otimes_{\sO^{sh}_K} \Omega^1_{\sO^{sh}_K}$ is injective.
However, the map $\sO_K \to \fm^{-m}_{K}$
($1 \mapsto \pi_K^{-m}$) induces an isomorphism
$\ff \xrightarrow{\cong} \frac{\fm^{-m}_K}{\fm^{-m+1}_K}$.
In particular, we have an isomorphism of $\ff$-vector spaces
$\alpha_{\ff} \colon \Omega^1_{\sO_K} \otimes_{\sO_K} \ff 
\xrightarrow{\cong} \frac{\fm^{-m}_K}{\fm^{-m+1}_K} \otimes_{\sO_K} 
\Omega^1_{\sO_K}$.
Since $\sO^{sh}_K$ is unramified over $\sO_K$, we can choose
$\pi_K$ to be a uniformizer of $K^{sh}$ as well. 
This implies that $\alpha_{\ff}$ is compatible with the similar
isomorphism $\alpha_{\ov{\ff}}$, where $\ov{\ff}$ is a separable closure of 
$\ff$.
Hence, we are reduced to showing that the map
$\Omega^1_{\sO_K} \otimes_{\sO_K} \ff \to \Omega^1_{\sO^{sh}_K} \otimes_{\sO^{sh}_K} 
\ov{\ff}$ is injective. 

Since $\sO^{sh}_K$ is {\'e}tale over $\sO_K$, the map
$\Omega^1_{\sO_K} \otimes_{\sO_K} \sO^{sh}_K \to \Omega^1_{\sO^{sh}_K}$ is an
isomorphism. Hence, we need to show that the map
$\Omega^1_{\sO_K} \otimes_{\sO_K} \ff \to \Omega^1_{\sO_K} \otimes_{\sO_K} \ov{\ff}$
is injective. But this is obvious since $\Omega^1_{\sO_K}$ is a free
$\sO_K$-module.

We now assume $p = m = 2$. Let $\ff^{1/2}$ be the field extension of
$\ff$ obtained by adjoining square roots of elements of $\ff$.
We define $\ov{\ff}^{1/2}$ similarly.
Then \cite[Proposition~1.17]{Yatagawa} says that there is an
injective modified refined Artin conductor map

\begin{equation}\label{eqn:Fil-SH-2}
{\rm rar}'_K \colon \gr^{\ms}_2 H^1(K) \to (\frac{\fm^{-2}_K}{\fm^{-1}_K}
\otimes_{\sO_K} \Omega^1_{\sO_K}) \otimes_{\ff} \ff^{1/2} \cong 
\fm^{-2}_K \Omega^1_{\sO_K} \otimes_{\sO_K} \ff^{1/2},
\end{equation}
where the isomorphism on the right depends on the choice of the
uniformizer $\pi_K$. Note that even if
Yatagawa works with complete fields in the latter part of her paper, 
the proof of the above works for any Henselian discrete valuation field 
with no modification.

Since $\pi_K$ is also a uniformizer of $K^{sh}$,
it suffices to show that the
map $\fm^{-2}_K \Omega^1_{\sO_K} \otimes_{\sO_K} \ff^{1/2} \to
\fm^{-2}_{K^{sh}} \Omega^1_{\sO^{sh}_K} \otimes_{\sO^{sh}_K} {\ov{\ff}}^{1/2}$
is injective. As we argued in the previous case, this reduces to showing that 
the map $\Omega^1_{\sO_K} \otimes_{\sO_K} \ff^{1/2} \to \Omega^1_{\sO_K} 
\otimes_{\sO_K} {\ov{\ff}}^{1/2}$ is injective.
But this follows again by the fact that $\Omega^1_{\sO_K}$ is a free
$\sO_K$-module. 
\end{proof}

We shall use the following consequence of \propref{prop:Fil-SH} in the
proof of \thmref{thm:Rec-mod}.
Let $k$ be a finite field and $X \in \Sch_k$ an integral normal scheme
of dimension $d$. 
Let $D \subset X$ be an effective Weil divisor with $U = X \setminus |D|$. 
Let $\lambda$ be a generic point of $|D|$.
Let $\eta$ denote the generic point of $X$ and $K = k(\eta)$.
Let $K_\lambda$ denote the Henselization of $K$ at $\lambda$.
Let $P = (p_0, \ldots , p_{d-2}, \lambda, \eta)$ be a Parshin
chain on $(U \subset X)$. 

Recall the following notations from \cite[\S~3.3]{Kato-Saito-2} or
\cite[\S~2.3]{Gupta-Krishna-REC}. 
Let $V \subset K$ be a $d$-DV which dominates $P$. 
Let $V = V_0 \subset \cdots \subset V_{d-2}
\subset V_{d-1} \subset V_d = K$ be the chain of valuation rings
in $K$ induced by $V$. Since $X$ is normal, it is easy to check
that for any such chain, one must have $V_{d-1} = \sO_{X,\lambda}$.
Let $V'$ be the image of $V$ in $k(\lambda)$. Let $\wt{V}_{d-1}$
be the unique Henselian discrete valuation ring 
having an ind-{\'e}tale local homomorphism 
 $V_{d-1} \to \wt{V}_{d-1}$ such that its residue field $E_{d-1}$
is the quotient field of $(V')^h$. Then $V^h$ is the inverse image of
$(V')^h$ under the quotient map $\wt{V}_{d-1} \surj E_{d-1}$.
It follows that its function field $Q(V^h)$ is a Henselian discrete valuation field whose
ring of integers is 
$\wt{V}_{d-1}$ (see \cite[\S~3.7.2]{Kato-Saito-2}).

It then follows that
there are canonical inclusions of discrete valuation rings 
\begin{equation}\label{eqn:Fil-SH-3}
\sO_{X,\lambda}  \inj \wt{V}_{d-1} \inj \sO^{sh}_{X,\lambda}.
\end{equation}
Moreover, we have (see the proof of \cite[Proposition~3.3]{Kato-Saito-2})
\begin{equation}\label{eqn:Fil-SH-5}
\sO^h_{X,P'} \cong {\underset{V \in \sV(P)}\prod} \wt{V}_{d-1},
\end{equation}
where $\sV(P)$ is the set of $d$-DV's in $K$ which dominate $P$.
As an immediate consequence of \propref{prop:Fil-SH},
we therefore get the following.

\begin{cor}\label{cor:Fil-SH-4}
For every $m \ge 1$, the square
\[
\xymatrix@C.8pc{
\Fil^{\ms}_m H^1(K_\lambda) \ar[r] \ar[d] & H^1(K_\lambda) \ar[d] \\
\Fil^{\ms}_m H^1(Q(V^h)) \ar[r]  & H^1(Q(V^h))}
\]
is Cartesian.
\end{cor}

\subsection{Proof of \thmref{thm:Rec-mod}}\label{sec:Pf-1}
Let $k$ be a finite field and $X \in \Sch_k$ an integral normal scheme.
Let $D \subset X$ be a closed subscheme of pure codimension one such that
$U = X \setminus D$ is regular. The density assertion of
\thmref{thm:Rec-mod} is a direct consequence of the
Chebotarev-Lang density theorem \cite[Theorem~5.8.16]{Szamuely}.
By \cite[Lemma~8.4]{Gupta-Krishna-REC}, the heart of the proof 
is to show that 
\begin{equation}\label{eqn:Rec-mod-00}
\rho_{X|D} \colon C(X,D)_0 \to \pi^{\adiv}_1(X,D)_0
\end{equation}
is an isomorphism if $X$ is projective over $k$.
The finiteness claim will then follow from \thmref{thm:KS-fin}.

We first show that ~\eqref{eqn:Rec-mod-00} is injective.
Using \thmref{thm:KS-fin}, this is
equivalent to showing that the map $\rho^{\vee}_{X|D} \colon
(\pi^{\adiv}_1(X,D)_0)^{\vee} \to (C(X,D)_0)^{\vee}$ is surjective.
We fix a character $\chi \in (C(X,D)_0)^{\vee}$.
We choose $m \in \N$ large enough such that $C(X,D)_0 \cong
{C(X,D)_0}/m$ using \thmref{thm:KS-fin}.
Using \cite[Proposition~4.8, Lemma~7.10]{Gupta-Krishna-REC}, $\chi$ lifts to
a character of ${C(X,D)}/m$. 
Choose one such lift and denote its image in $C(X,D)^{\vee}$ also by $\chi$.
We let $\wt{\chi} = p^{\vee}_{X|D}(\chi)$ and
consider the commutative diagram (see \cite[\S~5.6]{Gupta-Krishna-REC})

\begin{equation}\label{eqn:Rec-mod-0}
\xymatrix@C.8pc{  
\Fil_D H^1(K) \ar[r]^-{\rho^{\vee}_{X|D}} \ar@{^{(}->}[d]_{{q}^{\vee}_{X|D}} &  
C(X,D)^{\vee} 
\ar@{^{(}->}[d]^-{p^{\vee}_{X|D}}  & ({C(X,D)}/m)^{\vee}
\ar@{_{(}->}[l] \ar[dd] \\
H^1(U) \ar[r]^-{\wt{\rho}^{\vee}_{U/X}} 
\ar@/_2pc/[drr]_-{(\rho^{\infty}_{U/X})^{\vee}}^{\cong} & 
(\wt{C}_{U/X})^{\vee} 
&  \\
& &  ({\wt{C}_{U/X}}/{\infty})^{\vee}. \ar[ul]}
\end{equation}

Since $\chi \in ({C(X,D)}/m)^{\vee}$, it follows from \corref{cor:Completions}
that $\wt{\chi} \in ({\wt{C}_{U/X}}/{\infty})^\vee$. We conclude from 
\corref{cor:Rec-final} that there exists $\chi' \in H^1(U)$
such that $\wt{\rho}^{\vee}_{U|X}(\chi') = \wt{\chi}$.
It suffices to show that $\chi'$ lies in the image of
${q}^{\vee}_{X|D}$.

We let $C = D_\red$. We fix a point $x \in \Irr_C$ and let $\chi'_x$ be the 
image of $\chi'$ in
$H^1(K_x)$. We need to show that $\chi'_x \in \Fil^{\ms}_{n_x} H^1(K_x)$,
where $n_x$ is the multiplicity of $D$ at $x$.
By \corref{cor:Fil-SH-4}, it suffices to show that
for some maximal Parshin chain $P = (p_0, \ldots , p_{d-2}, x, \eta)$
on $(U \subset X)$
and $d$-DV $V \subset K$ dominating $P$, the image of $\chi'_x$ in 
$H^1(Q(V^h))$ lies in the subgroup $\Fil^{\ms}_{n_x} H^1(Q(V^h))$.

Choose any Parshin chain as above and call it $P_0$.
Let $V \subset K$ be a $d$-DV dominating $P_0$ and let
$\wh{\chi}_x$ denote the image of $\chi'_x$ in $H^1(Q(V^h))$.
Let $V = V_0 \subset \cdots \subset V_{d-1} \subset V_d = K$
be the chain of valuation rings associated to $V$.
Then $Q(V^h)$ is a $d$-dimensional Henselian local field whose
ring of integers is $\wt{V}_{d-1}$ (see \cite[\S2.3]{Gupta-Krishna-REC}).
By \cite[Theorem~6.3]{Gupta-Krishna-REC}, it suffices to show that
$\{\alpha, \wh{\chi}_x\} = 0$ for every $\alpha \in K^M_d(\wt{V}_{d-1}, I_D) =
K^M_d(\wt{V}_{d-1}, \fm^{n_x})$ under the pairing
$K^M_d(Q(V^h)) \times H^1(Q(V^h)) \to H^{d+1}(Q(V^h))$.
Here, $\fm$ is the maximal ideal of $\wt{V}_{d-1}$.

Now, we are given that $\wt{\chi}$ annihilates $\Ker(p_{X|D})$ and the latter
is the sum of images of $K^M_d(\sO^h_{X,P'}, I_D) \to C_{U/X}$, where
$P$ runs through all maximal Parshin chains on $(U \subset X)$.
In particular, $\chi' \circ \rho_{U/X}$ annihilates the image of
$K^M_d(\sO^h_{X, P'_0}, I_D) \to C_{U/X}$. It follows from
~\eqref{eqn:Fil-SH-5} that $\chi' \circ \rho_{U/X}$ annihilates the image of
$K^M_d(\wt{V}_{d-1}, \fm^{n_x}) \to C_{U/X}$. Equivalently,
$\{\alpha, \wh{\chi}_x\} = 0$ for every $\alpha \in 
K^M_d(\wt{V}_{d-1}, \fm^{n_x})$. We have thus proven the desired claim,
and hence the injectivity of ~\eqref{eqn:Rec-mod-00}. 

We now show that ~\eqref{eqn:Rec-mod-00} is surjective.
Since $\pi^{\adiv}_1(X,D)_0$ is a profinite group and since
$C(X,D)_0$ is finite by \thmref{thm:KS-fin}, it suffices to show that
the image of $C(X,D)_0$ is dense in $\pi^{\adiv}_1(X,D)_0$. 
By \cite[Lemma~7.11]{Gupta-Krishna-REC}, we need to show that every element of
$(\pi^{\adiv}_1(X,D)_0)^\vee$ which vanishes on $C(X,D)_0$ is zero. 
We fix $\chi \in (\pi^{\adiv}_1(X,D)_0)^\vee$ which vanishes on $C(X,D)_0$.

By \cite[Corollary~8.5]{Gupta-Krishna-REC}, there is a commutative
diagram
\begin{equation}\label{eqn:Rec-proj-4}
\xymatrix@C.8pc{
0 \ar[r] & (\wh{\Z})^\vee \ar[r] \ar[d]_-{\cong} & 
(\pi^{\adiv}_1(X,D))^\vee \ar[r] 
\ar[d]^-{\rho^{\vee}_{X|D}} & (\pi^{\adiv}_1(X,D)_0)^\vee \ar[r] 
\ar[d]^-{\rho^{\vee}_{X|D}} & 0 \\
0 \ar[r] & \Z^{\vee} \ar[r] & C(X,D)^\vee \ar[r] & (C(X,D)_0)^\vee \ar[r] & 0,}
\end{equation}
where the left vertical arrow is an isomorphism.

This diagram shows that
$\chi$ lifts to an element $\chi' \in (\pi^{\adiv}_1(X,D))^\vee$.
Since $\rho^{\vee}_{X|D}(\chi) = 0$, a diagram chase shows that there exists
an element $\chi'' \in (\wh{\Z})^\vee$ such that 
$\rho^{\vee}_{X|D}(\chi' - \chi'') = 0$.
On the other hand, since the image of the map $\rho_{X|D}$ is dense, 
it follows from 
\cite[Lemma~7.11]{Gupta-Krishna-REC} that the middle vertical arrow in 
~\eqref{eqn:Rec-proj-4}
is injective. It follows that $\chi' = \chi''$. Equivalently,
$\chi = 0$. The proof of \thmref{thm:Rec-mod} is now complete.
\qed
%$\hfill \square$

\vskip .4cm

\begin{cor}\label{cor:Rec-mod-2}
Let $X$ and $D$ be as in \thmref{thm:Rec-mod} with $X$ projective.
Then for every $m \in \N$, the reciprocity map $\rho_{X|D}$ induces an 
isomorphism of finite groups
\[
\rho_{X|D} \colon {C(X,D)}/m \xrightarrow{\cong} {\pi^{\adiv}_1(X,D)}/m.
\]
\end{cor}
\begin{proof}
Use \cite[Lemma~8.4]{Gupta-Krishna-REC} and \thmref{thm:Rec-mod}.
\end{proof}

\section{Reciprocity theorem for $\pi^{\ab}_1(X,D)$}
\label{sec:Curve-based}
In this section, we shall prove a reciprocity theorem for 
$\pi^{\ab}_1(X,D)$ in the pro-setting as $D$ varies over a set of closed 
subschemes of pure codimension one $X$ all of which have the same support. 
The main application for us will be in the proof of \thmref{thm:Main-1*}.

Let $k$ be a finite field and $X$ a projective integral
normal scheme of dimension $d \ge 1$ over $k$.
We let $C \subset X$ be a reduced closed subscheme of pure codimension one
with complement $U$. We assume that $U$ is
regular. 
We shall need the following results.

\begin{lem}\label{lem:Fun-funct}
Let $f \colon X' \to X$ be a morphism of integral normal schemes whose image
is not contained in $C$.
Let $D' \subset X'$ be an effective Weil divisor.
Assume that $f^*(D)$ is an effective Weil divisor on $X'$ such
that $f^*(D) \le D'$. Then there is a push-forward map
\[
f_* \colon \pi^{\ab}_1(X', D') \to \pi^{\ab}_1(X,D).
\]
\end{lem}
\begin{proof}
Obvious from Definition~\ref{defn:Curve-fund-grp}.
\end{proof}

\begin{lem}\label{lem:KS-C-fin}
Assume that $k$ is finite and $X$ is projective over $k$.
Then the group $\pi^{\ab}_1(X,D)_0$ is finite. In particular, the map
$\pi^{\ab}_1(X,D)_0 \to {\pi^{\ab}_1(X,D)_0}/{\infty}$ is an isomorphism.
\end{lem}
\begin{proof}
This is shown in \cite[Corollary~1.2]{Kerz-Saito-1} when $D \subset X$
is an effective Cartier divisor. But one does not need the latter assumption.
The reason is that since $D$ is anyway a closed subscheme 
which defines a Weil divisor,
its inverse image $f^*(D)$ becomes an effective Cartier divisor on
$X'$ where $f \colon X' \to X$ is any smooth alteration.
We can replace $X$ by any such alteration using \lemref{lem:Fun-funct}
because the map
$f_* \colon \pi^{\ab}(f^{-1}(U)) \to \pi^{\ab}_1(U)$ has finite cokernel. 
\end{proof}

The following is the key step for proving the main result
of this section.

\begin{lem}\label{lem:C-Rec-0}
For every $D \in \Div_C(X)$, there exists $D' \in \Div_C(X)$
such that the composite map
$\wt{C}_{U/X} \xrightarrow{\wt{\rho}_{U/X}} \pi^{\ab}_1(U)
\xrightarrow{q'_{X|D}} \pi^{\ab}_1(X,D)$ factors through 
\[
\rho^c_{X|D} \colon C(X,D') \to \pi^{\ab}_1(X,D).
\]
\end{lem}
\begin{proof}
We fix a closed subscheme $D \in \Div_C(X)$.
For any $D' \in \Div_C(X)$, let $F_{D'} = \Ker(\wt{C}_{U/X} \surj C(X,D'))$.
It follows from \cite[Proposition~4.8]{Gupta-Krishna-REC} that
$F_{D'} = \Ker((\wt{C}_{U/X})_0 \surj C(X,D')_0)$.
Hence, it suffices to show that the composite map
$(\wt{C}_{U/X})_0 \xrightarrow{\wt{\rho}_{U/X}} \pi^{\ab}_1(U) \surj
\pi^{\ab}_1(X,D)$ annihilates $F_{D'}$ for some $D'$.
Using \cite[Lemma~8.4]{Gupta-Krishna-REC}, we see that $\wt{\rho}_{U/X}$ 
induces the maps 
$(\wt{C}_{U/X})_0 \xrightarrow{\wt{\rho}_{U/X}} \pi^{\ab}_1(U)_0 \surj
\pi^{\ab}_1(X,D)_0$. Hence, we need to show that this composite map
annihilates $F_{D'}$ for some $D'$.

Now, we know from \thmref{thm:KS-fin} that each $C(X,D')_0$ is finite.
It follows from \cite[Corollary~4.9]{Gupta-Krishna-REC} that
$(\wt{C}_{U/X})_0 \xrightarrow{\cong} 
{\underset{D' \in \Div_C(X)}\varprojlim} C(X,D')_0$.
In particular, $(\wt{C}_{U/X})_0$ is profinite.
Since the maps 
$(\wt{C}_{U/X})_0 \xrightarrow{\wt{\rho}_{U/X}} \pi^{\ab}_1(U)_0 \surj
\pi^{\ab}_1(X,D)_0$ are continuous, it follows that the images of
$F_{D'}$ under these maps are closed subgroups.
We let $E_{D'}$ be the image of $F_{D'}$ in $\pi^{\ab}_1(X,D)_0$ under
the composite map. 

We now let $\chi \in (\pi^{\ab}_1(X,D)_0)^{\vee}$ be a continuous character.
Then the composite $\chi' := \chi \circ q'_{X|D} \circ \wt{\rho}_{U/X}$
is a continuous character of $(\wt{C}_{U/X})_0$.
Since $(\wt{C}_{U/X})_0$ is profinite as we just saw, $\chi'$ factors
through some $C(X, D_{\chi})_0$. In other words,
$\chi(E_{D_{\chi}}) = 0$. Since $\pi^{\ab}_1(X,D)_0$ is finite by
\lemref{lem:KS-C-fin}, $(\pi^{\ab}_1(X,D)_0)^\vee$ is also finite
(see \cite[Example~2.9.5]{Pro-fin}). 
We can therefore choose $D' \in \Div_C(X)$ which
dominates $D_{\chi}$ for all $\chi \in (\pi^{\ab}_1(X,D)_0)^\vee$.
It is then clear that $\chi$ annihilates $E_{D'}$ for all $\chi \in
(\pi^{\ab}_1(X,D)_0)^\vee$.
We have thus shown that $E_{D'}$ is a closed subgroup of
$\pi^{\ab}_1(X,D)_0$ which is annihilated by all characters of the
latter group. But then the Pontryagin duality theorem says that $E_{D'}$ must
be zero. Equivalently, $F_{D'}$ lies in the kernel of
the composite map 
$(\wt{C}_{U/X})_0 \xrightarrow{\wt{\rho}_{U/X}} \pi^{\ab}_1(U)_0 \surj
\pi^{\ab}_1(X,D)_0$. This proves the lemma.
\end{proof}

Let $D \in \Div_C(X)$ and $nD \subset X$ the closed subscheme
defined by $\sI^n_D$, where $\sI_D$ is the sheaf of ideals defining $D$.
For every $n \in \N$, the set of closed subschemes $n'D \in \Div_C(X)$ 
which have the property described in \lemref{lem:C-Rec-0} is linearly 
ordered by inclusion. Since $X$ is Noetherian, there exists smallest 
integer $\lambda(n) \in \N$ such that $\lambda(n) \ge \lambda(n-1)$ and
$\lambda(n)D$ satisfies the property asserted in \lemref{lem:C-Rec-0}.
That is, the composite map
$\wt{C}_{U/X} \xrightarrow{\wt{\rho}_{U/X}} \pi^{\ab}_1(U)
\xrightarrow{q'_{X|D}} \pi^{\ab}_1(X,nD)$ factors through 
\[
\rho^c_{X|nD} \colon C(X,\lambda(n)D) \to \pi^{\ab}_1(X,nD).
\]

We let $\lambda \colon \N \to \N$ be the above function.
We write $\lambda(1)D = D_\rho$.
As a consequence of \lemref{lem:C-Rec-0}, we
conclude the following.

\begin{thm}\label{thm:C-Rec}
The reciprocity map ${\rho}_{U/X} \colon {C}_{U/X} \to \pi^{\ab}_1(U)$
of \thmref{thm:Rec-D-map} induces a continuous homomorphism between the
topological pro-abelian groups
\[
\rho^{\bullet}_{X|D} \colon \{C(X,nD)\}_{n \in \N} \to
\{\pi^{\ab}_1(X,nD)\}_{n \in \N}
\]
such that ${\underset{n \in \N}\varprojlim} \ \rho^\bullet_{X|D} 
= \wt{\rho}_{U/X}$.
\end{thm}

From the construction of $\wt{\rho}_{U/X}$, we actually get the
following commutative diagram of short exact sequences of topological 
pro-abelian groups. We ignore to write
the indexing set $\N$.

\begin{equation}\label{eqn:C-Rec-2}
\xymatrix@C2.5pc{
0 \ar[r] & \{C(X,nD)_0\} \ar[r] \ar[d]_-{\rho^{\bullet}_{X|D}} & 
\{C(X,nD)\} \ar[r]^-{{\deg}/{n_0}} \ar[d]_-{\rho^{\bullet}_{X|D}} & \Z 
\ar@{^{(}->}[d] 
\ar[r] & 0 \\
0 \ar[r] & \{\pi^{\ab}_1(X,nD)_0\} \ar[r] & \{\pi^{\ab}_1(X,nD)\} 
\ar[r]^-{{\deg'}/{n_0}} &
\wh{\Z} \ar[r] & 0,}
\end{equation}
where the right vertical arrow is the profinite completion map and
$n_0 \in \N$ depends only on $U$ and not on $\Div_C(X)$. 
Moreover, taking the limits of the vertical arrows, we get the
commutative diagram of short exact sequences of topological abelian groups

\begin{equation}\label{eqn:C-Rec-3}
\xymatrix@C2.5pc{
0 \ar[r] & (\wt{C}_{U/X})_0 \ar[r] \ar[d]_-{\wt{\rho}_{U/X}} & 
\wt{C}_{U/X} \ar[r]^-{{\deg}/{n_0}} \ar[d]_-{\wt{\rho}_{U/X}} & \Z 
\ar@{^{(}->}[d] 
\ar[r] & 0 \\
0 \ar[r] & \pi^{\ab}_1(U)_0 \ar[r] & \pi^{\ab}_1(U) \ar[r]^-{{\deg'}/{n_0}} &
\wh{\Z} \ar[r] & 0}
\end{equation}
by \propref{prop:Fil-c-prop} and Corollary~\ref{cor:Fil-exaust-C-0}.
Note that the projective limits over $\Div_C(X)$ and $\{nD\}_{n \in \N}$
coincide. 

Our reciprocity theorem for the 1-skeleton {\'e}tale fundamental group
with modulus is the following.

\begin{thm}\label{thm:C-Rec-iso}
For each $D \in \Div_C(X)$, the reciprocity map of pro-abelian groups 
\[
\rho^{\bullet}_{X|D} \colon \{C(X,nD)\}_{n \in \N} \to
\{\pi^{\ab}_1(X,nD)\}_{n \in \N}
\]
is injective and has dense image.
\end{thm}
\begin{proof}
The proof of the density assertion is same as in \thmref{thm:Rec-mod}.
We shall prove pro-injectivity.

Using the commutative diagram~\eqref{eqn:C-Rec-2}, it suffices to show
the injectivity at the level of the degree zero subgroups.
Before we do this, we claim that $\rho^c_{X|nD} \colon C(X, \lambda(n)D)_0 \to 
\pi^{\ab}_1(X, nD)_0$ is surjective. 
Indeed, as the latter group is finite
by \lemref{lem:KS-C-fin}, it suffices to show that
this map has dense image. But the proof of this is identical to
that of the surjectivity of the map ~\eqref{eqn:Rec-mod-00}
in \thmref{thm:Rec-mod}, which only used 
that $\rho_{X|D}$ has dense image and not its
injectivity.

Let $n \in \N$ and write $n' = \lambda(n)$. We let $F_n$ denote the kernel of
$C(X,n'D)_0 \surj \pi^{\ab}_1(X, nD)_0$.
Using the Pontryagin duality for finite groups, it suffices to show
that the ind-abelian group $\{F^\vee_n\}$ is zero.

Choose a character $\chi \in F^\vee_n$ and lift it to a character of
$C(X, n'D)_0$. We denote this lift also by $\chi$.
We choose $m \in \N$ large enough such that $C(X, n'D)_0 \cong
{C(X, n'D)_0}/m$ using \thmref{thm:KS-fin}.
Using \cite[Proposition~4.8, Lemma~7.10]{Gupta-Krishna-REC}, $\chi$ lifts to
a character of ${C(X, n'D)}/m$. 
Choose one such lift and denote its image in $C(X, n'D)^{\vee}$ also by 
$\chi$.
We let $\wt{\chi} = p^{\vee}_{X|n'D}(\chi)$ and
consider the commutative diagram (see ~\eqref{eqn:Rec-mod-0})

\begin{equation}\label{eqn:C-Rec-iso-0}
\xymatrix@C.8pc{  
\Fil^c_{nD} H^1(K) \ar[r]^-{(\rho^c_{X|nD})^\vee} 
\ar@{^{(}->}[d]_{{q'}^{\vee}_{X|nD}} &  
C(X, n'D)^{\vee} 
\ar@{^{(}->}[d]^-{p^{\vee}_{X|{n'D}}}  & ({C(X, n'D)}/m)^{\vee}
\ar@{_{(}->}[l] \ar[dd] \\
H^1(U) \ar[r]^-{\wt{\rho}^{\vee}_{U/X}} 
\ar@/_2pc/[drr]_-{(\rho^{\infty}_{U/X})^{\vee}}^-{\cong} & 
(\wt{C}_{U/X})^{\vee} 
&  \\
& &  ({\wt{C}_{U/X}}/{\infty})^{\vee}. \ar[ul]}
\end{equation}

Since $\chi \in ({C(X, n'D)}/m)^{\vee}$, 
it follows from \corref{cor:Completions}
that $\wt{\chi} \in ({\wt{C}_{U/X}}/{\infty})^\vee$. We conclude from 
\corref{cor:Rec-final} that there exists $\chi' \in H^1(U)$
such that $\wt{\rho}^{\vee}_{U/X}(\chi') = \wt{\chi}$.
It follows from \propref{prop:Fil-c-prop} that ${\chi}' \in 
\Fil^c_{n_1 D} H^1(K)$ for some $n_1 \gg n$.

We let $\chi_\rho$ be the image of ${\chi}'$ under the 
map $(\pi^{\ab}_1(X, n_1 D))^\vee \to 
(\pi^{\ab}_1(X, n_1 D)_0)^\vee$.
Then it follows that the image of $\chi_\rho$ in $(C(X, \lambda(n_1)D)_0)^\vee$
lies in the image of $(\rho^c_{X|n_1 D})^\vee$.
Since $F^\vee_n = \coker((\rho^c_{X|nD})^\vee)$ by 
\cite[Lemma~7.10]{Gupta-Krishna-REC},
it follows that $\chi$ dies under the canonical map $F^\vee_n \to F^\vee_{n_1}$.
Since $F^\vee_n$ is finite, we can find
$n_1 \gg n$ such that the map
$F^\vee_n \to F^\vee_{n_1}$ is zero. We have thus shown that
the ind-abelian group $\{F^\vee_n\}$ is zero.
This finishes the proof.
\end{proof}

\begin{cor}\label{cor:C-Rec-iso-1}
The morphism of pro-abelian groups 
\[
\rho^{\bullet}_{X|D} \colon \{C(X,nD)_0\}_{n \in \N} \to
\{\pi^{\ab}_1(X,nD)_0\}_{n \in \N}
\]
is an isomorphism.
\end{cor}
\begin{proof}
The map $\rho^\bullet_{X|D}$ is injective by \thmref{thm:C-Rec-iso}.
Moreover, we also showed in the proof of this theorem that
$\rho^\bullet_{X|D}$ is  surjective.
\end{proof}

\begin{cor}\label{cor:C-Rec-iso-2}
The morphism of pro-abelian groups 
\[
\rho^{\bullet}_{X|D} \colon \{{C(X,nD)}/m\}_{n \in \N} \to
\{{\pi^{\ab}_1(X,nD)}/m\}_{n \in \N}
\]
is an isomorphism for every $m \in \N$.
\end{cor}
\begin{proof}
The proof is identical to that of \corref{cor:Rec-mod-2} using
\corref{cor:C-Rec-iso-1}.
\end{proof}

The following result will be improved in \S~\ref{sec:Reg-BF}
when $X$ is regular.
We nevertheless need this weaker version
to prove Bloch's formula for normal varieties. 

\begin{thm}\label{thm:C-Rec-iso-3}
For every $D \in \Div_C(D)$, 
the identity map of $\pi^{\ab}_1(U)$ induces an isomorphism of
topological pro-abelian groups
\[
\theta^\bullet_{X|D} \colon \{\pi^{\adiv}_1(X,nD)\}_{n \in \N} \xrightarrow{\cong} 
\{\pi^{\ab}_1(X,nD)\}_{n \in \N}
\]
making the diagram
\[
\xymatrix@C.8pc{
\wt{C}_{U/X} \ar[d]_-{\wt{\rho}_{U/X}} \ar[r] & 
\{C(X,nD)\} \ar[d]^-{\{\rho^{\divf}_{X|nD}\}} \ar[dr]^-{\rho^\bullet_{X|D}} & \\
\pi^{\ab}_1(U) \ar[r] & \{\pi^{\adiv}_1(X,nD)\} \ar[r]^-{\theta^\bullet_{X|D}} 
& \{\pi^{\ab}_1(X,nD)\}} 
\]
commute.
\end{thm}
\begin{proof}
The commutativity of the diagram is clear from the construction of various
reciprocity maps.
We fix $n$ and let $n' = \lambda(n)$.
Let $F$ denote the kernel of the map $\pi^{\ab}_1(U) \to 
\pi^{\adiv}_1(X, n'D)$.
Then $F$ is same as the kernel of the map $\pi^{\ab}_1(U)_0 \to 
\pi^{\adiv}_1(X, n'D)_0 \cong {\pi^{\adiv}_1(X, n'D)_0}/m$ for all
$m \gg 0$ (see \thmref{thm:Rec-mod}). 

By Corollary~\ref{cor:Rec-final}, $\rho^{\infty}_{U/X}$ is an isomorphism
and it induces an isomorphism between  $(\wt{C}_{U/X})_0$ and
$\pi^{\ab}_1(U)_0$ by \corref{cor:Completions}.
Hence, we conclude from \thmref{thm:Rec-mod} that
$\rho^{\infty}_{X|D}$ induces an isomorphism
$\Ker((\wt{C}_{U/X})_0 \to C(X, n'D)_0) \xrightarrow{\cong} F$.
However, the kernel on the left hand side dies in 
$\pi^{\ab}_1(X,nD)_0$. It follows that $F$ dies in
$\pi^{\ab}_1(X,nD)_0$.

To prove that $\theta^\bullet_{X|D}$ is an isomorphism, we only need to show
it is injective. To prove this, it is equivalent to show that
the map $\{\pi^{\adiv}_1(X,nD)_0\} \to \{\pi^{\ab}_1(X,nD)_0\}$
is injective. But this follows from \thmref{thm:Rec-mod} and
Corollary~\ref{cor:C-Rec-iso-1}.
\end{proof}

\vskip .3cm

We conclude this section by proving
the following property of $\Fil^c_D H^1(K)$ which will be used in the
proof of \thmref{thm:Main-5} in \S~\ref{sec:Reg-BF}.

\begin{prop}\label{prop:Fil-c-rest}
Let $k$ be any field.
Assume that $\dim(X) \ge 2$ and $A \subset X$ is a closed 
subscheme such that $A_\red \subset C$ and $\dim(A) \le \dim(X) - 2$. 
Let $X' = X \setminus A, \ C' = C \setminus A$ and $D' = D \setminus A$.
Then there is an inclusion $\Fil^c_D H^1(K) \subset \Fil^c_{D'} H^1(K)$ of
subgroups of $H^1(K)$.
\end{prop}
\begin{proof}
Suppose $\chi \in \Fil^c_D H^1(K)$ and let $Y' \subset X'$ be an
integral curve not contained in $D'$. Let $Y \subset X$ be the
scheme theoretic closure of $Y'$ in $X$. Then $Y$ is an integral
curve in $X$ not contained in $D$. Let $\nu \colon Y_n \to X$ be
induced map from the normalization of $Y$. We let $D_{Y_n} = D \times_X Y_n$
and $D'_{Y'_n} = D \times_X Y'_n$. We then get a
commutative diagram
\begin{equation}\label{eqn:Fil-c-rest-0}
\xymatrix@C.8pc{
D'_{Y'_n} \ar[r] \ar[d] & Y'_n \ar[r]^-{\nu'} \ar[d] & X' \ar[d] \\
D_{Y_n} \ar[r] & Y_n \ar[r]^-{\nu} & X}
\end{equation} 
in which the two squares are Cartesian, vertical arrows are
open immersions and horizontal arrows are finite.

We write $D_{Y_n} = {\underset{x \in \Sigma}\sum} m_x [x]$,
where $\Sigma$ is the support of $\nu^{-1}(C)$.
Then the above diagram says that 
$D'_{Y'_n} = {\underset{x \in \Sigma'}\sum} m_x [x]$,
where $\Sigma' = \nu^{-1}(C')$.
In particular, $D'_{Y'_n}$ is an effective Weil divisor on $Y_n$ such that
$D'_{Y'_n} \le D_{Y_n}$ and every $x \in \Sigma'$ (note that $D'_{Y'_n}$ will be 
zero if $Y \cap C \subset A$) has the property that
the multiplicity of $D'_{Y'_n}$ at $x$ is same is that of $D_{Y_n}$.
But this implies that $\nu'^*(\chi) \in H^1(\nu'^{-1}(U)) =
H^1(\nu^{-1}(U))$ has the property that its image in
$H^1(k(Y)_x)$ lies in $\Fil^{\ms}_{m_x} H^1(k(Y)_x)$, for all $x \in \Sigma'$. 
This proves the proposition.
\end{proof}

\section{A moving lemma for 0-cycles with modulus}
\label{sec:ML}
Our next goal is to prove a moving lemma for the Chow group of 0-cycles with
modulus which will be the key ingredient in the proof of \thmref{thm:Main-5}.
In this section, we shall recall the Chow group of 0-cycles with
modulus and prove some intermediate results.

\subsection{Chow group with modulus}\label{sec:cycles}
We recall the definition of the Chow group of 0-cycles with modulus from
\cite{Binda-Krishna} and \cite{Kerz-Saito-2}.
Let $k$ be any field and $X$ a reduced quasi-projective
scheme over $k$ of dimension $d \ge 1$. Let $D \subset X$ be an effective
Cartier divisor and $U$ its complement. Let $D' = D_\red$.
Let $\sZ_0(U)$ denote the free abelian group
on the set of closed points in $U$. Suppose that $Y \subset X$ is an
integral curve not contained in $D'$ and $\nu \colon
Y_n \to X$ is the projection map from the normalization of $Y$.
We let $D'_{Y} = \nu^{-1}(D')$. Let $I_{D_Y}$ denote the ideal
of $\nu^*(D)$ in the semilocal ring $\sO_{Y_n, D'_Y}$. We have the divisor
map $\divf \colon  K^M_1(\sO_{Y_n, D'_Y}, I_{D_Y}) \to \sZ_0(\nu^{-1}(U))$. 
We let $\sR_0(Y|D)$ denote the image of the composite map
\begin{equation}\label{eqn:Rat-equiv}
K^M_1(\sO_{Y_n, D'_Y}, I_{D_Y}) \xrightarrow{\divf} \sZ_0(\nu^{-1}(U))
\xrightarrow{\nu_*} \sZ_0(U).
\end{equation}

We let $\sR_0(X|D)$ be the image of the map ${\underset{Y}\bigoplus} \
\sR_0(Y|D) \to \sZ_0(U)$, where $Y$ runs through the set of curves as above.
An element $\alpha \in \sR_0(X|D)$ is said to be a 0-cycle rationally
equivalent to zero.
The Chow group of 0-cycles with modulus is the quotient
\begin{equation}\label{eqn:CHM}
\CH_0(X|D) = \frac{\sZ_0(U)}{\sR_0(X|D)}.
\end{equation}

The functor $(X,D) \mapsto \CH_0(X|D)$ has appropriate covariant
functorial property for proper maps and contravariant functorial
property for flat maps. It is also clear from the definition that
for $D_1 \le D_2$, there is a canonical map $\CH_0(X|D_2) \to \CH_0(X|D_1)$.
In particular, there is a canonical map $\CH_0(X|D) \to \CH^F_0(X)$.
If $X$ is projective over $k$, composing with the classical degree map
$\CH^F_0(X) \to \Z$, we get the degree map
\begin{equation}\label{eqn:CHM-0}
\deg \colon \CH_0(X|D) \to \Z,  
\end{equation}
whose image coincides with the image of the composite map
$\sZ_0(U) \inj \sZ_0(X) \to \Z$. Hence, there exists an integer $n \ge 1$
which depends only on $U$ such that the sequence
\begin{equation}\label{eqn:CHM-1}
0 \to \CH_0(X|D)_0 \to \CH_0(X|D) \xrightarrow{{\deg}/n}  \Z \to 0
\end{equation}
is exact. Furthermore, we have $n = 1$ if either
$k$ is finite and $X$ is geometrically irreducible or $k$ is separably closed.
We shall consider $\CH_0(X|D)$ a topological group with discrete topology.

We let $C(U) = {\underset{E}\varprojlim} \ \CH_0(X|E)$
and $C(U)_0 = \Ker(C(U) \xrightarrow{\deg} \Z)$, where the limit is over
all effective Cartier divisors $E$ with support $D'$.
Since we have $\CH_0(X|D_2)_0 \surj \CH_0(X|D_1)_0$ for every $D_1 \le D_2$, 
it follows that 
\begin{equation}\label{eqn:Chow-deg-0}
C(U)_0 \xrightarrow{\cong} {\underset{E}\varprojlim} \
\CH_0(X|E)_0.
\end{equation}
We shall consider $C(U)$ and $C(U)_0$ as topological groups with
their inverse limit topology.
It is then clear that $C(U)_0$ is a closed subgroup of $C(U)$.

\vskip .2cm

The moving lemma we wish to prove for $\CH_0(X|D)$ is the following.
This result is of independent interest in the theory of cycles with 
modulus aside from its current application.

\begin{thm}\label{thm:MLC}
Let $X$ be a smooth quasi-projective scheme of dimension $d \ge 2$ 
over a perfect field $k$ and let $D \subset X$ be an effective Cartier divisor.
Let $A \subset D'$ be a closed subscheme such that $\dim(A) \le d-2$.
Then $\sR_0(X|D)$ is generated by the images of $\sR_0(Y|D)$ as in 
~\eqref{eqn:Rat-equiv}, where $Y \subset X$ has the additional
property that $Y \cap A = \emptyset$.
\end{thm}
 
\subsection{The Levine-Weibel Chow group of the double}
\label{sec:Inf*}
The proof of \thmref{thm:MLC} is similar to that
of the fundamental exact sequence \cite[(6.3)]{Binda-Krishna},
which relates the Chow group with modulus with an improved version of
the Levine-Weibel Chow group of a singular variety. 
However, we need to introduce several modifications to prove \thmref{thm:MLC} and shall
therefore present a complete proof.

Throughout the rest of \S~\ref{sec:ML}, we shall follow the notations and
assumptions of \thmref{thm:MLC}.
Recall from \cite[\S~2.1]{Binda-Krishna} that the double of 
$X$ along $D$ is the quasi-projective scheme $S_X$ such that 
the square
\begin{equation}\label{eqn:double-0}
\xymatrix@C2pc{
D \ar[r]^-{\iota} \ar[d]_-{\iota} & X \ar[d]^-{\iota_{+}} & \\
X \ar[r]^-{\iota_{-}} & S_X}
\end{equation}
is co-Cartesian, i.e., $S_X = X_+ \amalg_D X_-$.
%where $X_{\pm} := \iota_{\pm}(X) \inj X$. 
The arrows $\iota_\pm$ are closed immersions and
$\iota \colon D \inj X$ is the inclusion of $D$ in $X$. 
There is a projection map $\Delta \colon S_X \to X$ 
which is finite and flat. Moreover, $\Delta \circ \iota_\pm = {\rm Id}_X$.
We also have $D' = (S_X)_\sing$ via the inclusion $\iota_\pm \circ \iota$.
Let us denote this by $\iota'$. The double $S_X$ is a reduced scheme
with two irreducible components $X_\pm$, each a copy of $X$.

Let $\CH^{LW}_0(S_X)$ denote the Levine-Weibel Chow group of $S_X$
(see \cite[\S~3.4]{Binda-Krishna}). Recall that this is the quotient of
$\sZ_0(S_X \setminus D)$ by the subgroup of rational
equivalences, denoted by ${\sR^{LW}_0(S_X)}$.
This subgroup is generated by the divisors of rational functions 
on certain `Cartier curves on $X$ relative to $D$'. We refer to
\cite[\S~3.4]{Binda-Krishna} for these terms.

In this subsection, we shall prove some lemmas which
will allow us to choose some refined Cartier curves in order
to define the rational equivalence of 0-cycles on $S_X$.
We shall consider $A$ as a closed subscheme of $S_X$
via the inclusions $A \inj D \inj S_X$.
We shall assume in this subsection that $k$ is infinite and perfect.

\begin{lem}\label{lem:good-Cartier-curves-dim>2}
Assume $d \ge 3$. Then $\sR^{LW}_0(S_X)$ 
is generated by the divisors of functions on (possibly non-reduced) 
Cartier curves $C\inj S_X$ relative to $D$,
where $C$ satisfies the following:
\begin{enumerate}
\item 
There is a locally closed embedding
$S_X \inj \P^N_k$ and distinct hypersurfaces 
\[H_1, \cdots , H_{d-2} \inj \P^N_k\] 
such that 
$Y = S_X \cap H_1 \cap \cdots \cap H_{d-2}$ is a complete
intersection which is reduced;
\item
$X_{\pm} \cap Y = X_{\pm} \cap H_1 \cap \cdots \cap H_{d-2}:= Y_{\pm}$ are
integral;
\item
No component of $Y$ is contained in $D$;
\item
$Y \cap A$ is finite;
\item
$C \subset Y$;
\item
$C$ is a Cartier divisor on $Y$;
\item
$Y_{\pm}$ are smooth away from $C$.
\end{enumerate}
\end{lem}    
\begin{proof}
The proof is identical to that of \cite[Lemma~5.4]{Binda-Krishna}:
we obtain $Y$ and $C$ using the Bertini theorems of Altman-Kleiman
\cite[Theorem~1]{KL} and  Jouanolou \cite{Jou}.
Since $k$ is infinite, these Bertini theorems also allow us to ensure that
a general hypersurface of a given degree 
in $\P^n_k$ will intersect $A$ properly
under any chosen locally closed embedding $X \inj \P^N_k$.
Since (4) is an open condition on the linear system of hypersurfaces
of a given degree in $\P^N_k$, we can also include it
in the proof of \cite[Lemma~5.4]{Binda-Krishna} because $\dim(A) \le d-2$. 
\end{proof}

\begin{lem}\label{lem:reduction-basic}
Let $d \ge 2$. Let $\nu\colon C \inj S_X$ be a (possibly non-reduced)
Cartier curve relative to $D\subset S_X$. 
Assume that either $d = 2$ or there are inclusions 
$C \subset Y \subset S_X$, where
$Y$ is a reduced complete intersection surface and $C$ is a Cartier divisor on 
$Y$, as in \lemref{lem:good-Cartier-curves-dim>2}.
Let $f \in \cO_{C, \nu^*D}^{\times} \subset k(C)^\times$, where
$k(C)$ is the total quotient ring of $\sO_{C, \nu^*D}$. 

We can then find two Cartier curves $\nu'\colon C' \inj S_X$ and 
$\nu''\colon C'' \inj S_X$ relative to $D$ satisfying the following:
\begin{enumerate}
\item
There are very ample line bundles $\sL', \sL''$ on $S_X$ and sections
$t' \in H^0(S_X, \sL'), \ t'' \in H^0(S_X, \sL'')$ such that
$C' = Y \cap (t')$ and $C'' = Y \cap (t'')$ 
(with the convention $Y=S_X$ if $d=2$);
\item
$C'$ and $C''$ are reduced;
\item
$C' \cap A = C'' \cap A = \emptyset$;
\item
The restrictions of both $C'$ and $C''$ to $X$ via the two closed 
immersions $\iota_{\pm}$ are integral curves in $X$,
which are Cartier and smooth along $D$;
\item
There are functions $f' \in \sO_{C', (\nu') ^*D}^{\times}$ and
$f'' \in \sO_{C'', (\nu'') ^*D}^{\times}$ such that
$\nu'_*({\divf}(f'))  + \nu''_*({\divf}(f''))  = 
\nu_*({\divf}(f))$ in $\sZ_0(S_X \setminus D)$. 
\end{enumerate}
\end{lem}
\begin{proof}
In this proof, we shall assume that $Y = S_X$ if $d =2$.
In the latter case, a Cartier curve on $S_X$ along $D' = (S_X)_{\sing}$ must be  
an effective Cartier divisor on $S_X$. Hence, we can assume that 
$C$ is an effective Cartier divisor on $Y$ for any $d \ge 2$. 
Note that $Y = Y_+ \amalg_D Y_-$.

Since $Y$ is quasi-projective over $k$ and $C$ is an effective 
Cartier divisor on $Y$,
we can add some very
ample effective divisor to $C$ to get another effective
Cartier divisor $C'$ (see the proof of \cite[Lemma~1.4]{Levine-2})
on $Y$ such that
\begin{enumerate}
\item
$C \subset C'$;
\item
$\ov{C'\setminus C} \cap (C \cap D) = \emptyset$;
\item
$\sO_{Y}(C')$ is very ample on $Y$.
\end{enumerate}
%Note that $C'$ is not claimed to be reduced. 
Setting $g = f$ on $C$ and $g =1$ on $C' \setminus C$, we then see that
$g$ defines an element in $\sO^{\times}_{C',C' \cap D}$ such that
$\divf(g) = \divf(f)$. 
We can thus assume that $\sO_{Y}(C)$ is very ample on $Y$.

We now choose $t_0 \in H^0(Y, \sO_{Y}(C))$ such that $C = (t_0)$, where
for a line bundle $\sL$ and a meromorphic section $h$ of $\sL$, $(h)$
denotes the (effective) divisor of zeros of $h$.
Since $\sO_{Y}(C)$ is very ample, and $Y$ is reduced, we can find a new 
section $t_{\infty} \in H^0(Y, \sO_{Y}(C))$, sufficiently general
%(see the proof of \cite[Lemma~1.3]{Levine-2}) 
so that:
\begin{enumerate}
\item
$(t_{\infty})$ is reduced;
\item
$(t_{\infty}) \cap (t_0) \cap D = \emptyset$;
\item
$(t_{\infty}) \cap A = \emptyset$;
\item
$(t_{\infty})$ contains no component of $(t_0)$;
\item
$D$ contains no component of $(t_{\infty})$.
\end{enumerate}

Note that we can achieve (3) because $Y \cap A$ is finite 
by \lemref{lem:good-Cartier-curves-dim>2}.
Denote by $C_{\infty}$ the divisor $(t_{\infty})$. Notice that the
function $h = (1,f)$ is meromorphic on $C_{\infty} \cup C$ and regular
invertible in a neighborhood of $(C_{\infty} \cup C) \cap D$ by (2). 
Since $Y \cap A$ is a finite closed set of $S_X$ which
does not meet $C_\infty$, we can write $Y \cap A = S_1 \amalg S_2$,
where $S_1 \subset C \cap A$ and $S_2 \cap (C_\infty \cup C) =
\emptyset$.
By setting $h = 1$ on $S_2$, we can extend the function $h$ to a 
meromorphic function on
$T := (Y \cap A) \cup C_{\infty} \cup C$
which is regular invertible in a neighborhood of
$T' := (Y \cap A) \bigcup ((C_{\infty} \cup C) \cap D)$.
Let $S$ denote the finite
set of closed points on $C_{\infty} \cup C$, which either lie on
$C_{\infty} \cap C$ or where $h$ is not regular. In particular,
$S\supset \{\text{poles of $f$}\} \cup (C\cap C_{\infty})$.  

Since $S_X$ is reduced quasi-projective and $X_{\pm}$ (as well as $Y_{\pm}$)
are integral closed subschemes of $S_X$, we can find a very ample line
bundle $\sL$ on $S_X$ and a section $s_\infty \in H^0(S_X, \sL)$ 
(see again \cite[Lemma~1.4]{Levine-2}) so that: 
\begin{listabc}
\item $(s_{\infty})$ and $(s_\infty) \cap Y$ are reduced;
\item
$Y \not\subset (s_\infty)$;
\item
$(s_\infty) \cap X_{\pm}$ and $(s_\infty) \cap Y_{\pm}$ are integral;
\item
$(s_\infty) \cap T' = \emptyset$;
\item
$(s_\infty) \supset S$;
\item
$C_{\infty} \cup C$ contains no component of $(s_{\infty}) \cap Y$.
\end{listabc}

If $\ov{S_X}$ is the scheme-theoretic 
closure of $S_X$ in the projective embedding given by
$\sL = \sO_{S_X}(1)$ and if $\sI$ is the ideal sheaf of 
$\ov{S_X} \setminus S_X$ in the ambient projective space, then we can find a 
section $s'_{\infty}$ of the sheaf $\sI \otimes \sO_{\P^N_k}(m)$ for some
$m \gg 0$,
%See Altman-Kleiman
which restricts to a section $s_{\infty}$ on $S_X$ satisfying
the properties (a) - (f) on $S_X = \ov{S_X} \setminus V(\sI)$.
This implies in particular that 
$S_X \setminus (s_{\infty}) = \ov{S_X} \setminus (s'_{\infty})$ 
is affine. Set $\sL' = \sL^m$.

We now know that $Y_+$ is smooth away from $C$, and (d) tells us that 
$(s_\infty)$ intersects $Y_+$ along $D$ at only those points which are away 
from $C_\infty \cup C$. It follows that $(s_\infty) \cap Y_+ \cap D \subset
(s_\infty) \cap (Y_+)_{\rm reg}$, where $ (Y_+)_{\rm reg}$ denotes the smooth
locus of $Y_+$. On the other hand, we can use the Bertini 
theorem of Altman and Kleiman  
\cite[Theorem~1]{KL} to ensure that $(s_\infty) \cap (Y_+)_{\rm reg}$ is
smooth. The same holds for $Y_-$ as well.
We conclude that we can choose $\sL'$ and $s_\infty \in H^0(S_X, \sL')$ such 
that (a) - (f) above as well as the following hold. \\
\hspace*{1cm} g) \ $(s_{\infty}) \cap Y_{\pm}$ are smooth along $D$.
\\
\hspace*{1cm} h) \ $S_X \setminus (s_\infty)$ is affine.

Now, because $h$ is a meromorphic function on $T$ which is regular outside 
$S$, and the latter is contained in $(s_\infty)$, the property (h) 
implies that $h|_{T\setminus (s_\infty)}$ extends to a regular function $H$ on 
$U = S_X \setminus (s_{\infty})$. 
Since $H$ is a meromorphic function on $S_X$ which has poles only along
$(s_{\infty})$, it follows that $Hs^N_{\infty}$ is an element of 
$H^0(U, (\sL')^N)$ 
which extends to a section $s_0$ of $(\sL')^N$ on all of $S_X$, if we choose
$N \gg 0$. 

Since $s_{\infty}$ and $h$ are both invertible on $T$ in a neighborhood of $T'$,
%$C \cup C_{\infty}$ along $D$, we see that $s_0$ is invertible on $C \cup C_{\infty}$ along $D$.
we see that $s_0$ is invertible on $T$ in a neighborhood of $T'$.
In particular, $s_0 \notin H^0(S_X, (\sL')^N\otimes \sI_{T})$.
Applying \cite[Lemma~1.4]{Levine-2} and \cite[Theorem~1]{KL} 
(see their proofs),
we can thus find $\alpha \in H^0(S_X, (\sL')^N \otimes \sI_{T}) 
\subset H^0(S_X, (\sL')^N)$
such that $s'_0 := s_0 + \alpha$ has the following properties:
\begin{listabcprime}
\item $(s'_0)$ and $(s'_0) \cap Y$ are reduced;
\item
$Y \not\subset (s'_0)$;
\item
$(s'_0) \cap X_{\pm}$ and $(s'_0) \cap Y_{\pm}$ are integral;
\item
$(s'_0) \cap T' = \emptyset$;
\item
$C_{\infty} \cup C$ contains no component of $(s'_0) \cap Y$;
\item
$(s'_0) \cap Y_{\pm}$ are smooth along $D$.
\end{listabcprime}

We then have 
\[
\frac{s'_0}{s^N_{\infty}} = \frac{Hs^N_{\infty} + (\alpha s^{-N}_{\infty})
s^N_{\infty}}{s^N_{\infty}} = H + \alpha s^{-N}_{\infty} = H', \ (\mbox{say}).
\]

Since $\alpha$ vanishes along $C_\infty \cup C$ and $s_{\infty}$ is
invertible along $U$, it follows that $H'_{|{(C_\infty \cup C) \cap U}} =
H_{|{(C \cup C_{\infty}) \cap U}} = h_{|U}$. In other words, we have
${s'_0}/{s^N_{\infty}} = h$ as rational functions on $C_\infty \cup C$.
We can now compute:

\[
\nu_*({\rm div}(f)) = (s'_0) \cdot C - N (s_\infty) \cdot C
\]
\[ 
0 = {\rm div}(1) = (s'_0) \cdot C_{\infty} - N (s_\infty)\cdot C_{\infty}.
\]

Setting $(s^Y_{\infty}) = (s_\infty) \cap Y$ and $(s'^Y_0) = (s'_0) \cap Y$,
we get

\[
\begin{array}{lll}
\nu_*(\divf(f)) & = & (s'_0) \cdot (C - C_{\infty}) - 
N(s_{\infty})(C - C_{\infty}) \\
& = & (s'^Y_0) \cdot (\divf({t_0}/{t_{\infty}})) - N(s^Y_{\infty}) \cdot
(\divf({t_0}/{t_{\infty}})) \\
& = & \iota_{{s'^Y_0}, *}(\divf(f')) - N \iota_{{s^Y_{\infty}}, *}(\divf(f'')),
\end{array}
\]
where $f' = ({t_0}/{t_{\infty}})|_{(s'^Y_0)} \in 
\sO^{\times}_{(s'^Y_0), D \cap (s'^Y_0)}$ (by (d')) and 
$f'' = (t_0/t_{\infty})|_{(s^Y_\infty)} \in 
\sO_{(s^Y_\infty), D\cap (s^Y_\infty) }^{\times}$ (by (d)). 

It follows from (g) and (f') that 
$(s'^Y_{0})_{| X_+}, \ (s'^Y_{0})_{| X_-} , \ (s^Y_{\infty})_{| X_+}$ and 
$(s^Y_{\infty})_{| X_-}$ are all smooth along $D$.
Setting $\sL'' = (\sL')^N, \ t'' = s'_0$ and $t' = s_\infty$,
we see that 
the curves $C' = (t') \cap Y$ and $C'' = (t'') \cap Y$ together with the 
functions $f'$ and $f''$ satisfy the conditions of the Lemma.
\end{proof}

\subsection{The map $\tau^*_X$}\label{sec:LW-mod}
Let $\CH^{A}_0(X|D)$ be the quotient of the free abelian
group $\sZ_0(X \setminus D)$ on the closed points on $X \setminus D$
 by the subgroup of
rational equivalences $\sR^A_0(X|D)$ generated by the images of $\sR_0(Y|D)$ 
as in ~\eqref{eqn:Rat-equiv}, where $Y \subset X$ has an additional
property that $Y \cap A = \emptyset$.
It is clear that there is a canonical
surjection $\CH^{A}_0(X|D) \surj \CH_0(X|D)$.
Our goal is to prove that this is an isomorphism.

We let $\tau^*_X \colon \sZ_0(S_X \setminus D) \to \sZ_0(X \setminus D)$ be 
the map $\tau^*_X = \iota^*_+ - \iota^*_-$ under the embeddings
$X \xrightarrow{\iota_\pm} S_X$. We want to show that $\tau^*_X$
preserves the subgroups of rational equivalences.
We continue to assume that $k$ is infinite and perfect.

\begin{lem}\label{lem:map-diff-surfaces}  
Assume $d = 2$. Then the map 
$\tau_X^*$ descends to a group homomorphism 
$\CH^{LW}_0(S_X)\to \CH^{A}_0(X|D)$.
\end{lem}
\begin{proof}
The proof is essentially identical to that of 
\cite[Proposition~5.7]{Binda-Krishna},
but subtle changes are required at several places. So we provide the details.

We have shown in \lemref{lem:reduction-basic} that in order to
prove that $\tau^*_X$ preserves the subgroups of rational equivalences,
it suffices to show that $\tau^*_X(\divf(f)) \in \sR^A_0(X|D)$, where $f$ is a 
rational function (which is regular and invertible along $D$)
on a Cartier curve $\nu\colon  C \inj S_X$ that we can 
choose in the following way.
\begin{enumerate}
\item
There is a very ample line bundle $\sL$ on $S_X$ and sections 
$t \in H^0(S_X, \sL), t_{\pm} = \iota^*_{\pm}(t) \in H^0(X, \iota^*_{\pm}(\sL))$ 
such that
$C = (t)$ and $C_{\pm} = (t_{\pm})$.
\item
$C$ is a reduced Cartier curve of the form $C = C_+ \amalg_E C_-$,
where $E = \nu^*(D)$ such that $C_{\pm}$ are integral curves on 
$X$, none of which is contained in $D$, none of which meets $A$
and each of which is smooth along $D$
(see \cite[Remark~5.6]{Binda-Krishna}). 
\end{enumerate}

If $E = \emptyset$, then $C_\pm$ are two integral curves on $X$
away from $D$ and $\tau^*_X(\divf(f)) = \divf(f|_{C_+}) - \divf(f|_{C_-})
\in \sR^A_0(X|D)$.  We can thus assume that $E \neq \emptyset$.
Then (1) implies that 
\begin{equation}\label{eqn:map-diff-surfaces-0*}
(t_+)_{|D} = \iota'^*(t) = (t_-)_{|D},
\end{equation}
where recall that $\iota' = \iota_+ \circ \iota = \iota_- \circ \iota \colon  
D \inj S_X$ denotes the inclusion map. 

Let $(f_+, f_-)$ be the image of $f$ in $\sO^{\times}_{C_+, E} \times 
\sO^{\times}_{C_-,E} \inj k(C_+) \times k(C_-)$.
It follows from \cite[Lemma~2.2]{Binda-Krishna} that there is an exact sequence
\[
0 \to \sO_{C,E} \to \sO_{C_+,E} \times \sO_{C_-, E} \to \sO_E \to 0.
\]
In particular, we have 
\begin{equation}\label{eqn:map-diff-surfaces-0} 
(f_+)_{|E} = (f_-)_{|E} \in \sO^{\times}_E .
\end{equation}

Let us first assume that $C_+ = C_-$ as curves on $X$.
Let $C$ denote this curve and let $C_n$ denote its normalization. 
Let $\pi\colon C_n \to C \inj X$ denote the composite map.
Since $C$ is regular along $E$ by (2), we get 
$f_+, f_- \in \sO^{\times}_{C_n,E}$.
Setting $g := f_+ f^{-1}_- \in \sO^{\times}_{C_n,E}$, it
follows from ~\eqref{eqn:map-diff-surfaces-0} that
$g \in \Ker(\sO^{\times}_{C_n,E}  \to \sO^{\times}_E)$. 
Moreover, $\tau^*_X(\divf(f)) = \iota^*_+(\divf(f)) -
\iota^*_-(\divf(f)) = \divf(f_+) - \divf(f_-) = \pi_*(\divf(g))$. 
Since $C \cap A = \emptyset$ by (2), we conclude that
$\tau^*_X(\divf(f))$
dies in $\CH^{A}_0(X|D)$.

We now assume that $C_+ \neq C_-$. 
Since $C_+ \cap D = C_- \cap D = E$ as closed subschemes, 
we see that the support of $(C_+ \cup C_-) \cap D$ is same as $E_\red$,
where $C_+ \cup C_- \subset X$ is the closed subscheme defined by the
ideal sheaf $\sI_{C_+} \cap \sI_{C_-}$. Note that since $C_\pm$ are
integral, $C_+ \cup C_-$ is a reduced closed subscheme of $X$
with irreducible components $C_\pm$.

Since $A$ is a finite closed subscheme of $X$
and $(C_+ \cup C_-) \cap A = \emptyset$, the functions 
$f_\pm$ extend to meromorphic functions on
$T_\pm = A \cup C_\pm$
which are regular invertible in a neighborhood of 
$T' = A \cup E$
by letting $f_\pm = 1$ on $A$.
Let $S_\pm$ denote the set of closed points
on $C_\pm$, where $f_\pm$ have poles.
We let $T = T_+ \cup T_-$ and $S = S_+ \cup S_-$.
It is clear that $S \cap D = \emptyset$.

We now repeat the constructions in the proof of \lemref{lem:reduction-basic}
to find a very ample line bundle $\sL$ on $X$ and a section 
$s_{\infty} \in H^0(X, \sL)$ (see \cite[Lemma~1.4]{Levine-2} and
\cite[Theorem~1]{KL}) such that:
\begin{listabc}
\item $(s_{\infty})$ is integral (because $X$ is integral);
\item $(s_\infty) \cap T' = \emptyset$;
\item $(s_\infty) \supset S$;
\item $(s_{\infty}) \not\subset C_+ \cup C_-$;
\item $(s_{\infty})$ is smooth away from $S$; 
\item $X \setminus (s_\infty)$ is affine.
\end{listabc}

It follows that $f_{\pm}|_{T_\pm \setminus (s_\infty)}$ 
extend to regular functions $F_{\pm}$ on
$X \setminus (s_\infty)$. Since $F_{\pm}$ are meromorphic functions
on $X$ which have poles only along $(s_\infty)$, it follows that 
$F_{\pm}s^N_{\infty}$ are elements of $H^0(X \setminus (s_\infty), \sL^N)$ which 
extend to
sections $(s_0)_{\pm}$ of $\sL^N$ on all of $X$, if we choose $N \gg 0$.

Since the functions $s_\infty$ and $F_\pm$ are all meromorphic functions
on $X$ which are regular on $X \setminus (s_\infty)$, 
it follows that each of them 
restricts to a meromorphic function on $T$ which is regular on
$T \setminus (s_\infty)$ and 
$F_\pm|_{T\pm \setminus (s_\infty)} = f_\pm|_{T_\pm \setminus (s_\infty)}$.
Since $s_\infty$ and $F_\pm$ are invertible on $T_\pm$ in some neighborhood of
$T'$, we see that $(s_0)_{\pm}$ are invertible on $T_\pm$ in some
neighborhood of $T'$.
In particular, $(s_0)_{\pm} \notin H^0(X, \sL^N\otimes \sI_{T})$.
As before, using \cite[Lemma~1.4]{Levine-2} and \cite[Theorem~1]{KL},
we can moreover find $\alpha_{\pm} \in H^0(X, \sL^N \otimes \sI_{T}) \subset 
H^0(X, \sL^N)$
such that $(s'_0)_{\pm} := (s_0)_{\pm} + \alpha_\pm$ have the following 
properties:
\begin{listabcprime}
\item $((s'_0)_{\pm})$ are integral;
\item
$((s'_0)_{\pm}) \not\subset C_+ \cup C_-$; 
\item $((s'_0)_{\pm}) \cap T'  = \emptyset$;
\item
$((s'_0)_{\pm})$ are smooth away from $S$.
\end{listabcprime}

Let $H_{\pm}:= \frac{(s'_0)_{\pm}}{s^N_{\infty}} $. We then have 
\begin{equation}\label{eqn::map-diff-surfaces-10}
H_{\pm}= 
\frac{F_{\pm}s^N_{\infty} + (\alpha_{\pm} s^{-N}_{\infty})
s^N_{\infty}}{s^N_{\infty}} = F_{\pm} + \alpha_{\pm} s^{-N}_{\infty} .
\end{equation}

%Let $U = X \setminus ((s_\infty) \cup ((s'_0)_+) \cup ((s'_0)_-))$.
We can now find a dense open subscheme $U' \subset 
X \setminus ((s_\infty) \cup ((s'_0)_+) \cup ((s'_0)_-))$
which contains $T'$ and where $F_\pm, \alpha_\pm, s_\infty$ and 
$(s'_0)_\pm$ are all regular. 
In particular, $H_\pm$ are rational functions on $X$
which are regular on $U'$. 

Since $T' \subset U'$ and $T' \neq \emptyset$ (because $E \neq \emptyset$), 
it follows that $T_\pm \cap U'$ are dense open in $T_\pm$.
It follows that $s_\infty$ and $(s'_0)_\pm$ restrict to regular functions
on $T_\pm \cap U'$ which are invertible in a neighborhood of $T'$.
Since $\alpha_{\pm}$ vanish along $T$ and $s_{\infty}$ is
invertible on $U'$, it follows that ${H_{\pm}}_{|{T_\pm \cap U'}} =
{F_{\pm}}_{|{T_\pm\cap U'}}$. 
Since $F_\pm$ restrict to regular functions on $T_\pm \cap U'$ and
are invertible along $T'$, it follows that $H_\pm$ restrict to
regular functions on $T_\pm \cap U'$ and are invertible along $T'$.

As $H_+$ (resp. $H_-$) is regular on $U'$, it restricts to
a regular function on the dense open subset $C_- \cap U'$ 
(resp. $C_+ \cap U'$) of
$C_-$ (resp. $C_+$). Furthermore,  we have
\begin{equation}\label{eqn:map-diff-surfaces-1*}
{H_+}_{|E} = {F_+}_{|E} = {f_+}_{|E} \ =^{\dagger} \ {f_-}_{|E} 
= {F_-}_{|E} = {H_-}_{|E},
\end{equation}
where $\dagger$ follows from ~\eqref{eqn:map-diff-surfaces-0}.

We thus saw above that $H_+$ and $H_-$ are both regular functions on
$C_- \cap U'$ such that $H_- \neq 0$. 
In particular, ${H_+}/{H_-}$ is a rational function on $C_-$. 
Since $(s'_0)_+$ and $(s'_0)_-$ are both invertible functions on $C_- \cap U'$, 
it follows that the restriction of ${(s'_0)_+}/{(s'_0)_-}$ on $C_-$ is a 
rational function on $C_-$, which is regular and invertible on
the dense open $C_- \cap U'$.
On the other hand, we have 
\begin{equation}\label{eqn:map-diff-surfaces-2}
\frac{(s'_0)_+}{(s'_0)_-} = \frac{{(s'_0)_+}/{s^N_\infty}}
{{(s'_0)_-}/{s^N_\infty}} = \frac{H_+}{H_-}, 
\end{equation}
as rational functions on $X$.
In particular, $(s'_0)_+ \cdot {H_-} = {(s'_0)_-} \cdot {H_+}$
as regular functions on $U'$.
In particular, this identity holds after restricting these regular functions
to $C_- \cap U'$.
We thus get 
\begin{equation}\label{eqn:map-diff-surfaces-20}
\frac{(s'_0)_+}{(s'_0)_-} = \frac{{(s'_0)_+}/{s^N_\infty}}
{{(s'_0)_-}/{s^N_\infty}} = \frac{H_+}{H_-}, 
\end{equation}
as rational functions on $C_-$. Note that $H_-$ is non-zero on $C_-$.
Since $\frac{(s'_0)_+}{(s'_0)_-}$ restricts to a rational function on $C_-$
which is regular and invertible in the dense open $C_- \cap U'$, we conclude
that ${H_+}/{H_-}$ restricts to an identical rational function on $C_-$
which is  regular and invertible on $C_- \cap U'$. 

We now compute
\[
\begin{array}{lll}
\tau^*_X(\divf(f)) & = & \iota^*_+(\divf(f)) -
\iota^*_-(\divf(f)) \\
& = & 
\divf(f_+) - \divf(f_-) \\
& = & 
\left[((s'_0)_+) \cdot C_+ - (s^N_{\infty}) \cdot C_+\right]
- \left[((s'_0)_-) \cdot C_- - (s^N_{\infty}) \cdot C_-\right] \\
& = & \left[((s'_0)_+) \cdot C_+ - ((s'_0)_+) \cdot C_-\right]
+ \left[((s'_0)_+) \cdot C_- - ((s'_0)_-) \cdot C_-\right] \\
& & - \left[(s^N_{\infty}) \cdot C_+ - (s^N_{\infty}) \cdot C_-\right] \\
& = & \left[((s'_0)_+) \cdot (C_+ - C_-)\right] +
\left[C_- \cdot (((s'_0)_+) - ((s'_0)_-))\right] -
\left[(s^N_{\infty}) \cdot (C_+ - C_-)\right] \\
& = & ((s'_0)_+) \cdot (\divf({t_+}/{t_-})) + C_- \cdot 
(\divf({(s'_0)_+}/{(s'_0)_-})) - N(s_{\infty}) \cdot (\divf({t_+}/{t_-})) \\
& = &  ((s'_0)_+) \cdot (\divf({t_+}/{t_-})) + C_- \cdot 
(\divf({H_+}/{H_-})) - N(s_{\infty}) \cdot (\divf({t_+}/{t_-})).
\end{array}
\]

It follows from (b) and (c') that $t_{\pm}$ restrict to regular invertible
functions on $((s'_0)_+)$ and $(s_{\infty})$ along $D$.
We set $h_1 = {(\frac{t_+}{t_-})}_{| ((s'_0)_+)}, \ 
h_2 = {(\frac{H_+}{H_-})}_{| C_-}$ and
$h_3 =  {(\frac{t_+}{t_-})}_{| s_{\infty}}$.
Let $((s'_0)_+)_n \to ((s'_0)_+)$, $(C_-)_n \to C_-$ and
$(s_\infty)_n \to (s_\infty)$ denote the normalization maps.
Let $\nu_1\colon  ((s'_0)_+)_n \to X$, $\nu_2\colon  (C_-)_n  \to X$ and
$\nu_3\colon  (s_\infty)_n \to X$ denote the composite maps.
We now note that $((s'_0)_+), \ C_-$ and $(s_\infty)$ are all regular along
$D$ by (2), (e) and (d'). Furthermore, none of these meets $A$
by (2), (b) and (c'). It follows from
~\eqref{eqn:map-diff-surfaces-0*} and ~\eqref{eqn:map-diff-surfaces-1*}
that $(\nu_1)_*(h_1), \ (\nu_2)_*(h_2)$ and $(\nu_3)_*(h_3)$
die in $\CH^{A}_0(X|D)$.
We conclude that $\tau^*_X(\divf(f))$
dies in $\CH^{A}_0(X|D)$. In particular, $\tau^*_X$ descends to a
map $\tau^*_X\colon \CH^{LW}_0(S_X) \to \CH^{A}_0(X|D)$. 
This finishes the proof.
\end{proof}

The following result generalizes \lemref{lem:map-diff-surfaces} to
higher dimensions.

\begin{prop}\label{prop:map-diff-higher} 
Assume $d \ge 2$. Then the map 
$\tau_X^*$ descends to a group homomorphism 
$\tau_X^* \colon \CH^{LW}_0(S_X)\to \CH^{A}_0(X|D)$.
\end{prop}
\begin{proof}
We can assume $d \ge 3$ by \lemref{lem:map-diff-surfaces}.
Let $\nu\colon  C \inj S_X$ be a reduced Cartier curve relative to $D$ and let 
$f \in \sO^{\times}_{C, E}$, where $E = \nu^*(D)$.
By \lemref{lem:good-Cartier-curves-dim>2}, we can assume that 
there are inclusions $C \inj Y \inj S_X$ satisfying the conditions
(1) - (7) of \lemref{lem:good-Cartier-curves-dim>2}.
The only price we pay by doing so is that $C$ may no longer be reduced.
We achieve its reducedness using \lemref{lem:reduction-basic} as follows.

We replace $C$ by a reduced Cartier curve
(which we also denote by $C$) that is of the form given in
\lemref{lem:reduction-basic}. We shall now continue with the notations of
the proof of \lemref{lem:reduction-basic}.

We write $C = (t) \cap Y$, where $t \in H^0(S_X, \sL)$ such that
$\sL$ is a very ample line bundle on $S_X$.
Let $t_{\pm} = \iota^*_{\pm}(t) \in H^0(X, \iota^*_{\pm}(\sL))$ and let
$C_{\pm} = (t_{\pm}) \cap Y = (t_{\pm}) \cap Y_{\pm}$.
Let $\nu_{\pm}\colon  C_\pm \inj X$ denote the inclusions.
It follows from our choice of the section that $(t_{\pm})$ are
integral.
If $C_+ = C_-$, exactly the same argument as in the
case of surfaces applies to show that 
$\tau^*_X(\divf(f))$ dies in $\CH^{A}_0(X|D)$.
So we assume $C_+ \neq C_-$. We can also assume $E \neq \emptyset$.

Let $\Delta(C) = C_+ \cup C_-$ denote the scheme theoretic image
in $X$ under the finite map $\Delta$.
Since $X$ is smooth and integral, we can find a complete intersection
integral surface $Z \subset X$ satisfying the following:
\begin{enumerate}
\item
$Z\supset \Delta(C)$;
\item
$Z \cap A$ is finite;
\item 
$Z\cap (t_{\pm})$ are integral curves;
\item
$Z$ is smooth away from $\Delta(C)$.
\end{enumerate}
Set $t^Z_\pm = (t_\pm)_{|Z}$.
Since $C_\pm$ are integral and contained in $Z \cap (t_{\pm})$, it follows that
\begin{equation}\label{eqn:map-diff>2-0}
(t^Z_{\pm}) = C_\pm.
\end{equation}

We let $f_\pm = \nu_\pm^*(f)$.
As in the proof of \lemref{lem:map-diff-surfaces}, the functions 
$f_\pm$ extend to meromorphic functions on
$T_\pm = (Z \cap A) \cup C_\pm = (Z \cap A) \amalg C_\pm$
which are regular invertible in a neighborhood of 
$T' = (Z \cap A) \cup E$
by letting $f_\pm = 1$ on $Z \cap A$.
Moreover,  $f_+|_E = f_-|_E$, where $E = \nu^*(D)$.
Let $S_\pm$ denote the set of closed points
on $C_\pm$, where $f_\pm$ have poles.
We let $T = T_+ \cup T_-$ and $S = S_+ \cup S_-$.
It is clear that $S \cap D = \emptyset$.

We now choose another very ample line bundle $\sM$ on $X$ and 
$s_\infty \in H^0(X, \sM)$ (see the proof of \lemref{lem:reduction-basic})
such that:
\begin{romanlist}
\item $(s_\infty)$ is integral;
\item $(s_\infty)\cap Z$ and $(s_\infty) \cap (t_\pm)$ are proper and integral;
\item $(s_\infty) \supset S$;
\item $(s_\infty) \cap T' = \emptyset$;
\item $X \setminus (s_\infty)$ is affine;
\item $(s_\infty)$ is smooth away from $S$;
\item $(s_\infty) \cap Z$ is smooth away from $\Delta(C)$;
\item $(s_\infty)\cap Z \not\subset \Delta(C)$.
\end{romanlist}

As shown in the proof of \lemref{lem:reduction-basic}, it
follows from (3), (iv) and (vii) 
above that $(s^Z_\infty) := (s_\infty) \cap Z$ is smooth along $D$.
Using (v), we can lift $f_\pm \in k(T_\pm)^\times$ to regular 
functions $F_\pm$ on $X\setminus (s_\infty)$.
Using an argument identical to that given in the proof of 
\lemref{lem:map-diff-surfaces}, we can extend $(s_0)_{\pm} = s_\infty^N F_\pm$ 
(for some $N \gg 0$) 
to global sections $(s_0)_\pm$ of $\mathcal{M}^{N}$ on $X$ so that:
\begin{listabc}
\item
$((s_0)_\pm)$ and $((s_0)_\pm) \cap Z$ are integral;
\item
$((s_0)_\pm) \cap T' = \emptyset$;
\item
$((s_0)_\pm) \cap Z \not\subset \Delta(C)$;
\item
$((s_0)_\pm) \cap Z$ are smooth away from $\Delta(C)$.
\end{listabc}
As we argued in the proof of \lemref{lem:reduction-basic}, it follows from
(iv), (vii),  (viii), (c) and (d) that $(s^Z_\infty)$ and $((s^Z_0)_\pm) := 
((s_0)_\pm) \cap Z$ are smooth along $D$.

Setting $H_\pm = {(s_0)_\pm}/{s^N_\infty}$ and using the argument of the 
proof of \lemref{lem:map-diff-surfaces}, we get 
$H_\pm \in k(X)^{\times}$ and they restrict to rational functions on
$C_\pm$ which are regular and invertible along $D$.
Moreover, ${H_+}/{H_-}$ restricts to a rational function on $C_-$
which is regular and invertible along $D$.
Since
\begin{equation}\label{eqn:map-diff-surfaces-1}
{H_+}_{|E} = {F_+}_{|E} = {f_+}_{|E} \ =^{\dagger} \ {f_-}_{|E} 
= {F_-}_{|E} = {H_-}_{|E},
\end{equation}
where $\dagger$ follows from ~\eqref{eqn:map-diff-surfaces-0},
we have ${H_+}/{H_-} = 1$ on $E$.

We now compute
\[
\begin{array}{lll}
\tau^*_X(\divf(f)) & = & \iota^*_+(\divf(f)) -
\iota^*_-(\divf(f)) \\
& = & 
\divf(f_+) - \divf(f_-) \\
& = & 
\left[((s^Z_0)_+) \cdot C_+ - N (s^Z_{\infty}) \cdot C_+\right]
- \left[((s^Z_0)_-) \cdot C_- - N (s^Z_{\infty}) \cdot C_-\right] \\
& = & \left[((s^Z_0)_+) \cdot C_+ - ((s^Z_0)_+) \cdot C_-\right]
+ \left[((s^Z_0)_+) \cdot C_- - ((s^Z_0)_-) \cdot C_-\right] \\
& & - N \left[(s^Z_{\infty}) \cdot C_+ - (s^Z_{\infty}) \cdot C_-\right] \\
& = & \left[((s^Z_0)_+) \cdot (C_+ - C_-)\right] +
\left[C_- \cdot (((s^Z_0)_+) - ((s^Z_0)_-))\right] -
N \left[(s^Z_{\infty}) \cdot (C_+ - C_-)\right] \\
& {=}^{\dagger} & \left[((s^Z_0)_+) \cdot ((t^Z_+) - (t^Z_-))\right] +
\left[C_- \cdot (((s^Z_0)_+) - ((s^Z_0)_-))\right] \\
& & - N \left[(s^Z_{\infty}) \cdot ((t^Z_+) - (t^Z_-))\right] \\
& = & ((s^Z_0)_+) \cdot (\divf({t^Z_+}/{t^Z_-})) - C_- \cdot 
(\divf({(s_0)_+}/{(s_0)_-})) - N (s^Z_{\infty}) \cdot (\divf({t^Z_+}/{t^Z_-})) \\
& = &  ((s^Z_0)_+) \cdot (\divf({t^Z_+}/{t^Z_-})) - C_- \cdot 
(\divf({H_+}/{H_-})) - N (s^Z_{\infty}) \cdot (\divf({t^Z_+}/{t^Z_-})),
\end{array}
\]
where ${=}^{\dagger}$ follows from ~\eqref{eqn:map-diff>2-0}.

It follows from (iv) and (b) that ${t^Z_+}/{t^Z_-}$ restricts to  regular and
invertible functions on $((s^Z_0)_+)$ and $(s^Z_{\infty})$ along $D$.
Since $t \in H^0(S_X, \sL)$ and $t_{\pm} = \iota^*_{\pm}(t) \in 
H^0(X, \iota^*_{\pm}(\sL))$, it follows that  
$(t_+)_{|D} = \iota'^*(t) = (t_-)_{|D}$. 
In particular, ${({t^Z_+}/{t^Z_-})}_{|E} = 1$.
We have seen before that ${(\frac{H_-}{H_+})}_{| C_-}$ is a
regular and invertible function on $C_-$ along $D$ and
${(\frac{H_+}{H_-})}_{| E} = 1$.

We set $h_1 = {(\frac{t^Z_+}{t^Z_-})}_{| ((s^Z_0)_+)}, \ 
h_2 = {(\frac{H_+}{H_-})}_{| C_-}$ and
$h_3 =  {(\frac{t^Z_+}{t^Z_-})}_{| s^Z_{\infty}}$.
Let $((s^Z_0)_+)_n \to ((s^Z_0)_+)$, $(C_-)_n \to C_-$ and
$(s^Z_\infty)_n \to (s^Z_\infty)$ denote the normalization maps.
Let $\nu_1\colon  ((s^Z_0)_+)_n \to X$, $\nu_2\colon  (C_-)_n  \to X$ and
$\nu_3\colon  (s^Z_\infty)_n \to X$ denote the composite maps.
The curves $((s^Z_0)_+)$ and $(s^Z_\infty)$ are all smooth along
$D$, and $C_-$ is smooth along $D$ by \lemref{lem:reduction-basic}.
Since none of these curves meets $A$ by (iv), (b) and
\lemref{lem:reduction-basic}, we see that
$(\nu_1)_*(h_1), \ (\nu_2)_*(h_2)$ and $(\nu_3)_*(h_3)$ all die in
$\CH^{A}_0(X|D)$. It follows that $\tau^*_X(\divf(f))$ dies
in $\CH^{A}_0(X|D)$. This finishes the proof.
\end{proof}

\section{Proof of the  moving lemma}\label{sec:Move-2}
In this section, we shall finish the proof of \thmref{thm:MLC}.
We need to recall the definition of the Chow group of 0-cycles on singular
varieties introduced by Binda-Krishna \cite{Binda-Krishna}.
This is an improved version of the Levine-Weibel Chow group.

Let $k$ be a field and $X$ a reduced quasi-projective scheme 
of dimension $d \ge 2$ over $k$. A good curve over $X$ is a
reduced scheme $C$ of pure dimension one together with a finite map
$\nu \colon C \to X$ whose image is not contained in $X_\sing$ and 
which is a local complete intersection (l.c.i.) morphism (see 
\cite[\S~2.3]{Binda-Krishna}
for the definition of such a morphism) over a neighborhood of
$\nu(C) \cap X_\sing$ in $X$.
We let $\sR_0(X)$ be the subgroup of $\sZ_0(X_\reg)$ generated by
$\nu_*(\divf(f))$, where $\nu \colon C \to X$ is a good curve over $X$
and $f \in \sO^{\times}_{C, \nu^{-1}(X_\sing)}$.
We let $\CH_0(X) = {\sZ_0(X)}/{\sR_0(X)}$.
One can in fact consider only those good curves $C$ in this definition which
are regular away from $\nu^{-1}(X_\sing)$ (see \cite[Lemma~3.5]{Binda-Krishna}).
It is known that the identity map of $\sZ_0(X_\reg)$ induces a
surjection $\CH^{LW}_0(X) \surj \CH_0(X)$. In certain cases, 
this map is known to be an isomorphism 
if $k$ is algebraically closed. Otherwise, $\CH_0(X)$ has better
behavior than $\CH^{LW}_0(X)$.
Later in this section, we shall also prove a moving lemma for
$\CH_0(X)$.

\subsection{Factorization of $\tau^*_X$ through $\CH_0(S_X)$}
\label{sec:lci}
Let $k$ be a field and $X$ a smooth quasi-projective scheme
of pure dimension $d \ge 2$ over $k$. Let $D \subset X$ be an effective Cartier
divisor with $D' = D_\red$. 
Let $A \subset D$ be a closed subscheme such that
$\dim(A) \le d-2$. We shall 
now show that the map $\tau^*_X$ that we constructed in
\propref{prop:map-diff-higher} actually factors through $\CH_0(S_X)$.

Before we do this, recall that if $f\colon {X_1} \to {X_2}$ is a proper map
and $D_1$ (resp. $D_2$) is an effective Cartier divisor on ${X_1}$ 
(resp. ${X_2}$) such that $f^*(D_2) \le D_1$, then there is a push-forward
map $f_*\colon \CH_0({X_1}|D_1) \to \CH_0({X_2}|D_2)$
(see \cite[Lemma~2.7]{Binda-Saito} or \cite[Proposition~2.10]{KPv}).
But the existing proofs of this use a different (but equivalent) 
definition of the Chow group of 0-cycles with modulus
from the one presented in \S~\ref{sec:cycles}.
In particular, if we are given a closed subscheme 
$A_2 \subset D_2$ such that $\dim(A_2) \le \dim(X_2) - 2$ and 
$A_1 = f^{*}(A_2)$, 
then it is not immediately clear from our definition that
the push-forward map $f_* \colon \CH^{A_1}_0(X_1|D_1) \to \CH^{A_2}_0(X_2|D_2)$
exists. The next lemma shows that this map actually exists.

\begin{lem}\label{lem:PF-A}
The map $f_* \colon \sZ_0(X_1 \setminus D_1) \to \sZ_0(X_2 \setminus D_2)$
induces a push-forward map 
\[
f_* \colon \CH^{A_1}_0(X_1|D_1) \to \CH^{A_2}_0(X_2|D_2).
\]
\end{lem}
\begin{proof}
It suffices to consider the case when $D_1 = f^*(D_2)$.
Let $Y_1 \subset X_1$ be an integral curve not contained in $D_1$ and
not meeting $A_1$.
Let $\nu_1 \colon (Y_1)_n \to X_1$ be the induced finite map from the
normalization of $Y_1$. Let $Y_2 = f(Y_1)$. As $f$ is proper,
$Y_2 \subset X_2$ is closed and $Y_2 \cap A = \emptyset$. If
$Y_2$ is a closed point, the proof is straightforward.
So we assume $Y_2$ is an integral curve. Then it can not be
contained in $D_2$.
We let $\nu_2 \colon (Y_2)_n \to X_2$ be the induced map from the
normalization of $Y_2$. This gives rise to a finite dominant map
$f' \colon (Y_1)_n \to (Y_2)_n$. 

We let $g \in K^M_1(\sO_{(Y_1)_n, \nu^{-1}_1(D_1)}, I_{D_1})$.
We know from \cite[Chapter~1]{Fulton} that 
$f_*(\divf(g)) = \divf(N(g))$, where $N \colon K^M_1(k(Y_1)) \to
K^M_1(k(Y_2))$ is the norm map. So all we need to show to finish the
proof is that 
\[
N(K^M_1(\sO_{(Y_1)_n, \nu^{-1}_1(D_1)}, I_{D_1})) \subset
K^M_1(\sO_{(Y_2)_n, \nu^{-1}_1(D_2)}, I_{D_2}).
\]
But this is well known (e.g., see \cite[Lemma~6.19]{Raskind})
since $I_{D_1} = f^*(I_{D_2})$.
\end{proof}

Using \lemref{lem:PF-A}, we can prove the following result which will be
used later in this section. This shows that $\CH^A_0(X|D)$ can also
be defined in the style of the definition of $\CH_0(X|D)$ given 
in \cite{Binda-Saito}. For an integral curve $C \subset \P^1_X$ and a
point $t \in \P^1_k(k)$ such that $C \not\subset (X \times \{t\})$, 
we let $[C_t] = \pi_*([C] \cdot (X \times \{t\}))$,
where $\pi \colon \P^1_X \to X$ is the projection
and $[C] \cdot (X \times \{t\})$ is the 0-cycle associated to the
scheme theoretic intersection of $C$ and $X \times \{t\}$.

\begin{lem}\label{lem:PF-A-0}
Let $X$ and $D$ be as above. 
Let $\sR'_0(X|D) \subset \sZ_0(X \setminus D)$ be the subgroup
generated by the 0-cycles $[C_{0}] - [C_\infty]$, where
$C \subset \P^1_X$ is an integral curve 
satisfying the following properties:
\begin{enumerate}
\item
$C \cap \P^1_D$ is finite;
\item
$C \cap (D\times_k \{0, \infty\}) = \emptyset$;
\item
$C \cap \P^1_A = \emptyset$;
\item
The Weil divisor $\nu^*(X \times \{1\}) - \nu^*(\P^1_D)$ is effective,
where $\nu: C_n \to C \inj \P^1_X$ is the composite finite map.
\end{enumerate}
Then $\CH^A_0(X|D) = \coker(\sR'_0(X|D) \to \sZ_0(X \setminus D))$.
\end{lem}
\begin{proof}
Let $\sR^A_0(X|D) = \Ker(\sZ_0(X \setminus D) \surj \CH^A(X|D))$.
Let $C \subset X$ be an integral curve not contained in $D$ and
not meeting $A$. Let $\nu \colon C_n \to X$ be the induced map
and let $f \in K^M_1(\sO_{C_n, \nu^{-1}(D)}, I_D)$.
Taking the closure $\Gamma_f$ of the graph of $f \colon C \dasharrow \P^1_k$
in $\P^1_X$, one easily sees that $\divf(f) \in \sR'_0(X|D)$.

Conversely, suppose $C \inj \P^1_X$ is an integral curve defining
an element in $\sR'_0(X|D)$. Let $\pi \colon C_n \to \P^1_X \to X$
be the composite map.
The projection $C \to \P^1_k$ defines
an element $g \in k(C)^{\times}$ which lies in
$K^M_1(\sO_{C_n, \pi^{-1}(D)}, I_D)$  by (1) - (4).
In particular, $\divf(g) \in  \sR_0(C_n|\pi^*(D)) =
\sR^{\pi^{-1}(A)}_0(C_n|\pi^*(D))$.

We let $C'$ be the image of $C$ under the
projection to $X$. If $C'$ is a closed point, the proof is straightforward.
So we can assume that $C'$ is a curve, not contained in $D$ and
not meeting $A$ by (1) and (3). Moreover, projection to $X$ defines
a finite dominant map $C \to C'$. This in turn induces a finite dominant map
$f' \colon C_n \to C'_n$. Furthermore, $\pi$ is same as
the composition $C_n \xrightarrow{f'} C'_n \xrightarrow{\nu} X$.
Since $[C_{0}] - [C_\infty] = \nu_ * \circ f'_*(\divf(g)) =
(\nu \circ f')_*(\divf(g))$,
it suffices therefore to show that $\pi_*(\divf(g)) 
= (\nu \circ f')_*(\divf(g)) \in \sR^A_0(X|D)$.
But this follows directly from \lemref{lem:PF-A}.
\end{proof}

\begin{lem}\label{lem:factor-PN-projection} 
Assume that for every integer $n \ge 0$,
the map $\tau^*_{\P^n_X}$ descends to a well defined homomorphism 
\[
\tau^*_{\P^n_X} \colon \CH_0^{LW}(S_{\P^n_X}) \to \CH^{\P^n_A}_0(\P^n_X|\P^n_D).
\]
Then the map $\tau^*_X \colon \CH_0^{LW}(S_X) \to \CH^{A}_0(X|D)$ factors 
through 
$\CH_0(S_X)$, giving a well defined homomorphism
\[\tau_X^*\colon  \CH_0(S_X) \to \CH^{A}_0(X|D).\]
\end{lem}
\begin{proof}
Let $\delta\colon \sZ_0(S_X \setminus D)\to \CH^{A}_0(X|D)$ be the 
composition $\sZ_0(S_X \setminus D)\to \CH_0^{LW}(S_X)
\xrightarrow{\tau^*_X} \CH^{A}_0(X|D)$. 
We have to show that $\delta$ factors through $\CH_0(S_X)$. 
Using \cite[Lemma~3.5]{Binda-Krishna}, we have to show more precisely 
that $\delta( \nu_*( {\rm div}(f))) =0$ for every finite 
l.c.i. morphism $\nu\colon C\to S_X$ 
from a reduced curve $C$ whose image is not contained in $D$
and for every rational function $f$ on $C$ that is regular and invertible 
along $\nu^{-1}(D)$.

Since $\nu$ is a finite l.c.i. morphism, we can factor it as a
composition $\nu = \pi\circ \mu$, where 
$\mu\colon C\hookrightarrow \P^n_{S_X} = S_{\P^n_X}$ 
(using \cite[Proposition~2.3]{Binda-Krishna}) is a regular embedding 
(see \cite[Lemma~37.55.3]{SP}) and 
$\pi\colon \P^n_{S_X}\to S_X$ is the projection. In particular, $\mu(C) =C$ is 
a Cartier curve on the double $S_{\P^n_X}$. 
We let $Y = S_{\P^n_X}$ and $E = \P^n_D$.
It is then clear that $\pi^{-1}(A) = \P^n_A$ is a closed subscheme
of $Y_\sing$ and $\dim(\P^n_A) \le \dim(Y) - 2$.

It follows from the commutative diagram
\begin{equation}\label{eqn:iota-pi}
\xymatrix{
\sZ_0(Y \setminus E) \ar[r]^-{\iota_{\pm}^*} 
\ar[d]_{\pi_*} & \sZ_0(\P^n_X \setminus E) \ar[d]^{\pi_*}\\
 \sZ_0(S_X \setminus D) \ar[r]_{\iota^*_\pm} & \sZ_0(X \setminus D)
}
\end{equation}
and the formula
$\delta = \iota^*_+ - \iota^*_-$ that the square
\[
\xymatrix@C1pc{
\sZ_0(Y \setminus E) \ar[r]^{\delta_Y} \ar[d]_{\pi_*} & 
\CH^{\P^n_A}_0(\P^n_X|E) \ar[d]^{\pi_*} \\
\sZ_0(S_X \setminus D) \ar[r]_{\delta} & \CH^{A}_0(X|D)}
\]
commutes. That is,
$\delta(\nu_*(\divf(f))) = \delta (\pi_* (\mu_*(\divf(f)))) = 
\pi_*( \delta_{Y}(\mu_*(\divf(f))))$.
Note that the push-forward map $\pi_*$ on the right exists by 
\lemref{lem:PF-A}.

By assumption, we have 
$\delta_{Y}(\mu_*(\divf(f)))=0 \in \CH^{\P^n_A}_0(\P^n_X|E)$, 
where $\delta_{Y}$ is the composition
$\sZ_0(Y \setminus E) \to \CH_0^{LW}(Y) \xrightarrow{\tau^*_{\P^n_X}}
\CH^{\P^n_A}_0(\P^n_X|E)$.
Since the push-forward map $\pi_*\colon \CH^{\P^n_A}_0(\P^n_X|E) \to 
\CH^{A}_0(X|D)$ is well defined as we saw above, we are done.
\end{proof}

\begin{lem}\label{lem:lci-LW}
Let $k$ be an infinite perfect field. Then $\tau^*_X$ factorizes through
a homomorphism
\[
\tau^*_X \colon \CH_0(S_X) \to \CH^{A}_0(X|D).
\]
\end{lem}
\begin{proof}
Combine \propref{prop:map-diff-higher} and \lemref{lem:factor-PN-projection}.
\end{proof}

\subsection{Behavior of $\CH^A_0(X|D)$ under the change of 
base field}\label{sec:B-C}
Let $k$ now be any field and $X$ a smooth quasi-projective scheme
of pure dimension $d \ge 2$ over $k$. Let $D \subset X$ be an effective Cartier
divisor. Let $A \subset D$ be a closed subscheme such that
$\dim(A) \le d-2$.
We shall need the following base change property of $\CH^A_0(X|D)$.

\begin{prop}\label{prop:PF-fields-mod} 
Let $k \inj k'$ be a separable algebraic (possibly infinite)
extension of fields. 
Let $X' = X_{k'}, \ D' = D_{k'}$ and $A' = A_{k'}$
denote the base change of $X, D$ and
$A$, respectively. Let ${\rm pr}_{{k'}/{k}}: X' \to X$ be the
projection map. Then the following hold.
\begin{enumerate}
\item
There exists a pull-back 
${\rm pr}^*_{{k'}/{k}}: \CH^A_0(X|D) \to \CH^{A'}_0(X'|D')$.
\item
If there exists a sequence of separable 
field extensions $k = k_0 \subset k_1 \subset
\cdots \subset k'$ with $k' = \cup_i k_i$, then
we have ${\underset{i}\varinjlim} \ \CH^{A_{k_i}}_0(X_{k_i}|D_{k_i}) 
\xrightarrow{\simeq} \CH^{A'}_0(X'|D')$. 
\item
If $k \inj k'$ is finite, then there exists a push-forward
${\rm pr}_{{k'}/{k} \ *}: \CH^{A'}_0(X'|D') \to \CH^A_0(X|D)$ such that
${\rm pr}_{{k'}/{k} \ *} \circ {\rm pr}^*_{{k'}/{k}}$ is multiplication by
$[k': k]$.
\end{enumerate}
\end{prop}
\begin{proof}
Let $x \in X \setminus D$ be a closed point. Since
${\rm pr}_{{k'}/{k}}$ is smooth (ind-smooth to be precise),
it follows from our hypothesis that 
${\rm pr}^*_{{k'}/{k}}([x])$ is a well defined 0-cycle in
$\sZ_0(X'|D')$. We thus have a pull-back map
${\rm pr}^*_{{k'}/{k}}: \sZ_0(X \setminus D) \to \sZ_0(X' \setminus D')$.
Let $\sR'_0(X|D) \subset \sZ_0(X \setminus D)$ be
as in \lemref{lem:PF-A-0}.
To show that ${\rm pr}^*_{{k'}/{k}}$ preserves rational equivalence,
it suffices to show that it takes $\sR'_0(X|D)$ to $\sR'_0(X'|D')$
by \lemref{lem:PF-A-0}.

Let $C \subset X \times_k \P^1_k$ be an integral curve as in 
\lemref{lem:PF-A-0}.
It follows from the smoothness of ${\rm pr}_{{k'}/{k}}$
that $C' = {\rm pr}^*_{{k'}/{k}}(C) = C_{k'} \inj 
(X \times_k \P^1_k)_{k'} = X' \times_{k'} \P^1_{k'}$ is a 
reduced curve whose irreducible components $C'_1, \ldots , C'_r$
satisfy conditions (1)-(4) of \lemref{lem:PF-A-0}. 
Furthermore, the flat pull-back property of Bloch's cycle complex
(see \cite[Proposition~1.3]{Bloch86}) says that
\[
{\rm pr}^*_{{k'}/{k}}([C_{0}] - [C_\infty]) = 
[C'_{0}] - [C'_\infty] = \
\stackrel{r}{\underset{i =1}\sum} ([(C'_i)_0] - [(C'_i)_\infty]).
\]
In particular, ${\rm pr}^*_{{k'}/{k}}([C_{0}] - [C_\infty])$ dies in
$\CH^{A'}_0(X'|D')$. This proves (1).

It is clear that the map ${\underset{i}\varinjlim} \ \CH^{A_i}_0(X_i|D_i) \to
\CH^{A'}_0(X'|D')$ is surjective. 
To show injectivity, suppose there is some $i \ge 0$
and $\alpha \in \sZ_0(X_i|D_i)$ such that
${\rm pr}^*_{{k'}/{k_i}}(\alpha) \in \sR'_0(X'|D')$.
We can replace $k$ by $k_i$ and assume $i = 0$.

Let $C^j \inj X' \times_{k'} \P^1_{k'} = (X \times_k \P^1_k)_{k'}$ for
$j = 1, \cdots , r$ be a collection of integral curves as in the proof of (1)
so that ${\rm pr}^*_{{k'}/{k}}(\alpha) =
(\stackrel{r}{\underset{j =1}\sum} \ [C^j_ {0}]) -
(\stackrel{r}{\underset{j =1}\sum} \ [C^j_ {\infty}])$.
Let $\nu^j: C^{j}_n \to X' \times_{k'} \P^1_{k'}$ denote the maps from
the normalizations of the above curves.

We can then find some $i \gg 0$ and integral curves
$W^j \inj X_i \times_{k_i} \P^1_{k_i}$
such that $C^j = W^j \times_{k_i} k'$ for each $j =1, \cdots , r$.
In particular, we have $C^j_0 = 
{\rm pr}^*_{{k'}/{k_i}}(W^j_0)$ and
$C^j_{\infty} = {\rm pr}^*_{{k'}/{k_i}}(W^j_\infty)$ for $j = 1, \cdots , r$.
Since ${\rm pr}_{{k'}/{k_i}}$ is smooth, it follows that 
$C^{j}_n = (W^{j}_n)_{k'}$ for each $j$.
Moreover, it follows from \cite[Lemma 2.2]{KPv} that condition (4) above
holds on each $W^{j}_n$.
It follows that each $W^j$ defines a rational equivalence for 
0-cycles with modulus $D_i$ on $X_i$.

We now set $\alpha_i = {\rm pr}^*_{{k_i}/{k}}(\alpha)$
and let $\beta = \alpha_i - (\stackrel{r}{\underset{j =1}\sum} 
([W^j_0] - [W^j_\infty])) \in \sZ_0(X_i \setminus D_i)$.
It then follows that 
${\rm pr}^*_{{k'}/{k_i}}(\beta) = {\rm pr}^*_{{k'}/{k_i}}(\alpha_i)
- \stackrel{r}{\underset{j =1}\sum} 
([C^j_0] - [C^j_\infty]) 
= {\rm pr}^*_{{k'}/{k}}(\alpha) - \stackrel{r}{\underset{j =1}\sum} 
([C^j_0] - [C^j_\infty]) = 0$ in
$\sZ_0(X' \setminus D')$.
Since the map ${\rm pr}^*_{{k'}/{k_i}}: \sZ_0(X_i \setminus D_i) \to
\sZ_0(X' \setminus D')$ of free abelian groups is clearly injective,
we get $\beta = 0$, which means that $\alpha_i \in
\sR_0(X_i|D_i)$. This proves (2).
The existence of push-forward follows from \lemref{lem:PF-A}
and the formula ${\rm pr}_{{k'}/{k} \ *} \circ 
{\rm pr}^*_{{k'}/{k}} = [k': k]$ is obvious from the definitions.
This proves (3).
\end{proof}

\subsection{Proof of \thmref{thm:MLC}}\label{sec:prf-mlc}
We shall now prove \thmref{thm:MLC}. This is equivalent to showing
that the canonical surjection
$\CH^A_0(X|D) \surj \CH_0(X|D)$ is also a monomorphism.
Using \propref{prop:PF-fields-mod} and the 
standard pro-$\ell$ extension trick, we easily reduce to the case when
$k$ is infinite and perfect. 
We assume this to be the case in the rest of the proof.

By \lemref{lem:lci-LW}, the map $(\iota^*_+ - \iota^*_-) \colon
\sZ_0(S_X \setminus D) \to \sZ_0(X \setminus D)$ defines
a homomorphism $\tau^*_X \colon \CH_0(S_X) \to \CH^A_0(X|D)$.
We now define two maps in opposite direction,
\[
p_{\pm,*} \colon   \sZ_0(X \setminus D) 
\rightrightarrows \sZ_0(S_X \setminus D)
\] 
by $p_{+,*}([x]) = \iota_{+, *}([x])$ 
(resp.~ by $p_{-,*}([x]) = \iota_{-, *}([x])$) for a closed point 
$x \in X \setminus D$. 
Concretely, the two maps $p_{+,*}$ and $p_{-,*}$ copy a cycle $\alpha$ in one 
of the two components of the double $S_X$ (the $X_+$ or the $X_-$ copy). 
Since $\alpha$ is supported outside $D$ (by definition of $\sR_0(X|D)$), 
the cycles $p_{+,*}(\alpha)$ and  
$p_{-,*}(\alpha)$ give classes in $\CH_0(S_X)$.
By \cite[Proposition~5.9]{Binda-Krishna}, the maps $p_{\pm,*}$ descend to 
group homomorphisms 
$p_{\pm,*}\colon \CH_0(X|D) \to \CH_0(S_X)$.
Composing with the canonical surjection $\CH^A_0(X|D) \surj \CH_0(X|D)$,
we get maps
$p^A_{\pm,*} \colon \CH^A_0(X|D) \to \CH_0(S_X)$.
Furthermore, it is clear from the definitions of
$p^A_{\pm,*}$ and $\tau^*_X$ that $\tau^*_X \circ p^A_{\pm,*} = {\rm Id}$.
We thus get maps
\[
\CH^A_0(X|D) \surj \CH_0(X|D) \xrightarrow{p_{\pm,*}} \CH_0(S_X)
\xrightarrow{\tau^*_X} \CH^A_0(X|D),
\]
whose composition is identity. It follows that the first arrow from the
left is injective. This finishes the proof.
\qed

%$\hfill\square$

\subsection{Moving lemma for $\CH_0(X)$}\label{sec:ML-lci}
We shall end this section with a moving lemma for the Chow group
of 0-cycles $\CH_0(X)$ for a singular scheme $X$.
This result is of independent interest in the study of 0-cycles and
algebraic $K$-theory on singular varieties. 
We remark that this moving lemma for $\CH^{LW}_0(X)$ is yet unknown 
over finite fields.

Let $k$ be a field and $X$ a reduced quasi-projective scheme of
dimension $d \ge 2$ over $k$.  
Suppose that $A \subset X_\sing$ is a closed subscheme
such that $\dim(A) \le d-2$. We can then define $\CH^A_0(X)$ by repeating
the definition of $\CH_0(X)$, except that we allow only those good curves
$\nu \colon C \to X$ which satisfy the additional condition that
$\nu(C) \cap A = \emptyset$.
We clearly have a canonical surjection $\CH^A_0(X) \surj \CH_0(X)$.

\begin{prop}\label{prop:lci-iso}
Let $k$ be any perfect field. Then the map
$\CH^{A}_0(X) \to  \CH_0(X)$ is an isomorphism.
\end{prop}
\begin{proof}
We only need to show injectivity. 
To reduce it to the case when $k$ is infinite,
we note that an analogue of \propref{prop:PF-fields-mod}
is proven for $\CH_0(X)$ in \cite[Proposition~6.1]{Binda-Krishna}.
The reader may check that this proof works verbatim for
$\CH^A_0(X)$. This allows us to use the pro-$\ell$ extension trick
as we did in \thmref{thm:MLC} to reduce the proof of the
proposition when $k$ is infinite and perfect.

Let $\nu \colon C \to X$ be a
good curve over $X$ and let $f$ be a rational function on
$C$ which is regular and invertible in a neighborhood of $\nu^{-1}(X_\sing)$.
It suffices to show that $\nu_*(\divf(f))$ belongs to the subgroup
$\sR^{A}_0(X)$ of rational equivalences that define $\CH^{A}_0(X)$.

First of all, we can assume by \cite[Lemma~3.5]{Binda-Krishna} that
$\nu \colon C \to X$ is an l.c.i. morphism.
As a consequence, we have a factorization $\nu = \pi\circ \mu$, where 
$\mu\colon C\hookrightarrow \P^n_{X}$ is a regular embedding and 
$\pi\colon \P^n_{X}\to X$ is the projection. In particular, $\mu(C) = C$ is 
a Cartier curve on $\P^n_X$. 
The smoothness of $\pi$ implies that $(\P^n_X)_\sing =
\P^n_{X_\sing}$. Furthermore, $\pi^{-1}(A) = \P^n_A$ is a closed subscheme
of $(\P^n_X)_\sing$ and $\dim(\P^n_A) \le \dim(\P^n_X) - 2$.
It is also clear that $\nu_*(\divf(f)) = \pi_* (\mu_*(\divf(f)))$.

We can now apply \cite[Lemma~1.3]{ESV} to find a reduced Cartier
curve $\mu' \colon C' \inj \P^n_X$ such that $C' \cap \P^n_{A} = \emptyset$,
and a rational function $f'$ on $C'$ which is regular and invertible along
$C' \cap \P^n_{X_\sing}$ and $\mu_*(\divf(f)) = \divf(f')$. 
This can also be easily deduced from the proof of  
\propref{prop:map-diff-higher}.  
We let $\nu' \colon C' \xrightarrow{\mu'} \P^n_X \xrightarrow{\pi} X$ denote 
the composite map. Since $\mu'$ is a regular immersion along $\P^n_{X_\sing}$,
it follows that $\nu'$ is an l.c.i. morphism along $X_\sing$.
In particular, $\nu' \colon C' \to X$ is a good curve over $X$.
Furthermore, we have $\nu'(C') \cap A = \pi(C' \cap \P^n_A) = 
\emptyset$. Since $\nu'_*(\divf(f')) = \pi_*(\divf(f')) =
\pi_* \circ \mu_*(\divf(f)) = \nu_*(\divf(f))$,
it follows that $\divf(f)$ dies in $\CH^A_0(X)$.
This finishes the proof.
\end{proof}

\section{Comparison of the idele class groups}
\label{sec:BF}
We shall work under the following assumptions in this section.
We let $k$ be a finite field and $X$ a normal projective integral 
scheme of dimension $d \ge 1$ over $k$. 
Let $C \subset X$ be a reduced closed subscheme of pure codimension one
whose complement $U$ is regular. We let $A = C_\sing$ with reduced
induced closed subscheme structure. Note that $\dim(A) \le d-2$.
In particular, $A = \emptyset$ when $d =1$.
Recall that $\Div_C(X)$ is the directed set of closed subschemes 
of $X$ with support $C$. Let $K$ denote the function field of $X$.
We fix a separable closure $\ov{K}$ of $K$ and let $K^{\ab} \subset \ov{K}$ be 
the maximal abelian extension of $K$.

We fix a closed subscheme $D \in \Div_C(X)$.
If $x \in U$ is a closed point, then it defines a Parshin chain on 
$(U \subset X)$ of length zero. Hence, we have the injection
$\Z = K^M_0(k(x)) \inj I_{U/X}$. Extending it linearly over $U_{(0)}$,
we see that there is canonical inclusion $\sZ_0(U) \inj I_{U/X}$. Composing with
the canonical maps $I_{U/X} \surj C_{U/X} \to \wt{C}_{U/X}$, we get a map 
\begin{equation}\label{eqn:CHM-2}
\cyc_{U/X} \colon \sZ_0(U) \to \wt{C}_{U/X}.
\end{equation}
Composing with the surjection $\wt{C}_{U/X} \surj C(X,D)$, we get our
cycle class map
\begin{equation}\label{eqn:CHM-3}
\cyc_{X|D} \colon \sZ_0(U) \to C(X,D).
\end{equation}
Recall that Bloch's formula for the 0-cycles with modulus asks whether
the following hold if $D$ is an effective Cartier divisor. 
\begin{enumerate}
\item
$\cyc_{X|D}$ annihilates the subgroup of rational equivalences.
\item
The resulting map $\cyc_{X|D} \colon \CH_0(X|D) \to C(X,D)$ is
an isomorphism.
\end{enumerate}
These conditions together 
are equivalent to asking whether the idele class groups 
$\CH_0(X|D)$ and $C_{KS}(X,D)$ are isomorphic.
The goal of this section is to prove this to be the case when $X$
is regular.

\subsection{Reciprocity maps for 0-cycles}\label{sec:CCM}
At any rate, we see using ~\eqref{eqn:CHM-1} and 
\cite[Proposition~4.8, Lemma~8.4]{Gupta-Krishna-REC} that we have a 
commutative diagram of short exact sequences
of topological abelian groups
\begin{equation}\label{eqn:CHM-4}
\xymatrix@C2pc{
0 \ar[r] & \sZ_0(U)_0 \ar[r] \ar[d] & \sZ_0(U) \ar[r]^-{{\deg}/n} 
\ar[d]^-{\cyc_{X|D}} & \Z \ar@{=}[d] \ar[r] & 0 \\
0 \ar[r] & C(X,D)_0 \ar[r] \ar[d] & C(X,D) \ar[r]^-{{\deg}/n} 
\ar[d]^-{\rho_{X|D}} & \Z \ar[r] \ar@{^{(}->}[d] & 0 \\
0 \ar[r] & \pi^{\adiv}_1(X,D)_0 \ar[r] & \pi^{\adiv}_1(X,D) \ar[r]^-{{\deg}/n} &
\wh{\Z} \ar[r] & 0,}
\end{equation} 
where $n \ge 1$ is an integer depending only on $U$.

We also have a commutative diagram of short exact sequences of
topological abelian groups (see \S~\ref{sec:Curve-based} for the definition of
$D_\rho$)
\begin{equation}\label{eqn:CHM-5}
\xymatrix@C2pc{
0 \ar[r] & \sZ_0(U)_0 \ar[r] \ar[d] & \sZ_0(U) \ar[r]^-{{\deg}/n} 
\ar[d]^-{\cyc_{X|D_\rho}} & \Z \ar@{=}[d] \ar[r] & 0 \\
0 \ar[r] & C(X,D_\rho)_0 \ar[r] \ar[d] & C(X,D_\rho) \ar[r]^-{{\deg}/n} 
\ar[d]^-{\rho^c_{X|D}} & \Z \ar[r] \ar@{^{(}->}[d] & 0 \\
0 \ar[r] & \pi^{\ab}_1(X,D)_0 \ar[r] & \pi^{\ab}_1(X,D) 
\ar[r]^-{{\deg}/n} & \wh{\Z} \ar[r] & 0.}
\end{equation} 

Let $\vartheta_{U/X} = \wt{\rho}_{U/X} \circ \cyc_{U/X}, \ \vartheta_{X|D} =
\rho_{X|D} \circ \cyc_{X|D}$ and $\vartheta^c_{X|D} = \rho^c_{X|D} \circ
\cyc_{X|{D_\rho}}$.
It is then clear from the definition of the reciprocity map
$\rho_{U/X}$ (see \S~\ref{sec:Rec-00}) and the cycle class map $\cyc_{U/X}$
that the composition $\vartheta_{U/X}$ is the Frobenius substitution
which takes a closed point $x \in U$ to the image of the Frobenius
automorphism under the canonical map 
$\Gal({\ov{k(x)}}/{k(x)}) \to \pi^{\ab}_1(U)$.
Hence, the same holds for $\rho_{X|D}$ and $\vartheta^c_{X|D}$.

It follows from the classical ramified class field theory for curves
that for every effective Cartier divisor $D \in \Div_C(X)$, the map 
$\vartheta^c_{X|D}$ factors
through rational equivalences and defines a map (obviously continuous)
$\vartheta^c_{X|D} \colon \CH_0(X|D) \to \pi^{\ab}_1(X,D)$ 
(see \cite[Proposition~3.2]{Kerz-Saito-2} and 
\cite[\S~8]{BKS}).
Taking the limit and using \propref{prop:Fil-c-prop}, we see that
$\vartheta_{U/X}$ descends to a continuous homomorphism
\begin{equation}\label{eqn:CHM-6}
\vartheta_{U/X} \colon C(U) \to \pi^{\ab}_1(U).
\end{equation}
It follows from the generalized Chebotarev-Lang density theorem 
(e.g., see \cite[Theorem~5.8.16]{Szamuely}) that
this map has dense image.

\subsection{Bloch's formula for regular schemes}
\label{sec:Reg-BF}
Throughout this subsection, we shall assume that $X$ is regular.
Note that every $D \in \Div_C(X)$ is then an effective Cartier divisor on $X$.
Under the extra assumption, 
we would like to work with a new subgroup of $H^1(K)$ which we introduce below. 

\begin{defn}\label{defn:mod-fil}
Let $\Fil^A_D H^1(K)$ be the subgroup of characters $\chi \in H^1(U)$ such that
for every integral curve $Y \subset X$ not contained in $C$ and not meeting 
$A$, the finite map $\nu \colon Y_n \to X$ has the property that
the image of $\chi$ under $\nu^* \colon H^1(U) \to H^1(\nu^{-1}(U))$
lies in $\Fil_{\nu^*(D)} H^1(k(Y))$, where $Y_n$ is the normalization of $Y$
(see Definition~\ref{defn:Fid_D}). 
We let $\pi^{A}_1(X,D)$ denote the Pontryagin dual of $\Fil^A_D H^1(K)$
and let $\pi^{A}_1(X,D)_0$ be the kernel of the map
$\pi^{A}_1(X,D) \to \Gal({\ov{k}}/k)$. 
\end{defn}
We have the canonical inclusions
\begin{equation}\label{eqn:mod-fil-0}
\Fil^c_D H^1(K) \subset \Fil^A_D H^1(K) \subset H^1(U) \subset H^1(K).
\end{equation}
These give rise to the surjective continuous homomorphisms of 
profinite groups
\begin{equation}\label{eqn:mod-fil-1}
\Gal({K^{\ab}}/K) \surj \pi^{\ab}_1(U) \surj \pi^{A}_1(X,D) \surj
\pi^{\ab}_1(X,D).
\end{equation}

\begin{lem}\label{lem:Cyc-A}
The composite map $C(U) \xrightarrow{\vartheta_{U/X}} \pi^{\ab}_1(U) \surj
\pi^{A}_1(X,D)$ factors through a continuous homomorphism
\[
\vartheta^A_{X|D} \colon \CH^A(X|D) \to \pi^{A}_1(X,D).
\]
\end{lem}
\begin{proof}
Since $C(U)$ is the limit of the pro-abelian group
$\{\CH^A_0(X|D)\}_{D \in \Div_C(X)}$ by \thmref{thm:MLC},
it suffices to show that the composite map
$\sZ_0(U) \to \pi^{\ab}_1(U) \surj \pi^{A}_1(X,D)$ kills
$\sR^A_0(X|D)$. Equivalently, we need to show that for every 
character $\chi \in \Fil^A_D H^1(K)$, the induced character
$\vartheta^*_{U/X}(\chi)$ annihilates $\sR^A_0(X|D)$.

Let $Y \subset X$ be an integral curve not contained in $D$ and
not meeting $A$ and let $\nu \colon Y_n \to X$ be the induced finite map.
Let  $V = \nu^{-1}(U)$ and $E = \nu^*(D)$.
Then we get a commutative diagram
\begin{equation}\label{eqn:Cyc-A-0}
\xymatrix@C.8pc{
\Fil^A_D H^1(U) \ar[r]^-{\vartheta^*_{U/X}} \ar[d]_-{\nu^*} & 
\Hom_{\cont}(\sZ_0(U), {\Q}/{\Z}) \ar[d]^-{\nu^*} \\
\Fil_E H^1(V) \ar[r]^-{\vartheta^*_{V/{Y_n}}} & \Hom_{\cont}(\sZ_0(V), {\Q}/{\Z}),}
\end{equation}
where the right vertical arrow is induced by the push-forward map
$\nu_* \colon \sZ_0(V) \to \sZ_0(U)$ and the left vertical arrow is
induced by the definition of $\Fil^A_D H^1(U)$.
We need to show that $\vartheta^*_{V/{Y_n}} \circ \nu^*(\chi)$ annihilates 
$\sR_0(Y_n|E)$.
But this follows from the classical ramified class field theory for curves 
(e.g., see \cite{Serre-CF}).
\end{proof}

It is clear from the above definitions that there is a commutative 
diagram
\begin{equation}\label{eqn:Cyc-A-1}
\xymatrix@C.8pc{
C(U) \ar@{->>}[r] \ar[d]_{\vartheta_{U/X}} & \CH^A_0(X|D) 
\ar[d]^-{\vartheta^A_{X|D}} \ar@{->>}[r] & \CH_0(X|D) 
\ar[d]^-{\vartheta^c_{X|D}} \\
\pi^{\ab}_1(U) \ar@{->>}[r] & \pi^A_1(X,D) \ar@{->>}[r] &
\pi^{\ab}_1(X,D).}
\end{equation}

\begin{lem}\label{lem:Main*-1}
The map $\vartheta^A_{X|D} \colon \CH^{A}_0(X|D) \to \pi^{A}_1(X,D)$
is injective with dense image. Moreover, the induced
map $\vartheta^A_{X|D} \colon \CH^{A}_0(X|D)_0 \to \pi^{A}_1(X,D)_0$
is an isomorphism.
\end{lem}
\begin{proof}
The density of the image follows from ~\eqref{eqn:CHM-6}. 
We show the other assertions.
We first show that $\vartheta^A_{X|D} \colon \CH^{A}_0(X|D)_0 \to 
\pi^{A}_1(X,D)_0$ is surjective.
It follows from \cite[Lemma~8.4, Theorem~8.5]{BKS} that
the map $\vartheta^c_{X|D} \colon \CH_0(X|D)_0 \to \pi^{\ab}_1(X,D)_0$
is an isomorphism for all $D \in \Div_C(X)$.
Taking the limit and noting that
$C(U)_0 \cong {\underset{D}\varprojlim} \ \CH_0(X|D)_0$
by ~\eqref{eqn:Chow-deg-0}, it follows that
$\vartheta_{U/X} \colon C(U)_0 \to \pi^{\ab}_1(U)_0$ is an isomorphism.
We now use the commutative diagram

\begin{equation}\label{eqn:Main*-1-2}
\xymatrix@C.8pc{
C(U)_0 \ar[r]^-{\vartheta_{U/X}} \ar[d] & \pi^{\ab}_1(U)_0 \ar[d] \\
\CH^{A}_0(X|D)_0 \ar[r]^-{\vartheta^A_{X|D}} & \pi^{A}_1(X,D)_0.}
\end{equation}
We have shown above that the top horizontal arrow is an isomorphism.
The right vertical arrow is surjective. The surjectivity of 
$\vartheta^A_{X|D}$ on degree zero subgroups follows.

To show injectivity, we consider the diagram
\begin{equation}\label{eqn:Main-0-1}
\xymatrix@C.8pc{
\CH^{A}_0(X|D) \ar[r]^-{\vartheta^A_{X|D}} \ar@{->>}[d] &
\pi^{A}_1(X,D) \ar@{->>}[d] \\
\CH_0(X|D) \ar[r]^-{\vartheta^c_{X|D}} & \pi^{\ab}_1(X,D).}
\end{equation}

The left vertical arrow is an isomorphism by \thmref{thm:MLC} 
and the bottom horizontal arrow is injective by 
\cite[Lemma~8.4, Theorem~8.5]{BKS}.
It follows that the top horizontal arrow is injective.
\end{proof}

\begin{cor}\label{cor:Main*-2}
The canonical map 
$\pi^A_1(X,D) \to \pi^{\ab}_1(X,D)$ is an isomorphism.
In particular, the canonical inclusion
$\Fil^c_D H^1(K) \inj \Fil^A_D H^1(K)$ is a bijection.
\end{cor}
\begin{proof}
We only need to show that $\pi^A_1(X,D)_0 \to \pi^{\ab}_1(X,D)_0$
is injective. But this follows immediately by restricting
the diagram ~\eqref{eqn:Cyc-A-1} to degree zero subgroups and 
noting that all arrows in the right square except the bottom 
horizontal arrow are isomorphisms by 
\cite[Lemma~8.4, Theorem~8.5]{BKS}, \thmref{thm:MLC} and
\lemref{lem:Main*-1}. 
\end{proof}

\vskip .3cm

{\bf{Proof of \thmref{thm:Main-5}.}}
By \propref{prop:Fil-D-prop} and Definition~\ref{defn:Curve-fund-grp},
the theorem is equivalent to the statement that
$\Fil_D H^1(K) = \Fil^c_D H^1(K)$ as subgroups of $H^1(K)$.

We first show that $\Fil_D H^1(K) \subset \Fil^c_D H^1(K)$.
In order to show this, we can replace $\Fil^c_D H^1(K)$ by $\Fil^A_D H^1(K)$ 
using \corref{cor:Main*-2}.
We now let $\chi \in \Fil_D H^1(K)$ and let $Y \subset X$ be an
integral curve not contained in $D$ and not meeting $A$.
Let $\nu \colon Y_n \to X$ be the induced map from the normalization of $Y$.
We need to show that $\nu^*(\chi) \in \Fil_{\nu^*(D)} H^1(k(Y))$.

Since $Y \cap C_\sing = Y \cap A = \emptyset$, we can replace $X$ (resp. $C$) by
$X \setminus A$ (resp. $C_\reg$) to show the above assertion.
Since $C_\reg$ is a smooth divisor inside the smooth variety
$X \setminus A$, \cite[Corollary~2.8]{Kerz-Saito-2} 
(which does not require $X \setminus A$ to be projective) applies.
This yields $\nu^*(\chi) \in \Fil_{\nu^*(D)} H^1(k(Y))$.

We now show the reverse inclusion.
Let $\chi \in \Fil^c_D H^1(K) \subset H^1(U)$. 
We let $X' = X \setminus A, \ D' = D \cap X'$ and $C' = C \cap X'$.
Note that $U = X' \setminus C'$.
By \cite[Theorem~7.19]{Gupta-Krishna-REC}, 
it suffices to show that $\chi \in \Fil_{D'} H^1(K)$.
To prove this, we first deduce from \propref{prop:Fil-c-rest} that $\chi \in
\Fil^c_{D'} H^1(K)$. Since $X'$ is regular and $C'$ is a regular
divisor on $X'$ with complement $U$, it follows from 
\cite[Corollary~2.8]{Kerz-Saito-2} that $\chi \in \Fil_{D'} H^1(K)$.
This finishes the proof.
\qed
%$\hfill \square$

\vskip .4cm

\begin{remk}\label{remk:Higher-rank}
With $(X,D, U)$ as in \thmref{thm:Main-5} and $r \ge 1$ an integer, let 
$\sR_r(U)$ be the set of isomorphism classes of lisse $\ov{\Q}_{\ell}$-Weil
sheaves on $U$ of rank $r$ up to semi-simplification. For a given $V \in 
\sR_r(U)$, one can define the Swan conductor ${\rm Sw}(V)$ as 
${\rm Sw}(V) = {\underset{E}\sum} {\rm Sw}_E(V) [E] \in \Div(X)$, where
$E$ runs through integral divisors on $X$ and ${\rm Sw}_E(V)$ is the
Swan conductor of $V$ at the generic point of $E$.
The latter was defined in \cite[Definition~3.1]{Esnault-Kerz} when $X$
is a curve. But ${\rm Sw}_E(V)$ makes sense in any dimension if we 
use the Abbes-Saito filtration $G^{(\bullet)}_{K_E}$ of $G(K_E)$, 
where $K_E$ is the Henselization of $K$ at the generic point of $E$
(see \cite{Abbes-Saito})
instead of the classical ramification filtration $I^{(\bullet)}_{K_E}$ of $G(K_E)$ for
curves, used in \cite{Esnault-Kerz}.
One says that $V \in \sR^{\divf}_r(X,D) \subset \sR_r(U)$ if
${\rm Sw}(V) \le D$. 

We can now ask if $\sR^{\divf}_r(X,D) = \sR_r(X,D)$, where the latter
group is as in \cite[Definition~3.6]{Esnault-Kerz}.
\thmref{thm:Main-5} can be viewed as an answer to the $r =1$ case of this
question. An attempt to answer this question may also lead one to a 
proof of \thmref{thm:Main-5} using only ramification theory.
\end{remk}

\vskip .3cm

{\bf{Proof of \thmref{thm:Main-0}.}}
We first assume $k$ to be finite. We need to show that $\cyc_{X|D} \colon \sZ_0(U) \to C_{KS}(X,D)$ 
induces an isomorphism
$\cyc_{X|D} \colon \CH_0(X|D) \xrightarrow{\cong}
C_{KS}(X,D)$.

The theorem is already known when $d \le 2$ by \cite{BKS}. 
We shall therefore assume that $d \ge 3$. We can also assume that $X$
is integral. Since the composite map
\begin{equation}\label{eqn:BF-00}
\sZ_0(U) \xrightarrow{\cyc_{X|D}} C(X,D) \xrightarrow{\psi_{X|D}} C_{KS}(X,D)
\end{equation}
is the cycle class map of ~\eqref{eqn:Cycle-map},
we can replace $C_{KS}(X,D)$ by $C(X,D)$ using
\cite[Theorem~3.8]{Gupta-Krishna-REC}.

Let $s_{X|D} \colon \sZ_0(U) \surj \CH_0(X|D)$ be the canonical quotient map.
We denote the quotient maps
$C(U) \surj \CH_0(X|D)$ and $\pi^{\ab}_1(U) \surj \pi^{\ab}_1(X,D)$
by $t_{X|D}$ and $q'_{X|D}$, respectively.
Let $\wh{\rho}_{U/X} \colon \sZ_0(U) \to C(U)$ be the
canonical map induced by taking the limit of the
maps $s_{X|D'}$ as $D'$ runs through $\Div_C(X)$.

We now consider the commutative diagram

\begin{equation}\label{eqn:Main-00}
\xymatrix@C.8pc{
\sZ_0(U) \ar[dr]_-{s_{X|D}} \ar[ddd]_-{\cyc_{U/X}} \ar[r]^-{\wh{\rho}_{U/X}} 
& C(U) \ar[r]^-{\vartheta_{U/X}} \ar[d]^-{t_{X|D}} &
\pi^{\ab}_1(U) \ar[d]^-{q'_{X|D}} \\
& \CH_0(X|D) \ar[r]^-{\vartheta^c_{X|D}} 
\ar@{.>}[d]^-{\cyc_{X|D}} & \pi^{\ab}_1(X,D) 
\ar[d]^-{r_{X|D}} \\
& C(X,D) \ar[r]^-{\rho_{X|D}} &
\pi^{\adiv}_1(X,D) \\
\wt{C}_{U/X} \ar[ur]^-{p_{X|D}} \ar[rr]^-{\wt{\rho}_{U/X}} & &
\pi^{\ab}_1(U) \ar[u]_-{q_{X|D}}.}
\end{equation}

The triangle on the top left commutes by the definitions its various
arrows. The square on the top right commutes by ~\eqref{eqn:Cyc-A-1}.
The bottom trapezium commutes by ~\eqref{eqn:Rec-final-0}.
The big outer square is commutative because
we have seen before that for any closed point $x \in U$,
both the maps $r_{X|D} \circ q'_{X|D} \circ \vartheta_{X|D} \circ
\wh{\rho}_{U/X}$ and $q_{X|D} \circ \wt{\rho}_{X|D} \circ \cyc_{U/X}$ coincide 
with the Frobenius substitution at $x$.
The map $\vartheta^c_{X|D}$ is injective  by \thmref{thm:MLC},
 \lemref{lem:Main*-1} and \corref{cor:Main*-2}.
The map $r_{X|D}$ is an isomorphism by \thmref{thm:Main-5}
and $\rho_{X|D}$ is injective by \thmref{thm:Rec-mod}.
By a diagram chase, it follows that the composite map $p_{X|D} \circ \cyc_{U/X}$
annihilates $\sR_0(X|D)$ and induces an honest map
\begin{equation}\label{eqn:Main-01}
\cyc_{X|D} \colon \CH_0(X|D) \to C(X,D)
\end{equation}
such that the middle square in ~\eqref{eqn:Main-00} commutes.
Moreover, this map is injective.
It is surjective by ~\eqref{eqn:ICG-maps} and
\cite[Theorem~2.5]{Kato-Saito-2}.
This proves the theorem for finite base field case.

When $k$ is the algebraic closure of a finite field, an easy descent argument
shows that there exists a finite field $k' \subset k$ and a geometrically integral
and smooth projective variety $X'$ over $k'$ together with an effective
Cartier divisor $D' \subset X'$ such that $X = X'_{k}$ and $D = D'_k$.
Since $cyc_{X|D}$ is clearly compatible with the pull-back via field extensions,
and since $\CH_0(X|D)$ and $C_{KS}(X,D)$ are continuous functors
(see \cite[Proposition~6.2]{Binda-Krishna} and \cite[Sublemma~7.3]{Kerz09} and recall that
$K$-theory is also a continuous functor), we conclude from the case of finite
base field.
\qed
%$\hfill \square$

\vskip .4cm

\subsection{Bloch's formula without regularity}
\label{sec:No-reg}
In this subsection, we shall prove \thmref{thm:Main-1*}, Bloch's formula
in the pro-setting when $X$ is not regular.
We remark that it is not clear if one should even expect Bloch's formula for each 
$D$ when $X$ is singular.
Our assumption now is the following.
We let $X$ be a normal projective integral
scheme of dimension $d \ge 1$ over 
a finite field $k$ and let $D$ be an effective 
Cartier divisor on $X$. 
We assume that 
$U = X \setminus D$ is regular.

We shall prove \thmref{thm:Main-1*} after the following lemma.
We shall ignore to write the indexing set $\N$ of the
pro-abelian groups used in the theorem in what follows.

\begin{lem}\label{lem:CCM-D}
The cycle class map $\cyc_{U/X} \colon \sZ_0(U) \to \wt{C}_{U/X}$ descends
to a continuous homomorphism of pro-abelian groups
\[
\cyc^{\bullet}_{X|D} \colon \{\CH_0(X|nD)\} \to \{C(X,nD)\}.
\]
\end{lem}
\begin{proof}
We need to show that
the map of pro-abelian groups $\{\cyc_{X|D}\} \colon \{\sR_0(X|nD)\} \to 
\{C(X,nD)\}$ is zero. 
Since $\sR_0(X|nD) \subset \sZ_0(U)_0$, it suffices to show using
~\eqref{eqn:CHM-4} that the map of pro-abelian groups 
$\{\cyc_{X|nD}\} \colon \{\sR_0(X|nD)\} \to 
\{C(X,nD)_0\}$ is zero.
To show this, we look at the sequence of maps
\begin{equation}\label{eqn:CCM-D-0}
\{\sR_0(X|nD)\} \xrightarrow{\{\cyc_{X|nD}\}} \{C(X,nD)_0\}
\xrightarrow{\{\rho_{X|nD}\}} \{\pi^{\ab}_1(X,nD)_0\},
\end{equation}
where the second map $\{\rho_{X|nD}\} = \rho^\bullet_{X|D}$ exists by
\thmref{thm:C-Rec}.

We have seen just above
~\eqref{eqn:CHM-6} that the map $\sZ_0(U) \to \pi^{\ab}_1(X,nD)$ factors
through $\vartheta^c_{X|nD} \colon \CH_0(X|nD) \to \pi^{\ab}_1(X,nD)$
for every $n$. It follows that the composite arrow in
~\eqref{eqn:CCM-D-0} is zero.
The lemma now follows from \thmref{thm:C-Rec-iso}.
\end{proof}

Using \lemref{lem:CCM-D}, we get commutative diagrams

\begin{equation}\label{eqn:CCM-D-2}
\xymatrix@C.8pc{
C(U) \ar[r]^-{\cyc_{U/X}} \ar[dr]_-{\vartheta_{U/X}} & 
\wt{C}_{U/X} \ar[d]^-{\wt{\rho}_{U/X}} & 
\{\CH_0(X|nD)\} \ar[r]^-{\cyc^{\bullet}_{X|D}} \ar[dr]_-{\vartheta^{\bullet}_{X|D}}
& \{C(X,nD)\} \ar[d]^-{\rho^{\bullet}_{X|D}}  \\
& \pi^{\ab}_1(U) & & \{\pi^{\ab}_1(X,nD)\},}
\end{equation}
where the triangle on the left is the
limit of the one on the right.

\vskip .3cm

{\bf{Proof of \thmref{thm:Main-1*}.}}
The composite map $\sZ_0(U) \to \{\CH_0(X|nD)\} \to \{C_{KS}(X,nD)\}$
is surjective by \cite[Theorem~2.5]{Kato-Saito-2}.
So we only need to show that $\cyc^{\bullet}_{X|D}$ is injective.
For this, we can replace $\{C_{KS}(X,nD)\}$ by $\{C(X,nD)\}$ 
using ~\eqref{eqn:ICG-maps}.

Using the commutative diagram of exact 
sequences of pro-abelian groups
\begin{equation}\label{eqn:Main-formula-0}
\xymatrix@C.8pc{
0 \ar[r] & \{\CH_0(X|nD)_0\} \ar[r] \ar[d]_-{\cyc^{\bullet}_{X|D}} &
\{\CH_0(X|nD)\} \ar[r]^-{{\deg}} \ar[d]_-{\cyc^{\bullet}_{X|D}} & 
\Z \ar@{=}[d]  \\
0 \ar[r] & \{C(X,nD)_0\} \ar[r] & \{C(X,nD)\} \ar[r]^-{{\deg}} &
\Z,}
\end{equation}
it suffices to show that the left vertical arrow is injective.

For this, we look at the sequence of maps (see ~\eqref{eqn:CCM-D-2})
\begin{equation}\label{eqn:Main-formula-1}
\{\CH_0(X|nD)_0\} \xrightarrow{\cyc^{\bullet}_{X|D}} \{C(X,nD)_0\}
\xrightarrow{\rho^{\bullet}_{X|D}} \{\pi^{\ab}_1(X,nD)_0\}.
\end{equation}
We have seen in the proof of \thmref{thm:Main-0} that 
the composite map is an (in fact level wise) isomorphism
(see \cite[Lemma~8.4, Theorem~8.5]{BKS}).
The map $\rho^{\bullet}_{X|D}$ is an isomorphism by \corref{cor:C-Rec-iso-1}.
We conclude that $\cyc^{\bullet}_{X|D}$ is an isomorphism.
This finishes the proof.
\qed
%$\hfill\square$
\vskip .3cm

As a consequence of \thmref{thm:Main-1*}, we have the following. 
\begin{cor}
The inverse limit idele class group $\wt{C}_{U/X}$ is independent of the normal compactification $X$ of $U$. 
\end{cor}
\begin{proof}
Follows from \thmref{thm:Main-1*} and \cite[Lemma~3.1]{Kerz-Saito-2}.
\end{proof}

\vskip .4cm

\noindent\emph{Acknowledgements.}
Gupta was supported by the
SFB 1085 \emph{Higher Invariants} (Universit\"at Regensburg).
He would also like to thank TIFR, Mumbai for
invitation in March 2020 and extending the invitation during 
the tough times of the Covid-19 pandemic. 
The authors would like to thank Shuji Saito for telling them
about his positive expectation of \thmref{thm:Main-5} when they were 
working on its proof, and to Moritz Kerz for some fruitful 
discussion with Gupta on the contents of this manuscript.
The authors would also like to thank the referee for reading the manuscript
very thoroughly and providing many helpful comments.

\end{document}